\documentclass[10pt, twoside]{amsart}

\usepackage[T1]{fontenc}
\usepackage[utf8x]{inputenc}
\usepackage[english]{babel}
\usepackage{amsmath}
\usepackage{amsthm}
\usepackage{tensor}
\usepackage{amssymb}
\usepackage{bbm}   
\usepackage{fancyhdr}               
\usepackage{fixltx2e}               
\usepackage{enumitem}               
\usepackage[all,cmtip]{xy}          
\usepackage{graphicx}
\usepackage{esint} 
\usepackage{mathrsfs} 
\usepackage{faktor}
\usepackage{mathtools}

\newtheorem{theorem}{Theorem}            
\newtheorem{corollary}[theorem]{Corollary}
\newtheorem{lemma}[theorem]{Lemma}
\newtheorem{prop}[theorem]{Proposition}

\theoremstyle{definition}              
\newtheorem{definition}{Definition}

\theoremstyle{remark}                  

\newtheorem{remark}{Remark}

\newcommand{\R}{\mathbb{R}}                 

\newcommand{\weak}{\rightharpoonup}              

\renewcommand{\d}{\mathrm{d}}               

\newcommand{\vf}{\varphi}
\newcommand{\eps}{\varepsilon}
\newcommand{\cy}{\textup{cyl}}
\newcommand{\argmin}{\textup{argmin}}
\renewcommand{\exp}{\textup{exp}}

\renewcommand{\weak}{\rightharpoonup}

\DeclareMathOperator{\diam}{diam}                                   

\DeclareMathOperator{\curl}{\curl}                                   

\renewcommand{\div}{\mathrm{div}\,}                                 
\newcommand{\Vg}{\d\mathrm{vol}_g}
\newcommand{\Vge}{\d\mathrm{vol}_{g_\eps}}
\newcommand{\Vbge}{\d\mathrm{vol}_{\bar{g}_\eps}}
\newcommand{\dpu}{\coloneqq}

\newcommand{\abs}[1]{\left| #1 \right|}                             
\newcommand{\norm}[1]{\left\| #1 \right\|}                          

\usepackage{color}
\usepackage[normalem]{ulem} 

\numberwithin{equation}{section}
\numberwithin{definition}{section}
\numberwithin{theorem}{section}
\numberwithin{remark}{section}

\begin{document}

\title[]
{Heat flow of harmonic maps into CAT($0$)-spaces}

\author{Fang-Hua Lin}
\address{Courant Institute of Mathematical Sciences, New York University, NY 10012, USA}
\email{linf@math.nyu.edu}
\author{Antonio Segatti}
\address{Dipartimento di Matematica ``Felice Casorati'',
Universit\`a di Pavia, via Ferrata 1, 27100 Pavia, Italy}
\email{antonio.segatti@unipv.it}
\author{Yannick Sire}
\address{Johns Hopkins University, Krieger Hall, 3400 N. Charles St., Baltimore, MD, 21218, USA}
\email{ysire1@jhu.edu}
\author{Changyou Wang}
\address{Department of Mathematics, Purdue University, West Lafayette, IN 47907, USA}
\email{wang2482@purdue.edu}

\date{}

\begin{abstract} We introduce a new approach to prove the global existence and uniqueness of suitable weak solutions of the heat flow of harmonic mappings into CAT(0) metric spaces. Our method allows also to prove Lipschitz continuity in spatial variables for such solutions into any CAT$(0)$-space, answering a long-standing open problem in the field. Our approach is based on an elliptic regularization of the gradient flow of the Dirichlet energy and even in the case of smooth Riemannian targets provides a novel viewpoint, together with a new Dynamical Variational Principle and a new proof of the celebrated Eells-Sampson theorem. The spatial Lipschitz regularity for such weak solutions
is achieved by fully exploiting  the variational structure of the problem at the regularized level and introducing a parabolic frequency function of Almgren-Poon type. Our contribution is the first instance of the use of monotonicity methods for parabolic deformations of maps into singular targets. 
\end{abstract}

\maketitle
\tableofcontents

\section{Introduction}
This paper is devoted to a new approach to establish the existence of global weak solutions to the heat flow of harmonic maps into a CAT($0$)-space, which is a metric space generalization of a smooth manifold with non-positive sectional curvature. The seminal paper by Eells and Sampson \cite{eells-sampson} constructs harmonic maps between smooth manifolds by running the (negative) gradient flow of the Dirichlet energy functional and proving that whenever the maps take values into a non-positively curved Riemannian manifold, the flow is smooth, globally defined and converges (up to a subsequence) when $t\to +\infty$  to a minimizing harmonic map that is homotopic to the initial map.
It is almost impossible to establish the gradient flow for mappings into non-smooth target spaces directly. 
In the present paper, we introduce an alternative approach through a variational approximation scheme
that does not rely on the smoothness of target spaces. The approach can be viewed as an elliptic regularization of the gradient flow, 
which is based on the Weighted-Energy-Dissipation (WED) scheme developed by the second author with Rossi, Savar\'e and Stefanelli in \cite{rsss-wed}.

Before introducing our strategy, we introduce some notations. 
Let $(M,g)$ be a $n$-dimensional complete Riemannian manifold without boundary and $(X,d)$ be a {Polish (i.e. complete and separable)} metric space.
In the following sections, we concentrate on two situations for the space $X$: either it is a Nonpositively Curved space (NPC) or a smooth Riemannian manifold $N$ with nonpositive sectional  curvature. The theory of harmonic mappings into singular spaces was initiated by Gromov and Schoen in \cite{GS} and extended subsequently by Korevaar and Schoen \cite{KS,KS2} (see as well a different approach by Jost \cite{Jost1,Jost2}) to provide a flexible tool to prove geometric rigidity in various settings. In both cases (of smooth or singular targets), one can naturally define a Dirichlet energy associated to Lipchitz functions as follows (see Appendix \ref{appendixSpaces} for more details): 
\begin{equation}
\label{eq:diri_intro}
E(u):=
\begin{cases}
 \frac{1}{2}\int_M \abs{\nabla u}^2 \Vg, &\qquad u\in H^1(M, X)\\
  + \infty &\qquad \textrm{ otherwise in } L^2(M, X)
  \end{cases}
\end{equation}
with domain $D(E) = H^1(M, X)$, where the space $H^1(M,X)$ is described in the Appendix A (following \cite{KS,KS2}). 

For $\eps>0$, define the WED energy functional by
\begin{equation}
\label{eq:wed_intro_0}
\mathcal{I}_\eps(v)=\frac{1}{2}\iint_{M\times\R_+}\frac{e^{-t/\eps}}{\eps}\left(\eps\abs{\partial_t v}^2+\abs{\nabla v}^2\right)\Vg\d t,
\end{equation}
for all $v\in H^1_{\rm{loc}}(M\times\R_+, X)$ with finite $\iint_{M\times\R_+}(\abs{\partial_t v}^2 +\abs{\nabla v}^2)e^{-t/\eps}\Vg \d t$.

For any $u_{0}\in L^2_{{\rm{loc}}}(M, X)$ with finite Dirichlet energy $E(u)$, it is easy to see that
\begin{equation}
\label{eq:Hut0}
{\mathfrak{V}_{{u_{0}}}:= \left\{v\in H^1_{\textrm{loc}}(M\times\R_+; X): \,v(0)= {u_0},  
\,\,\iint_{M\times\R_+}(\abs{\partial_t v}^2 +\abs{\nabla v}^2)e^{-t/\eps}\Vg \d t<+\infty\right\}}\not=\emptyset.
\end{equation}
Hence, by the direct method, there exists at least one minimizer $u_\eps\in {\mathfrak{V}}_{{u_0}}$ of $\mathcal{I}_\eps(\cdot)$, i.e.,
\[
\mathcal{I}_\eps(u_\eps)=\min\Big\{\mathcal{I}_\eps(v): \ v\in {\mathfrak{V}}_{{u_0}}\Big\}.
\]
To have an intuition on the type of approximation we are considering, let us momentarily assume that $X$ is a smooth Riemannian manifold $(N,h)$ isometrically embedded in $\R^L$. 
It is not difficult to show that in this situation any minimizer $u_\eps$ is actually a solution of the Euler-Lagrange equation 
\begin{equation}
\label{eq:elliptic_intro}
\begin{cases}
-\eps\partial_t^2 u_\eps +\partial_t u_\eps -\Delta u_\eps \perp T_{u_\eps} N, &\qquad \textrm{ in }M\times (0,+\infty),\\
u_\eps(x,0) = {u}_0(x), & \qquad \textrm{ on } M.
\end{cases}
\end{equation}
Formally,  when $\eps\to 0$ one expects that $u_\eps$ converges to a solution $u$ 
of the heat flow of harmonic maps into $N$
\begin{equation}
\label{eq:heat_flow_intro1}
\begin{cases}
 \partial_t u -\Delta u \perp T_{u}N, &\qquad \textrm{ in }M\times (0,\infty),\\
 u(x,0) = {u}_0(x), &\qquad \textrm{ in }M.
\end{cases}
\end{equation}
Thus $u_\eps$ can be viewed as
an {\it elliptic regularization} of the heat flow of harmonic maps. Also, since 
$u_\eps$ is constructed {\sl variationally} as a minimizer of the energy $\mathcal{I}_\eps$,
we will often refer to this approximation as  a {\sl variational approximation}. 
We prove in this paper that the previous elliptic regularization indeed provides a solution of the heat flow  \eqref{eq:heat_flow_intro1} when the manifold $N$ is negatively curved. It provides  a new approach to investigate  the heat flow of harmonic maps between smooth manifolds. More importantly, this variational approximation happens to be very well suited even when the target $X$ is a metric space, with negative curvature in the sense of Alexandrov. 
\vspace{0.3cm}

The main purpose of the paper is to address the long-standing open problem of constructing suitable weak solutions to a parabolic deformation of harmonic maps into singular spaces, which enjoy (sharp) Lipschitz regularity in the space variable.  
Therefore from now on we will consider a class of singular spaces $X$ that we describe below. 
\begin{definition} The space $X$ is a complete and separable non-positively curved (NPC) metric space (also known as CAT($0$)-metric space). 
More precisely, the following two conditions hold (see also \cite{GS,KS,KS2}): 
\begin{itemize}
\item[(i)] $(X,d)$ is a complete length space: For any $P, Q\in X$, $d(P,Q)$ equals to the length of a rectifiable curve $\gamma:[0,1]\to X$
such that $\gamma(0)=P$ and $\gamma(1)=Q$. Such a $\gamma$ is called a minimal geodesic between $P$ and $Q$.

\item[(ii)] For any three points $P, Q, R\in X$, let $\gamma_{Q,R}$ be a minimal geodesic between $Q$ and $R$. For $0<\lambda<1$,
let $Q_\lambda\in\Gamma_{Q,R}$ be such that $d(Q_\lambda, Q)=\lambda d(Q,R)$ and $d(Q_\lambda, R)=(1-\lambda)d(Q,R)$, then the
following comparison inequality holds:
\begin{equation}\label{NPC1}
d^2(P, Q_\lambda)\le (1-\lambda)d^2(P,Q)+\lambda d^2(Q, R)-\lambda(1-\lambda)d^2(Q,R).
\end{equation}
\end{itemize}
\end{definition}
\vspace{0.5cm}

It follows from (i) and (ii) that for any pair of points in $P, Q\in X$ there is a unique, minimal geodesic $\gamma:[0,1]\to X$
joining $P$ and $Q$. Moreover, if, for a fixed point $Q\in X$ and a $\lambda\in [0,1]$,  
we define the contraction map $R_{\lambda, Q}: X\to X$ by letting
\[
R_{\lambda, Q}(P)=P(\lambda), \ \forall P\in X,
\]
where $P(t), t\in [0,1]$, is the constant speed geodesic joining $Q$ to $P$ parametrized on $[0,1]$. Then $R_{\lambda, Q}$ is Lipschitz, with Lipschitz norm at most $\lambda$, i.e.
\[
d(R_{\lambda, Q}(P_1), R_{\lambda, Q}(P_2))\le \lambda d(P_1, P_2), \ \forall P_1, P_2\in X.
\]
We recall that the CAT($0$) property of $X$ is inherited by $L^2(M, X)$ when endowed with the distance
\begin{equation}
\label{eq:metric_L2}
d_2(u, v):=\Big(\int_M d^2(u(x),v(x))\Vg\Big)^\frac12, \ \forall u, v\in L^2(M, X).
\end{equation}
We also recall that the CAT($0$) assumption guarantees that both the energy $E$ and $d^2_2$, and hence $\mathcal{I}_\eps$, satisfy suitable convexity properties along geodesics. 

We will tacitly identify curves $f:[0,\infty)\to L^2(M, X)$ with their concrete representation $f:M\times [0,\infty)\to X$ (with the same notation, for simplicity). 
In particular once we define the metric derivative $\abs{f'}(t)$ for an absolutely continuous curve with values in $L^2(M, X)$ (see Appendix \ref{appendixSpaces}) we will write
\[
\abs{f'}^2(t) = \int_M \abs{\partial_t f(x,t)}^2 \Vg, \qquad \textrm{ a.e. in }\,\,(0,\infty). 
\]
We also advise the reader that we will prefer the metric derivative notation in Sections \ref{sec:wed_hm} and \ref{sec:conve_general} while we will use the concrete realization when dealing with the regularity results (see Sections \ref{sec:mono} and \ref{sec:Lipregularity}). 
For any (separable) CAT($0$) space $(X,d)$, we will show that if $u_\eps\in\mathfrak{V}_u$ is a minimizer of $\mathcal{I}_\eps(\cdot)$,
then the limit of $u_\eps$, as $\eps\to 0$, will provide a suitable weak solution of the heat flow of harmonic maps
from $(M,g)$ to $(X,d)$. 

Here is the notion of weak solution we consider (see Appendix \ref{appendixSpaces} for the definition of absolutely continuous curve).

\begin{definition}\label{hfhm_def} For ${u}_0\in H^1(M, X)$, a curve $u:[0,+\infty)\to L^2(M,X)$ with 
$u\in AC^2([0,+\infty);L^2(M, X))\cap L^\infty(0,+\infty;H^1(M,X))$ 
%
%
%
is called a suitable weak solution to the heat flow of harmonic maps
into $(X,d)$ with initial value
${u}_0$, if $u\big|_{t=0}={u}_0$ 
and if $u$ is a solution of the Evolution Variational Inequality (EVI), 
\begin{equation}
\label{eq:EVI_intro}
\frac{1}{2}\frac{\d}{\d t}d_2^2(u(t), v) + E(u(t))\le E(v)\qquad \textrm{ in } \mathscr{D}'(0,+\infty), \,\,\,\textrm{for any }v\in D(E).
\end{equation}

\end{definition} 

This concept of solution enables us to encode, even in this non smooth framework, the concept of ``$L^2$-gradient flow'' 
of the Dirichlet energy. Moreover the recent results of Gigli and Nobili (see \cite{gigli_nobili21}) show that the gradient flow
solution $u$ in Definition \ref{hfhm_def} actually verifies, similarly to what happens for gradient flows in Hilbert spaces, a suitable differential inclusion relating the subdifferential of the energy $E$ and the time derivative (see \cite{gigli_nobili21} for the Definitions of subdifferential and of time derivative in this metric context). In particular 
when the target is a smooth Riemannian manifold with non positive sectional curvature a curve that satisfies Definition \ref{hfhm_def} is the (smooth) solution of the heat flow of harmonic maps \eqref{eq:heat_flow_intro1}.

One of the main features of the convergence result (already explored and used in \cite{rsss-wed}) is the value function $V_\eps$, a concept borrowed from the {\sl Optimal Control Theory}. The value function is the functional 
\begin{equation}
\label{eq:value_intro}
V_\eps (u) = \min_{v\in \mathfrak{V}_{u}} \mathcal I _\eps [v], \ u\in H^1(M, X).
\end{equation}
As a consequence of the NPC assumption on $(X,d)$, we have  that $V_\eps$ is a geodesically convex functional in $L^2(M, X)$. Moreover, the {\sl Dynamic Programming Principle} (as in \cite{rsss-wed}) gives that any minimizer $u_\eps$ of $\mathcal{I}_\eps$ is indeed a gradient flow for the functional $V_\eps$. In particular, in the convex regime the Value Function satisfies an {\sl Hamilton-Jacobi Equation} formally similar to the one satisfied by the value function in the smooth finite dimensional setting (see \cite{bardi-dolcetta}). 
Finally, the underlying geodesic convexity (both in the Dirichlet energy and in distance square of $(L^2(M, X),d_2)$ entails that the notion of gradient flow is indeed the {\sl strongest possible}, namely the Evolution Variational Inequality (EVI) formulation (see \cite{ags2005}).  
Once one proves that there is enough compactness to pass to the limit $\eps\to 0$, the identification of the limit $u$ as a suitable weak solution then follows by the properties of the Value function.

Now we are ready to state our first main theorem. 

\begin{theorem}[Existence of a Gradient Flow and H\"olderianity of solutions]
\label{th:main_intro_0}
Let $(X,d)$ be a CAT(0)-space and ${u}_0\in H^1(M,X)$. Then
\begin{enumerate}
\item  For any $\eps>0$, there exists a unique minimizer $u_\eps\in \mathfrak{V}_{\bar{u}}$ of $\mathcal{I}_\eps(\cdot)$, which is the unique solution with $u_\eps(0) = {u}_0$ of the Evolution Variational Inequality 
\begin{equation}
\label{eq:EVI_V_intro}
\frac{1}{2}\frac{\d}{\d t}d_2^2(u_\eps(t), v) + V_\eps(u_\eps(t))\le V_\eps(v)\qquad \textrm{ in } \mathscr{D}'(0,+\infty), \,\,\,\textrm{for any }v\in D(E).
\end{equation}
\item  There exists a unique suitable weak solution of the heat flow of harmonic maps 
$u: M\times\R_+\to X$, with initial value ${u}_0$,  such that 
$$u_\eps\xrightarrow{\eps\to 0}u \ \ {\rm{in}}\ \ L^2_{\rm{loc}}(M\times\R_+).
$$ 
\item Assume additionally that the space $X$ is locally compact. Then $u\in C^{\alpha}(M\times (0,\infty), X)$ for some $\alpha\in (0,1)$.
\end{enumerate} 
\end{theorem} 
{
It is worth recalling that the existence of weak solutions to the harmonic map heat flow with values in a CAT$(0)$ metric space was established in the late 1990s by U. Mayer \cite{mayer} and, independently, by J. Jost \cite{Jost98}. In this setting, weak solutions are obtained as limits of De Giorgi’s {\sl Minimizing Movements Scheme}. These works are particularly interesting in that they show how the classical Crandall--Liggett generation theorem can be extended to the metric framework of CAT$(0)$ spaces, and they make clear that the interplay between the geometry of the target space—here encoded by the convexity assumption \eqref{NPC1}—and the convexity of the energy is a key ingredient in the existence theory for gradient flows.

Later, in the framework of their general theory of gradient flows in metric spaces, L. Ambrosio, N. Gigli, and G. Savaré shed light on the connection between the weak solutions constructed by Mayer and Jost and the notion of suitable weak solution considered in the present paper, namely the E.V.I. formulation. More precisely, the limit of De Giorgi’s {\sl Minimizing Movements Scheme} coincides with the unique suitable weak solution emanating from $u_0$; see \cite[Chapter 4]{ags2005}.

Against this background, the novelty of Theorem \ref{th:main_intro_0} is twofold. 
 On one hand we construct the unique suitable solution via a different approximation scheme {of elliptic variational nature} which is a novelty in the harmonic maps framework. This approximation scheme reveals to be { simple and} flexible enough to address the metric space setting. {In particular, the class of admissible domain manifolds can include Riemannian polyhedral manifolds and sub Riemannian manifolds such as Carnot groups, while the target spaces may be taken to be geodesic convex subsets of CAT($\kappa$)-metric space, for positive $\kappa$.} 
 On the other hand, this flexibility allows to prove the $C^\alpha$ regularity in space and time for the suitable weak solution.}
Motivated by the regularity theory of Eells-Sampson \cite{eells-sampson} valid for smooth targets with non positive sectional curvature, one expects that $u$ might reach the higher regularity possible with non smooth targets, namely is Lipschitz continuous in 
$x$-variable. We will confirm this { for all CAT(0)-spaces $(X,d)$ at least when $(M,g)=(\R^n, dx^2)$}.
First, we observe that since $u_\eps$ minimizes the WED functional $\mathcal{I}_\eps$, we can use the contraction
map $R_{\lambda, Q}:X\to X$  to perform outer variations of $u_\eps$ on $X$ to achieve a perturbed
{\it sub-harmonicity} of $d^2(u_\eps, Q)$ for the operator $\mathcal{L}_\eps=\eps \partial^2_t+\Delta-\partial_t$; 
while performing the domain (or inner) variations of $u_\eps$ on $(M,g)$ yields a perturbed version of stationarity identity. 
Based on these observations, we introduce a parabolic version of Almgren's frequency function \cite{Almgren}, analogous to that 
introduced by Poon \cite{Poon} for studying the unique continuation property of parabolic equations, and establish its monotonicity property. 
\vspace{0.2cm} 

\begin{remark}
As it will be clear in subsequent sections, Almgren-Poon monotonicity formula for the frequency function relies on the NPC condition of $(X,d)$.
On the other hand, we can derive by means of domain variations a useful Struwe monotonicity formula (i.e. monotonicity of renormalized energy of parabolic type), 
for {\sl any} target metric space $(X,d)$ (in particular CAT$(k)$ spaces for $k>0$ or smooth manifolds with positive curvature). This may lead to an alternative, WED-approach
to establish both existence and partial regularity of {\it stationary} weak solutions of heat flow of harmonic maps { into smooth Riemannian manifolds}.  
We plan to address this problem in a future work.
\end{remark}

Because of technical reasons, we will only focus on $(M, g)=(\R^n, dx^2)$,  the Euclidean space
$\R^n$,  for the next Theorem. From Theorem \ref{th:main_intro_0},
there exists a nonnegative Radon measure $\nu=\nu_t \d t$ on $\R^{n+1}_+$ such that
\[
e_\eps(u_\eps)\d x\d t\rightharpoonup \mu=\mu_t\d t\equiv (|\nabla u|^2\d x+ \nu_t)\d t
\]
as convergence of Radon measures on $\R^{n+1}_+$, where $e_\eps(u_\eps)=\eps |\partial_t u_\eps|^2+|\nabla u_\eps|^2$.

For $z_0=(x_0, t_0)\in\R^n\times (0,\infty)$, let $G_{z_0}$ be the backward heat kernel on $\R^n$ with center $z_0$, given by
\[
G_{z_0}(x,t)=(4\pi |t_0-t|)^{-\frac{n}2} \exp\big(-\frac{|x-x_0|^2}{4(t_0-t)}\big), \ \ \ x\in\R^n, \ t<t_0.
\]
For any fixed $Q\in X$, as in \cite{Poon} we define
\[
E(u; z_0, R)=2R^2\int_{\R^n\times\{t_0-R^2\}} G_{z_0}(x,t)\,d\mu_t(x),
\]
\[
H(u; z_0, R, Q)=\int_{\R^n\times\{t_0-R^2\}} d^2(u(x,t),Q)G_{z_0}(x,t)\,dx
\]
and the frequency function
\[
N(u; z_0, R,Q)=\frac{E(u; z_0, R)}{H(u; z_0, R, Q)}
\]
for $0<R<\sqrt{t_0}$, provided the denominator is not zero.

In section 3, we show that $N(u; z_0, R, Q)$ is monotonically increasing for $0<R<\sqrt{t_0}$ so that
\[
N(u; z_0, Q)=\lim_{R\to 0^+} N(u; z_0, R, Q)
\]
exists and is upper semi-continuous in $\R^{n+1}_+$. It turns out that if we choose $Q=u(z_0)$, 
then the quantity $N(u,z_0, u(z_0))$ is closely related to the (spatial Lipschitz) regularity of $u$ near $z_0$.
We will show in section 5.2 below that 
\begin{equation}\label{frequency-low-bound}
N(u; z_0, u(z_0))\ge 1, \ \forall z_0\in \R^n\times (0,\infty). 
\end{equation}
This leads to the second main theorem { addressing Lipschitz regularity of suitable harmonic heat flows}. 
\begin{theorem}[Spatial Lipschitz regularity of solutions] \label{Lip_reg} 
Let $(X,d)$ be any CAT$(0)$ metric space. If
$u_0:\R^n\to X$ satisfies $E(u_0)<\infty$ and  $d(u_0, Q)\in L^\infty(\R^n)$ for some $Q\in X$,
then the solution $u$ of the heat flow of harmonic maps obtained by 
Theorem \ref{th:main_intro_0} is Lipschitz continuous in $x$ and H\"older continuous with exponent 
$\frac12$ in $t$ on $\R^n\times (0,\infty)$.
\end{theorem} 
{
To the best knowledge of the authors, the introduction of parabolic Almgren-Poon type frequency functions into the study of harmonic heat flows is new even for smooth NPC target manifolds.
This technique has potential applications to other geometric flow problems. Let us mention that it can be used to study the rank zero and rank one sets for the harmonic heat flows into CAT(0)-spaces, and another application is related  to the coupling dynamics of free sharp interfaces and bulk maps in the vectorial phase transition problems.}{ We plan to address these topics in the near future. }

{We would like to point that when the CAT(0) space $(X,d)\hookrightarrow\R^L$ is an Euclidean building that was studied by Gromov-Schoen \cite{GS}, \eqref{frequency-low-bound}
can be obtained by utilizing the almost everywhere approximate differentiability property for $H^1$-maps into $\R^L$.  On the other hand, for an arbitrary CAT(0)-space $(X,d)$,
we establish \eqref{frequency-low-bound} by
adapting the following two properties (see Theorem \ref{lowerboundfreq} below):
\begin{enumerate}
\item [i)] Kirchheim's metric differentiability theorem (\cite{Kirchheim1994}) asserting that any Lipschitz map $f:\R^L\to (X,d)$ is metric differentiable for a.e. $x_0\in \R^L$,
i.e. there exists a semi-norm ${\rm{md}}_{x_0}(f):\R^n\to \R_+$ such that
\begin{equation}\label{md0}
d(f(x), f(x_0))={\rm{md}}_{x_0}(f)(x-x_0)+o(|x-x_0|),
\end{equation}
\item [ii)] Gigli-Tyulenev's approximate metric differentiability theorem (\cite{GigliTyulenev}) asserting that given any $H^1$-map $f:\R^L\to (X,d)$, it holds that
for a.e. $x_0\in\R^L$, 
\begin{equation}\label{app-diff20}
\lim_{r\to 0} \frac{\Big|\big\{y\in B_r(x_0):  |d(f(x), f(x_0))-{\rm{md}}_{x_0}(f)(x-x_0)|>\varepsilon |x-x_0|\big\}\Big|}{\omega_n r^n}=0, \ \forall \varepsilon>0.
\end{equation}
\end{enumerate}
}

{We now comment on important examples of CAT$(0)$ spaces. In the smooth case, classical examples of CAT(0) spaces are Riemannian symmetric spaces of non-compact type. In a singular setting, the following structures are prominent examples (see the classical book by Bridson and Haefliger \cite{bridson-haefliger} for more details on the geometric models):

\begin{itemize}
\item Euclidean buildings as described in  Gromov-Schoen \cite{GS} which are locally finite Riemannian simplicial complexes. Many important results on geometric rigidity have been considering such spaces (see also the works of Corlette, Daskalopoulos-Mese e.g. \cite{GS,corlette,DM1,DM2,DM3} and references therein).
\item Homogeneous trees (also known as $\mathbb R$-trees). Those are discrete analogues of hyperbolic spaces (they are $0-$hyperbolic spaces in the sense of Gromov \cite{Gromov1987}).  An homogeneous tree of degree $Q$ is an infinite connected graph with no loops, in which every vertex is adjoint to $Q$ other vertices (see Figure~\ref{fig:Tree}).
\end{itemize} 
 
\begin{figure}[h!]
 \centering
    \includegraphics[height=40mm]{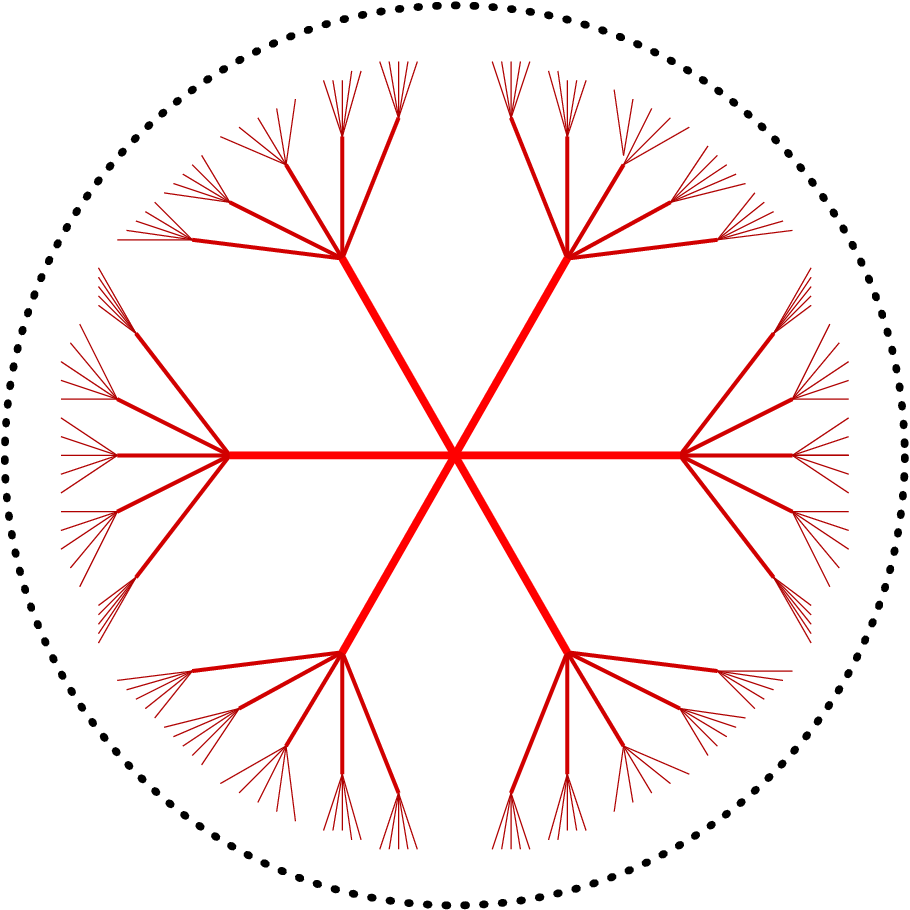}
    \caption{The homogeneous tree with $Q=6$}
 \label{fig:Tree}
\end{figure}

Harmonic maps into real trees have been considered by Sun \cite{Sun-Xiao}. As mentioned in \cite{Sun-Xiao}, $\mathbb R$-trees are limits in the Gromov-Hausdorff sense of NPC spaces whose curvature goes to negative infinity. Real trees are also instrumental in gauge theory and the construction of Higgs bundles after Corlette \cite{corlette-bundles} and the works of Daskalopoulos, Dostoglou and Wentworth \cite{DDW}. 
Interestingly, we would like also to point out that harmonic maps into $\mathbb R$-trees, which we can also describe as in Caffarelli and the first author \cite{caffa-lin}, as an Euclidean splitting  $\mathbb R^L= \bigoplus_k \mathbb R^{L_k}$, where the subspaces $\mathbb R^{L_k}$ are orthonormal subspaces,  appear naturally in shape optimization, spectral geometry (optimal partitions) and segregation problems (see e.g. the works of Caffarelli and the first author, Conti-Terracini-Verzini \cite{caffa-lin,caffa-lin-2,caffa-lin-3,conti-terra-verzini,conti-terra-verzini2} and several subsequent papers). In those PDEs, the components of the map have disjoint supports, e.g. 
$$
X=\left \{ x=(x_1,...,x_L ) \in \mathbb R^L,\,\, x_i\,x_j=0,\,\,i \neq j,\,\, x_i\geq 0 \right \}. 
$$}

For any CAT(0) space $(X,d)$, we also believe that the minimizers $u_\eps$ in Theorem \ref{th:main_intro_0} enjoy the spatial Lipschitz regularity in $\R^n\times (0, \infty)$, uniformly in $\eps$.  We confirm it when $(X,d)$ is induced by a 
smooth closed Riemannian manifold $(N, h)$ with non-positive sectional curvature, that is, 
$d$ is the geodesic distance function $d_N$ on $N$ induced by the Riemannian metric $h$. Without loss of generality, we assume
that $(N,h)$ is isometrically embedded in $\R^L$. We prove the following result, which provides yet an alternative proof of Eells-Sampson theorem. 
\begin{theorem}
\label{th:main_intro}
Let $(N,h)\hookrightarrow \R^L$ be a compact Riemannian manifold without boundary
 and with nonpositive sectional curvatures. 
For $u_0\in H^1(M,N)$ and $\eps>0$, let $u_\eps:M\times [0,\infty)\to N$ be the unique minimizer of WED functional
$\mathcal{I}_\eps$ over $\mathfrak{V}_{u_0}:=\Big\{v: M\times [0,\infty) \to N\ | \ \mathcal{I}_\eps(v)<+\infty, \ v\big|_{t=0}=u_0\Big\}$.
Then for any $t_0>0$, there holds
\begin{equation}\label{gradient_est} 
e_\eps(u_\eps)(x,t):=\big(\eps|\partial_t u_\eps|^2+|\nabla u_\eps|^2\big)(x,t)\le C(E(u_0), n, t_0), \ \forall (x,t)\in M\times [t_0,\infty).
\end{equation}
In particular, $u_\eps\xrightarrow{\eps\to 0}u$ in $L^2_{\rm{loc}}\cap C^\frac12_{\rm{loc}}(M\times \R_+, N)$, and the map $u\in C^\infty(M\times (0,\infty), N)$
is the unique solution of the heat flow of harmonic maps to $(N,h):$
\begin{equation}
\label{eq:heat_flow_intro}
\begin{cases}
 \partial_t u -\Delta u \perp T_{u}N, &\qquad \textrm{ in }M\times (0,\infty),\\
 u(x,0) = {u}_0(x), &\qquad \textrm{ in }M.
\end{cases}
\end{equation}
\end{theorem}

As already observed, the minimizer $u_\eps$ solves the elliptic (in space-time) system \eqref{eq:elliptic_intro}.
The ideas to obtain a uniform gradient estimate \eqref{gradient_est} are based on an $\eps$-version of Bochner inequality:
\begin{equation}
\label{eq:bochner0}
-\eps\partial_t^2 e_\eps(u_\eps) -\Delta e_\eps(u_\eps) + \partial_t e_\eps(u_\eps)\le C_M e_\eps(u_\eps) \ \ {\rm{on}}\ \ M\times (0,\infty),
\end{equation}
and the modified version of Moser's Harnack inequality: $\exists r_0=r_0(M)>0$ such that for $z_0=(x_0, t_0)\in M\times (0,\infty)$,
$0<r<\min\{r_0, \frac{\sqrt{t_0}}2\}$, and $\forall 0<\eps\le r^2$, it holds
\begin{equation}
\sup_{B_{r}(x_0)\times (t_0-r^2, t_0)}e_{\eps}(u_\eps)\le c(n)\left[\frac{\eps}{r^{n+2}}+ \frac{1}{r^{n}}\right]E(u_0).
\end{equation}
{We stress that Theorem \ref{th:main_intro} (and likewise Theorem \ref{th:main_intro_0}) does not follow directly from the abstract WED theory for gradient flows in metric spaces developed in \cite{rsss-wed} (see in particular Theorem 3.6 therein), since at that level of generality the regularity of weak solutions is not addressed.

Rather, the abstract framework produces a limit curve which can then be identified as a curve of maximal slope for the energy $E$ and, in our setting, with the unique smooth (or Lipschitz) solution of the flow \eqref{eq:heat_flow_intro}.}

The heat flow \eqref{eq:heat_flow_intro} was first introduced by Eells and Sampson to continuously deform the map ${u}_0$ and obtain, in the limit $t\to +\infty$, a harmonic map in the homotopy class of ${u}_0$. This leads in particular to geometric descriptions of homotopy groups of manifolds. 
This problem is classical and dates back to the seminal paper of Eells and Sampson \cite{eells-sampson} and of Hartman \cite{hartman}. The basic geometric assumption for this program to be successful is again the non positivity of the sectional curvature of the target manifold $N$, which brings the global existence and the smoothness of the flow, as already noticed. Therefore, given the map $u_0$ the existence of a 
harmonic map, which turns out to be  unique,  in the homotopy class of $u_0$ is obtained by investigating the long time behavior of the flow.   

The WED approximation provides a different proof of this classical result directly at the level $\eps>0$. 
This new approach combines the properties of the Value Function, among others the fact that it satisfies (when the energy $E$ is convex) a
Hamilton-Jacobi equation, with the interpretation of $u_\eps$ as an EVI gradient flow for $V_\eps$. This point of view highlights the basic role of the convexity of the energy $E$ and of the square of geodesic distance in $N$ (both consequences of the nonpositive curvature assumption) for the validity of the Eells and Sampson's result. 
We prove the following 


\begin{theorem}
\label{th:ES-optimal_control_intro}
Let $N$ be a compact Riemannian manifold without boundary and with nonpositive sectional curvature 
and let ${u}_0\in H^1(M,N)\cap C^\infty(M, N)$.
 Then, given $u_\eps \in \argmin\left\{\mathcal I _\eps[v],\,\, v\in \mathfrak{V}_{{u}_0}\right\}$ there holds
 \[
 \forall \eps>0\quad u_\eps(t)\xrightarrow{t\to +\infty}u_\infty \qquad \textrm{ in } L^2_{\textrm{loc}}(M, N), 
 \]
 with $u_{\infty}\in C^\infty(M, N)$ the unique minimizing harmonic map in the homotopy class of ${u}_0$.
\end{theorem}



\begin{remark}
In the smooth setting, there are several ways to construct solutions for the heat flow of harmonic maps.  A classical strategy is to penalize the constraint in the equation \eqref{eq:heat_flow_intro}, and introduce the following smooth evolution 
\begin{equation}
\label{eq:GL_intro}
\begin{cases}
\partial_t u_\eps -\Delta u_\eps + F_\eps(u_\eps) = 0 \qquad &\textrm{ in } M\times (0,+\infty)\\
u_{\eps}(x,0) = {u}_0(x) \qquad &\textrm{ on } M,
\end{cases}
\end{equation}
where $u_\eps:M\times [0,+\infty)\to \R^L$ and $F_\eps:\R^L\to \R_+$ ``approximates'' the target manifold $N$ as $\eps \to 0$. System \eqref{eq:GL_intro} is the so called {\sl Ginzburg-Landau system}. The papers of Chen and Struwe \cite{chen}, \cite{chen-struwe}, proved the existence of partially regular solutions of the heat flow of harmonic maps using such approach. In \cite{LW1,LW2,LW3}, the first and last author introduced another method by considering limits of {\sl parabolic defect measures}, an object which describes the obstructions to regularity of the flow. In particular, they provide another proof of the Eells-Sampson theorem. 

\end{remark}

We emphasize once again that the approach we pursue in the present paper is new and shed some new light on the variational theory of harmonic maps. Building on \cite{rsss-wed,ags2005}, the current framework can be summarized as:

\begin{enumerate}
\item It provides a new regularization, independent of the afore mentioned Ginzburg-Landau approximation scheme.  
\item The approach is variational and therefore it paves the way to the use of calculus of variations techniques. 
\item The approach is very well-suited to targets with {\sl nonpositive} curvature of Alexandrov type. 
\end{enumerate}

{
We also observe that 
 the variational approximation via the functional $\mathcal{I}_\eps$ has also a geometric flavor (already observed, in a different context by Ilmannen in \cite{ilmanen94}). 
For any $\eps>0$, the energy $\mathcal{I}_\eps$ is the Dirichlet energy for maps $v:M\times (0,+\infty)\to N$ when we endow the cylindrical manifold $M\times (0,+\infty)$ with a suitable  conformal metric $\bar{g}_\eps$ to the rescaled cylindrical one. This interpretation gives that $u_\eps$, being a (minimizing) harmonic map, is indeed smooth as soon as $N$ has nonpositive curvature. With this respect the time variable $t$ appears as a new ``space variable'' and the initial condition $u_0$ plays the role of the Dirichlet boundary condition on the boundary $M\times \left\{0\right\}$. }



{\begin{remark}
We would like to also point out that the framework developed in the present paper is general enough to incorporate several others geometric objects introduced: half-harmonic maps \`a la Da Lio-Rivi\`ere \cite{DR,DR2}, harmonic maps with free boundary \cite{chen-lin,HSSW,StruweManMath1991}, and several variations of those just to name a few examples. We hope to address such systems in the future. 
\end{remark}
}

{
As mentioned, the motivation to address the present approximation comes from the paper \cite{rsss-wed} on the Weighted Energy Dissipation approximation for Gradient flows in metric spaces.
The WED approximation has however a longer history (we refer to the recent \cite{stefanelli2024report} for a detailed discussion and reference therein). 
The first appearance of the functional $\mathcal{I}_\eps$ to address a parabolic problem seems to be in the paper of Ilmanen on the Brakke's flow \cite{ilmanen94} (however, references to elliptic regularization of parabolic problems were already been considered, for instance, by J.L. Lions in \cite{lions65}). More recently, the WED approximation has been used to address different evolutions of parabolic type, ranging from rate-independent problems (\cite{mielke-ortiz}) to gradient flows in Hilbert spaces \cite{mielke-stefanelli}. We also mention the very recent paper of Audrito \cite{audritoJEQ23}.

A particular credit should also be given to the late Ennio De Giorgi who, moving from the functional $\mathcal{I}_\eps$ proposed (see \cite{deGiorgi}) a modification of $\mathcal{I}_\eps$ to variationally  approximate the nonlinear wave equation ($m$ integer $>1$)
\[
\partial_{t}^2 u-\Delta u + m \abs{u}^{2m-1} =0.
\]
The De Giorgi conjecture has been positively fully solved by Serra and Tilli \cite{serra-tilli}. See also the paper by Stefanelli \cite{stefanelli-wave} for an earlier positive result.
}

\section{Weighted Energy Dissipation functional for Harmonic maps}
\label{sec:wed_hm}

In this section, we introduce our variational approximation scheme for maps defined on the space time cylinder
$M_\infty=M\times(0,\infty)$ with values in a CAT(0) metric space $(X,d)$. 

As explained in the introduction, we consider the following functional defined on (half) trajectories 
\[
\mathcal{I}_\eps(v):= \int_{0}^\infty\frac{e^{-t/\eps}}{\eps}\left(\frac{\eps}{2}\abs{v'}^2(t) + E(v(t))\right)\d t, 
\]
where $\abs{v'}$ is the metric derivative of $v$ and $E$ is the Dirichlet energy \eqref{eq:diri_intro}. 
Given ${u}_0\in D(E)= H^1(M, X)$ and
$\eps>0$,  we minimize  the WED functional $I_\eps$ on the (non empty) set 
\begin{equation}
\label{eq:Hut}
{\mathfrak{V}_{{{u}_0}}:= \left\{v:M_\infty\to X: v\in H^1_{\textrm{loc}}(M_\infty, X): \,v(0)= {u}_0,  
\,\,\iint_{M_\infty}(\abs{\partial_t v}^2 +\abs{\nabla v}^2)e^{-t/\eps}\Vg \d t<+\infty\right\}}.
\end{equation}
It is a standard matter to show that the Dirichlet energy is lower semicontinuous in $L^2(M, X)$ and, as soon as $M$ is compact manifold, with compact sublevels in $L^2(M, X)$. 
Therefore for any $\eps>0$, there exists at least one minimizer $u_\eps$ for $\mathcal{I}_\eps$.
When $X$ is a CAT(0) space, the uniqueness of the minimizer follows as a consequence of 
the convexity properties of the energy that we discuss now. 

As $(X,d)$ is a CAT(0) metric space, the map $v\mapsto d^2(v,w)\,\,(w\in X)$ enjoys nice convexity 
properties. 
To this end, we recall that 
 given a metric space $(X,d)$, we say that a curve $\gamma:[0,1]\to X$ is a 
constant speed geodesic (or simply a geodesic), if 
\begin{equation}
\label{eq:constant_speed}
d(\gamma_s,\gamma_t) = d(\gamma_0,\gamma_1)\abs{t-s} \qquad \forall s, t\in (0,1).
\end{equation}
A CAT(0) space is a length space, namely, the distance between any two points can be realized as the length
of a geodesic (actually a unique geodesic) joining them, 
and it turns out 
that
\begin{equation}
\label{eq:NPC_N}
X\ni v\mapsto d^2(w,v) \qquad \textrm{ is }1\textrm{-geodesically convex}, 
\end{equation}
that is, for any given $u, v\in X$, if $\gamma:[0,1]\to X$ is the minimal geodesic joining $u$ to $v$, then 
\[
d^2(\gamma_s, w)\le (1-s)d^2(u,w) + sd^2(v,w) -s(1-s)d^2(u,v)\qquad \forall w\in X, \forall s\in (0,1). 
\]
An analogous property holds in $(L^2(M, X), d_2)$, where $d_2$ is defined in \eqref{eq:metric_L2}. 
In fact, for any given $f_0, f_1\in L^2(M, X)$,
for almost any $x\in M$ we geodesically interpolate the points $f_0(x)$ and $f_1(x)$ in $X$ via the unique 
geodesic $\gamma_s(x)$. 
Therefore, given a map $f\in L^2(M, X)$ we get, for almost any $x\in M$, 
\[
d^2(f(x), \gamma_{s}(x)) \le (1-s)d^2(f(x), f_0(x)) + sd^2(f(x), f_1(x)) -s(1-s)d^2(f_0(x),f_1(x)).
\]
Thus, integrating on $M$, we get that for all $f\in L^2(M, X)$ 
\begin{equation}
\label{eq:convex_d}
L^2(M, X)\ni h\mapsto d^2_2(f,h)\qquad \textrm{ is }1\textrm{-geodesically convex}.
\end{equation}
Being CAT(0) also has important consequences on the convexity of the Dirichlet energy. 
Namely, we have the following (see 
\cite{KS})
\begin{lemma}[Convexity of $E$]
\label{lemma:convexity_E}
Let $X$ be a CAT(0) metric space. 
Then the energy $E$ is geodesically convex and satisfies
that 
for any $u_0, u_1\in D(E)$ there exists a constant speed geodesic $\gamma_t$ with $\gamma_0 = u_0$, $\gamma_1 = u_1$ such that $\gamma_t\in H^1(M, X)$ and 
\begin{equation}
\label{eq:convexity_E}
E(\gamma_t) \le (1-t)E(\gamma_0) + tE(\gamma_1) -t(1-t)\int_{M}\abs{\nabla d(u_0,u_1) }^2\Vg,
 \qquad \forall t\in [0,1].
\end{equation}
\end{lemma}

%
%
%
%
%

\begin{theorem}[Convexity of the WED functional]
\label{th:convexity_wed}
Let $X$ be a CAT(0) metric space. Then the WED functional $\mathcal{I}_\eps$ is (stricly) geodesically convex, namely
$u\in \mathfrak{V}_{{\bar{u}}}$ and $v\in \mathfrak{V}_{{\bar{v}}}$ (with $\bar{u}, \bar{v}\in H^1(M, X)$) and a constant speed geodesic 
$\gamma_s$ connecting them, there holds
\begin{equation}
\label{eq:strict_conv}
\begin{split}
\mathcal{I}_\eps(\gamma_s) \le (1-s)\mathcal{I}_\eps(u) + s\mathcal{I}_\eps(v)
 -s(1-s)\iint_{M\times\R_+}\frac{e^{-t/\eps}}{\eps}\left(\frac{\eps}{2}\abs{\partial_t d(u(t),v(t))}^2 + \abs{\nabla d(u(t),v(t))}^2\right)\Vg \d t.
 \end{split}
\end{equation}

\end{theorem}
\begin{proof}
Suppose we are given $\bar{u}$ and $\bar{v}$ in $H^1(M, X)$. We consider the geodesic 
$\bar{\gamma}_s\in H^1(M, X)$ for any $s\in [0,1]$
such that $\bar{\gamma}_{0} = \bar{u}$ and $\bar{\gamma}_1= \bar{v}$. Then we consider $u\in \mathfrak{V}_{{\bar{u}}}$ and $v\in \mathfrak{V}_{{\bar{v}}}$ and we let, for any fixed $t\in (0,+\infty)$, $\gamma_s(t)$ be the geodesic connecting $u(t)$ with $v(t)$, namely $\gamma_0(t) = u(t)$ and $\gamma_1(t)= v(t)$. 
Clearly we have that $\gamma_s(0) = \bar{\gamma}_s$ and that $\gamma_s\in \mathfrak{V}_{{\bar{\gamma_s}}}$.
Thanks to \cite[Corollary 2.1.3]{KS} we have that, for almost any $(x,t)\in Q$ and for any $s\in (0,1)$,
\[
\begin{split}
d^2(\gamma_s(t+h), \gamma_s(t)) &\le (1-s)d^2(u(t+h),u(t)) + s d^2(v(t+h), v(t))\\
& -s(1-s)\left(d(u(t+h),v(t+h))-d(u(t),v(t))\right)^2
\end{split}
\]
Thus,  we get 
\[
\abs{\gamma_s'}^2(t) \le (1-s)\abs{u'}^2(t) + s\abs{v'}^2(t) -s(1-s)\int_M \abs{\partial_t d(u(t),v(t))}^2\Vg.
\]

%
%
%
%
On the other hand, geodesic convexity of the energy guarantees that 
\[
E(\gamma_s(t)) \le (1-s)E(u(t)) + s E(v(t)) -s(1-s)\int_{M}\abs{\nabla d(u(t),v(t))}^2\Vg.
\]
Therefore,
\[
\begin{split}
&\mathcal{I}_\eps(\gamma_s) = \int_{0}^{+\infty}\frac{e^{-t/\eps}}{\eps}\left(\frac{\eps}{2}\abs{\gamma'_s}^2(t) + E(\gamma_s(t))\right)\d t\\
 \,\,\,\,&\le (1-s)\mathcal{I}_\eps(u) + s\mathcal{I}_\eps(v)
 -s(1-s)\iint_{M\times\R_+}\frac{e^{-t/\eps}}{\eps}\left(\frac{\eps}{2}\abs{\partial_t d(u(t),v(t))}^2 + \abs{\nabla d(u(t),v(t))}^2\right)\Vg \d t.
\end{split}
\]
\end{proof}

\begin{theorem}[Well posedness of the minimum problem]
\label{th:wp_minimum}
Let $X$ be a CAT(0) space. Then for any given ${u}_0\in H^1(M, X)$ and $\eps>0$, there exists a unique 
$u_\eps\in \mathfrak{V}_{{{u}_0}}$ such that 
\[
\mathcal{I}_\eps(u_\eps) = \min_{v\in \mathfrak{V}_{{{u}_0}}}\mathcal{I}_\eps(v).
\]
\end{theorem}
\begin{proof}
As the functional $\mathcal{I}_\eps$
 is convex (strictly), the existence and uniqueness of a minimizer depends on lower semicontinuity of $\mathcal{I}_\eps$ on 
 $\mathfrak{V}_{{{u}_0}}$, which is inherited by lower semicontinuity of the Dirichlet Energy. 
\end{proof}

Now, let $u_\eps$ be the minimizer of $\mathcal{I}_\eps$ in $\mathfrak{V}_{{u}_0}$. 
As we are interested in the behavior as $\eps\to 0$ of $u_\eps$, we collect the energy estimates and the compactness 
properties of the family $\left\{u_\eps\right\}$. Note that some of these estimates and the 
resulting compactness properties remain valid without assuming that 
$X$ is a CAT(0) space. For instance, they can be used to study the WED approximation of the heat flow of harmonic maps into
a compact smooth Riemannian manifold without the curvature assumption. 

We have the following result that follows from \cite{rsss-wed}. However, for the sake of completeness, we will give 
a direct proof of the energy estimates as a consequence of proper inner variations with respect to the time variable and the space variable (see the subsection \ref{ssec:energy_bounds} below).

\begin{theorem}[Energy estimate and compactness]
\label{th:energy_est}
Let $M$ be a $n$-dimensional complete manifold with $\partial M=\emptyset$ and let $X$ be a CAT(0) metric space. 
Let ${u}_0\in H^1(M, X)$ and  let $u_\eps\in  \argmin_{v\in \mathfrak{V}_{{u}_0}}\mathcal{I}_\eps(v)$.
Then
\begin{equation}
\label{eq:convex_E_t}
t\mapsto E(u_\eps(t)) \quad \textrm{ is convex in } (0,+\infty)
\end{equation}
and, for any $t>s>0$, 
\begin{eqnarray}
\displaystyle\int_{s}^t\abs{u_\eps'}^2(\tau) \d \tau \le E(u_0),\label{eq:bound_dissipation}\\
\displaystyle\int_{s}^t E(u_\eps(\tau)) \d \tau\le ((t-s) +\eps) E(u_0), \label{eq:bound_energy}\\
\displaystyle \sup_{t\in (0,+\infty)}E(u_\eps(t))\le E({u}_0). \label{eq:better_energy}
\end{eqnarray}
Consequently, there exist $u\in 
\mathfrak{V}_{{u}_0}$
and a not relabelled sequence of $\eps\searrow 0$ such that 
\begin{equation}
\label{eq:compactness}
\begin{split}
&u_{\eps}(t) \to u(t) \qquad \textrm{ in } L^2_{\textrm{loc}}(M, X) \qquad \forall t\in [0,+\infty),\\
&\liminf_{\eps\to 0}\int_{0}^\infty E(u_\eps(t))\d t \ge \int_{0}^\infty E(u(t))\d t,\\
&\abs{u_\eps'}\weak v \qquad \textrm{ in } L^2_{\textrm{loc}}(0,+\infty), \,\,\,\,v\ge \abs{u'} \,\,\,\textrm{ for a.e. in } (0,+\infty).\\
\end{split}
\end{equation}
\end{theorem}
\begin{proof}
See Corollary 4.4 and Theorem 2.6 in \cite{rsss-wed} for \eqref{eq:bound_dissipation} and \eqref{eq:bound_energy}. 
These two estimates hold without any assumption on the curvature of the target. Contrarily, \eqref{eq:convex_E_t} and \eqref{eq:better_energy} require the geodesic convexity of the energy $E$
(see \cite[Theorem 6.4 and Lemma 6.6]{rsss-wed}) which follows from the 
CAT(0) property. 
\end{proof}
\begin{remark}
Note that when we have some smoothness in the target (for instance when $X\hookrightarrow\R^L$ is a compact Riemannian submanifold), 
the second and third in \eqref{eq:compactness} can be rewritten as 
\[
\begin{split}
u_{\eps} \longrightarrow u\qquad \hbox{ weakly in } L^2(0,T; H^1(M, X)),\\
\partial_t u_{\eps}\longrightarrow \partial_t u\qquad \hbox{ weakly in } L^2(0,T; L^2(M, X)).
\end{split}
\]
\end{remark}

The WED approximation of the heat flow of harmonic maps has an interesting geometric flavor.
We introduce on the space time cylinder $M_\infty$ the following $\eps$-dependent metric, which we call  
WED metric. 
For $t\in (0,+\infty)$, we define 
${g}_\eps = e^{2\phi_\eps(t)}\left(\eps g + \d t^2 \right)$
where  $\phi_\eps(t)= \frac{1}{n-1}\left(-\frac{t}{\eps}-\frac{n}{2}\log\eps\right)$, the conformal metric to the cylindrical one $g_{\cy, \eps}\dpu \frac{1}{\eps}\d t^2 + g.$
The following theorem is a classical result.

\begin{theorem}[WED Minimizers are harmonic maps]
\label{th:ex_minimizer}
Let $X$ be a CAT(0) metric space. For any $\bar{u}\in D(E)$ we let $u_\eps$ be the unique minimizer of the 
the WED functional $\mathcal{I}_\eps$. 
Then, 
 $u_\eps$ is a (minimizing) harmonic map from $(M\times\R_+, {g}_\eps)$ to $X$. Furthermore, for every $\eps >0$, the minimizer $u_\eps$ is Lipschitz on $M \times (0,\infty)$. 
 \end{theorem}
\begin{proof}
The WED functional $\mathcal{I}_\eps$ rewrites, once we introduce the metric ${g}_\eps$, as
\[
\label{eq:harmonic_wed}
\mathcal{I}_\eps(v) = \int_{0}^{+\infty}\frac{e^{-t/\eps}}{\eps}\left\{\int_M\left(\frac{\eps}{2}\abs{\partial_t v}^2  + \frac{1}{2}\abs{\nabla v}^2 \right)\Vg\right\}\d t = \iint_{M\times \R_+}\bar{e}_{\eps}(v) \Vge, 
\]
where 
\[
\bar{e}_{\eps}(v) \dpu \frac{1}{2}\abs{\bar{\nabla}v}_{{g}_\eps}^2, \ \ \ \Vge =\eps^{n/2}e^{(n+1)\phi_\eps(t)}\Vg\d t.
\]
The regularity statement follows from \cite{KS} (see aslo \cite{GS}, \cite{KS2}, or Zhang-Zhu \cite{zhang-zhu}). We would like also to point out some Bochner type results into NPC spaces by Mondino-Semola \cite{mondino-semola} and Gigli \cite{gigliLip} (see also \cite{freidin}). 
\end{proof}
\begin{remark}[Digression on general smooth manifolds]
We would like to point out that NPC character in the previous statement amounts to the geodesic convexity of the functional. Whenever the target $X$ is a smooth Riemannian manifold much more can be said.  Assume for now that $X=N$ is a smooth Riemannian manifold isometrically embedded in $\R^L$ by Nash's theorem, then $u_\eps$ addtionnally satisfies the following equation in the distributional sense 
\begin{equation}
\label{eq:EL}
\begin{cases}
-\eps\partial_t^{2}u_\eps + \partial_t u_\eps -\Delta u_\eps \perp T_{u_\eps}N &\textrm{ in }  M\times (0,+\infty). \\
u(x,0) = u_0(x) &\textrm{ on } M.
\end{cases}
\end{equation}
Finally, there exists a closed $\mathcal{S}_\eps\subset M\times (0,\infty)$, with Hausdorff dimension $\textrm{dim}_{\mathscr{H}}(\mathcal{S}_\eps)\le n-2$,
such that $u_\eps\in C^{\infty}(M\times (0,\infty)\setminus \mathcal{S}_\eps; N)$.
If, in addition, 
the sectional curvature $K_N$ of $N$ is nonpositive, then $u_\eps$ is smooth in $M\times (0,\infty)$ (i.e. $\mathcal S_\eps =\emptyset $). This follows from classical results by Schoen-Uhlenbeck \cite{schoen-uhlenbeck} (see also \cite{bookLW}). 
\end{remark}

\section{The Convergence proof for CAT(0) targets}
\label{sec:conve_general}
In this section, we prove that as $\eps\to 0$, the minimizer $u_\eps$ of the WED functional $\mathcal{I}_\eps$ converges
to the unique suitable weak solution of the heat flow of harmonic maps into the CAT(0) metric space $X$. 
Even if the proof of the convergence would follow from the results in \cite{rsss-wed}, we fully exploit the CAT(0) hypothesis 
and provide a slightly different argument. 
%
%
\subsection{ The Value Function and the Dynamic Programming Principle}

The main character of the analysis in \cite{rsss-wed} is the value function for the minimization of  $\mathcal{I}_\eps$. 

The value function $V_\eps: L^2(M, X)\to [0,+\infty]$ is 
\begin{equation}
\label{eq:value_functional}
V_\eps(u) :=\min_{v\in \mathfrak{V}_{u}} \mathcal{I}_\eps(v)
\end{equation}
and satisfies (see \cite[Lemma 5.1]{rsss-wed}, \cite{rssscras} and \cite{stefanelli2024report}) 
\begin{lemma}
\label{lem:value_function_basic}
The following hold
\begin{enumerate}
\item $V_\eps$ is lower semicontinuous in $L^2(M, X)$ and verifies 
\begin{equation}
\label{eq:upperV}
0\le V_\eps(u)\le E(u) \qquad \forall u\in L^2(M,  X).
\end{equation}
\item For every $u\in L^2(M,  X)$, the map $\eps\mapsto V_\eps(u)$ is non increasing, namely,
\begin{equation}
\label{eq:mono_V}
V_{\eps_1}(u)\le V_{\eps_0}(u)\qquad \forall \eps_1\ge \eps_0.
\end{equation} 
\item For any $u\in L^2(M, X)$, there holds
\begin{equation}
\label{eq:conv_below}
V_\eps(u)\nearrow E(u)\qquad \textrm{ as }\eps\searrow 0.
\end{equation}
\item For any $u_\eps\xrightarrow{\eps\to 0}u$ in $L^2(M, X)$, there holds
\begin{equation}
\label{eq:gammaliminf}
E(u) \le \liminf_{\eps\to 0}V_\eps(u_\eps).
\end{equation}
\end{enumerate}
\end{lemma}
Since the value function is defined via the minimization of the geodesically convex $\mathcal{I}_\eps$,
we expect that $V_\eps$ enjoys the geodesic convexity as well.
We prove the following.
\begin{prop}
\label{prop:convex_V}
Let $X$ be a CAT(0) metric space.  
Then, for any $\eps>0$, the value function $V_\eps$ is 
geodesically convex. 
\end{prop}
\begin{proof}
As in the proof of Theorem \ref{th:convexity_wed}, we take
$\bar{u}$ and $\bar{v}$ in $H^1(M, X) = D(E)$ and denote by $\bar{\gamma}_s$ the unique geodesic connecting them. 
Observe that $\bar{\gamma}_s\in D(E)$ for all $s\in(0,1)$. 
Now, take $u\in \mathfrak{V}_{\bar{u}}$ and $v\in \mathfrak{V}_{\bar{v}}$ and consider the geodesic $\gamma_s(t)$ connecting $u(t)$ with $v(t)$ for almost any $t\in (0,+\infty)$. In particular, note that for any $s\in (0,1)$ there holds $\gamma_s
\in \mathfrak{V}_{\bar{\gamma_s}}$.
Thanks to Theorem \ref{th:convexity_wed} 
\[
\begin{split}
V_\eps(\bar{\gamma}_s)\le \mathcal{I}_\eps(\gamma_s) &= \int_{0}^{+\infty}\frac{e^{-t/\eps}}{\eps}\left(\frac{\eps}{2}\abs{\gamma_s'}^2(t)+E(\gamma_s(t))\right)\d t\\
&\le (1-s)\mathcal{I}_\eps(u) + s\mathcal{I}_\eps(v).
\end{split}
\]
As $u$ and $v$ are arbirtrary in $\mathfrak{V}_{\bar{u}}$ and in $\mathfrak{V}_{\bar{v}}$, respectively, we get 
\[
V_\eps(\bar{\gamma}_s)\le (1-s)V_\eps(\bar{u}) + sV_\eps(\bar{v}),
\]
namely the geodesic convexity of $V_\eps$. 
\end{proof}


The fundamental property of the value function $V_\eps$ 
is the fact that it satisfies the {\itshape Dynamic Programming Principle} (see \cite[Proposition 4.2]{rsss-wed}). Thus, $\forall T>0$ we have that 
\begin{equation}
\label{eq:dyn_prog_princ}
V_\eps({u}) = \min_{v\in \mathfrak{V}_{{u}}}\left(\int_{0}^T\frac{e^{-t/\eps}}{\eps}\left(\frac{\eps}{2}\abs{v'}^2(t)+ E(v(t))\right)\d t + V_\eps(v(T))e^{-T/\eps}\right).
\end{equation}
In particular, any minimizer $u_\eps\in \mathfrak{V}_{{u}}$ of $\mathcal{I}_\eps$ is 
a minimizer of the above minimum problem 
and thus
\begin{equation}
V_\eps({u})= \int_{0}^T\frac{e^{-t/\eps}}{\eps}\left(\frac{\eps}{2}\abs{u_\eps'}^2(t) + E(u_\eps(t))\right)\d t + V_\eps(u_\eps(T))e^{-T/\eps}
\end{equation}
We get the following (see \cite[Proposition 4.3]{rsss-wed})
\begin{prop}[Fundamental properties of the value function]
Let ${u}\in D(V_\eps)$ and $u_\eps\in \argmin_{v\in \mathfrak{V}_{{u}}} \mathcal I_\eps$. Then, $t\mapsto V_\eps(u_\eps(t))$ is absolutely continuous and satisfies
\begin{equation}
\label{eq:value_function_fundamental}
\begin{split}
-\frac{\d}{\d t}V_\eps(u_\eps(t)) &= \frac{1}{2}\abs{u_\eps'}^2(t) + \frac{1}{\eps}E(u_\eps(t))-\frac{1}{\eps}V_\eps(u_\eps(t))\qquad \textrm{ for a.a. }t\in (0,+\infty),\\
V_\eps(u_\eps(t)) &= E(u_\eps(t))-\frac{\eps}{2}\abs{u_\eps'}^2(t) \qquad \textrm{ for a.a. }t\in (0,+\infty).
\end{split}
\end{equation}
\end{prop}

In finite dimension the value function satisfies a Hamilton Jacobi equation. When $E$ is convex, we have a similar result even in 
the infinite dimensional context. 
Following \cite{rsss-wed}, we set 
\begin{equation}
\label{eq:slopeV}
\abs{{\partial}V_\eps}(u):=\limsup_{v\xrightarrow{L^2} u}\frac{(V_\eps(u)-V_\eps(v))^+}{d_2(u,v)};
\quad
\abs{\tilde{\partial}V_\eps}(u):=\limsup_{v\xrightarrow{E} u}\frac{(V_\eps(u)-V_\eps(v))^+}{d_2(u,v)},
\end{equation}
where $v\xrightarrow{E} u$ means that $v\to u$ in $L^2(M, X)$ and $E(v)\longrightarrow E(u)$. 
Note that the definition of $\abs{{\partial}V_\eps}$ and $\abs{\tilde{\partial}V_\eps}$ implies that 
\begin{equation}
\label{eq:comparing_slope}
\abs{\tilde{\partial}V_\eps}(u)\le \abs{\partial V_\eps}(u),\qquad \forall u \in L^2(M, X).
\end{equation}

We thus have (see \cite[Theorem 7.1]{rsss-wed})
\begin{theorem}[Hamilton-Jacobi 
{equation} for $V_\eps$]
\label{th:betterV}
Assume $(X,d)$ is CAT(0) and 
$u\in D(E)$. Then the value function $V_\eps(u)$ satisfies the Hamilton-Jacobi equation
\begin{equation}
\label{eq:HJ-V}
\frac{1}{2}\abs{\tilde{\partial}V_\eps}^2(u)= \frac{E(u)-V_\eps(u)}{\eps}.
\end{equation} 
In particular, for any ${u}\in D(E)$ and $\eps>0$,  $u_\eps\in \argmin_{v\in \mathfrak{V}_{{u}}}\mathcal I_\eps(v)$ is 
a curve of maximal slope for $V_\eps$
with respect to the slope $\abs{\tilde{\partial}V_\eps}$, namely 
\begin{equation}
\label{eq:cmsV}
-\frac{\d}{\d t}V_\eps(u_\eps(t)) = \frac{1}{2}\abs{u_\eps'}^2(t) + \frac{1}{2}\abs{\tilde{\partial}V_\eps}^2(u_\eps(t)) \qquad \textrm{ for a.a. }t\in (0,+\infty). 
\end{equation}
\end{theorem}
From Theorem \ref{th:betterV}, we deduce the following
\begin{corollary}
\label{cor:minV=minE}
Let ${u}\in D(E)$ be a minimum point of $V_\eps$. Then ${u}$ is also a minimum point of $E$, and 
$E(u) = V_\eps(u)$.  
\end{corollary}
\begin{proof}
On the one hand, thanks to \eqref{eq:upperV}, we have that 
\[
V_\eps({u}) \le \min E.
\]
On the other hand, thanks to Hamilton Jacobi equation \eqref{eq:HJ-V}, we have that 
\[
0 = \frac{1}{2}\abs{\tilde{\partial}V_\eps}^2({u})= \frac{E({u})-V_\eps({u})}{\eps},
\]
and thus 
\[
E({u}) = V_\eps({u}) \le \min E,
\]
namely, $E({u}) = \min E$. 
\end{proof}

Theorem \ref{th:betterV} shows that the Euler-Lagrange equation for a minimizer of $\mathcal I_\eps$ can be understood as a curve of maximal slope for the geodesically convex value function $V_\eps$ with respect to the slope $\abs{\tilde{\partial}V_\eps}$. 
Now we show that the unique minimizer $u_\eps$ actually verifies the Evolution Variational Inequality
\begin{equation}
\label{eq:evi_u_eps}
\frac{\d}{\d t}d^2_2(u_\eps(t), v) + V_\eps(u_\eps(t)) \le V_\eps(v)\qquad \textrm{ in }\mathscr{D}'(0,+\infty), \,\,\forall v\in D(E).
\end{equation}
We argue as follows. Since $V_\eps$ is geodesically convex and $d_2^2$ is $1$-geodesically convex, 
it turns out (see, e.g., \cite[Chapter 4]{ags2005}) that for any ${u}\in D(V_\eps)$ there exists a unique $v_\eps:[0,+\infty)\to L^2(M, X)$ that satifies the variational evolution inequality 
\begin{equation}
\label{eq:evi_v}
\frac{\d}{\d t}d^2_2(v_\eps(t), v) 
+ V_\eps(v_\eps(t))\le V_\eps(v), \quad  \textrm{ in }\mathscr{D}'(0,+\infty), \qquad \forall v\in D(V_\eps).
\end{equation}
In particular, we have that $v_\eps$ is a curve of maximal slope w.r.t. $\abs{\partial V_\eps}$ and therefore w.r.t. $\abs{\tilde{\partial} V_\eps}$, thanks to \eqref{eq:comparing_slope}, namely,
\[
-\frac{\d}{\d t}V_\eps(v_\eps(t))\ge \frac{1}{2}\abs{v_\eps'}^2(t) +\frac{1}{\eps}E(v_\eps(t))-\frac{1}{\eps}V_\eps(v_\eps(t)).
\]
Multiply both sides of the former inequality by $e^{-t/\eps}$, integrate over $(0,t)$ and use that $V_\eps$ is nonnegative. We get
\[
V_\eps({u}) \ge \int_{0}^t\frac{e^{-s/\eps}}{\eps}\left(\frac{\eps}{2}\abs{v_\eps'}^2(s) + E(v_\eps(s))\right)\d s + V_\eps(v_\eps(t))e^{-t/\eps}.
\]
Therefore, in the limit $t\to +\infty$, 
\[
\min_{w\in \mathfrak{V}_{{u}}}\mathcal{I}_\eps(w) = V_\eps({u})\ge \int_{0}^{\infty}\frac{e^{-t/\eps}}{\eps}\left(\frac{\eps}{2}\abs{v_\eps'}^2(t) + E(v_\eps(t))\right)\d t = \mathcal{I}_\eps(v_\eps),
\]
which implies that $v_\eps = u_\eps$, thanks to the uniqueness of the WED minimizer, hence the result. We collect these statements in the following
\begin{theorem}
\label{th:evi_V}
For ${u}\in H^1(M, X)$ and $\eps>0$, let $u_\eps=\argmin_{v\in \mathfrak{V}_{{u}}} \mathcal{I}_\eps(v)$. 
Then $u_\eps$ verifies the EVI for the value function $V_\eps$, namely,
\begin{equation}
\label{eq:evi_V}
\frac{1}{2}\frac{\d}{\d t}d^2_2(u_\eps(t),v) 
+ V_\eps(u_\eps(t))\le V_\eps(v), \qquad \forall t>0, \,\,\forall v\in D(V_\eps).
\end{equation}
\end{theorem}

We have now all the ingredients to start the limiting process with respect to $\eps$ and conclude the proof of Theorem \ref{th:main_intro_0}.
First, recall that (see Thereom \ref{th:energy_est}) there exist $u\in \mathfrak{V}_{u_0}$ a non relabelled subsequence of $\eps$ along which 
\[
\begin{split}
&u_{\eps}(t) \longrightarrow u(t) \qquad \textrm{ in } L^2_{\textrm{loc}}(M, X) \qquad \forall t\in [0,+\infty),\\
&\liminf_{\eps\to 0}\int_{0}^\infty E(u_\eps(t))\d t \ge \int_{0}^\infty E(u(t))\d t,\\
&\abs{u_\eps'}\weak v \qquad \textrm{ in } L^2_{\textrm{loc}}(0,+\infty), \,\,\,\,v\ge \abs{u'} \,\,\, \textrm{for a.e. in } (0,+\infty).
\end{split}
\]
Moreover, thanks to Lemma \ref{lem:value_function_basic}, we have
\begin{equation}
\label{eq:liminfEV}
E(u(t))\le \liminf_{\eps\to 0}V_\eps(u_\eps(t)), \qquad \textrm{ for any } t>0.
\end{equation}

Then we test \eqref{eq:evi_V} with $\vf\in C^{\infty}_c(0,+\infty)$ such that $\vf\ge 0$ and we recall that $D(E)\subseteq D(V_\eps)$.
We get, for any $v\in D(E)$,
\[
-\frac{1}{2}\int_{0}^{\infty}d^2_2(u_\eps(t),v)\vf'(t)\d t 
+\int_{0}^{\infty}V_\eps(u_\eps(t))\vf(t)\d t \le \int_{0}^{\infty}V_\eps(v)\vf(t)\d t.
\]
Thus \eqref{eq:conv_below} and \eqref{eq:liminfEV} give 
\[
-\frac{1}{2}\int_{0}^{\infty}d^2_2(u(t),v)\vf'(t)\d t 
+\int_{0}^{\infty}E(u(t))\vf(t)\d t \le \int_{0}^{\infty}E(v)\vf(t)\d t, 
\]
that is the Evolution Variational Inequality
\[
\frac{\d}{\d t}\frac{1}{2}d^2_2(u(t), v)+ E(u(t))\le E(v)\qquad \textrm{ in } \mathscr{D'}(0,+\infty), \,\,\,\forall v\in D(E). 
\]
Note that the EVI formulation brings uniqueness of the flow and the contraction estimate
\[
d^2_2(u(t), v(t))\le d^2_2({u}_0, {v}_0) \qquad \forall t>0,
\]
where $u, v$ are solutions originating in ${u}_0$ and in ${v}_0$, respectively. 
Finally note also that since $u$ verifies the EVI above, the map $t\mapsto d^2_2(u(t), Q)$ is Lipschitz
for any $Q\in X$ (see, for instance, \cite[Theorem 4.0.4]{ags2005}).


 
\section{Monotonicity of parabolic frequency functions} 
\label{sec:mono}
In this section, we will assume that $(X,d)$ is CAT(0) and
present some crucial properties of minimizing maps into $(X,d)$, $u_\eps={\rm{argmin}} \big\{\mathcal{I}_\eps(v): \ v\in \mathfrak{V}_{u_0}\big\}$.
They are obtained by both the inner and outer variations of the WED functional $\mathcal{I}_\eps$ at $u_\eps$,
which ultimately leads to the monotonicity of the parabolic frequency functions $N(u; z_0, R, Q)$,
analogous to Almgren and Poon,  for the limiting map $u$ of $u_\eps$.

\subsection{$\mathcal{L}_\eps$-subharmonicity of $e_\eps(u_\eps)$} For any $Q\in X$, define $\rho_\varepsilon(x,t)
=d^2(u_\varepsilon (x,t), Q)$.
Then we have 
\begin{lemma} \label{sub-harmonicity} For any $\varepsilon>0$, $\rho_\varepsilon$ satisfies
\begin{equation}\label{subharmonic}
(\partial_t-\Delta)\rho_\varepsilon 
-\varepsilon \partial^2_t\rho_\varepsilon 
\le -2(\varepsilon |\partial_t u_\varepsilon|^2+|\nabla u_\varepsilon|^2)\ \ {\rm{in}}\  \ \mathscr{D}'(M\times (0,\infty)).
\end{equation}
\end{lemma}

\smallskip
\begin{proof} Fixed $Q\in X$, let $R_{\lambda,Q}: [0,1]\times X\to X$ be the contraction map
defined in section 1, see also Gromov and Schoen \cite{GS}.

For any nonnegative $\phi\in C_0^\infty(M\times (0,\infty)$ and $\tau\ge 0$, we define a family of outer variations of $u_\varepsilon$ by letting
\[
u_\varepsilon^\tau(x,t)=R_{1-\tau\phi(x,t), Q}(u_\varepsilon(x,t)), \ \forall (x,t)\in M\times [0,\infty).
\]
Note that $u_\eps^\tau\in \mathfrak{V}_{u_0}$, and since $u_\eps^0=u_\eps$ minimizes $\mathcal{I}_\eps$, we must have
$$\frac{d}{d\tau}\big|_{\tau=0} \mathcal{I}_\eps(u_\eps^\tau)\ge 0.$$
As in \cite{GS},  direct calculations imply
\begin{align*}
\Big|\frac{\partial u_\varepsilon^\tau}{\partial x_i}(x,t)\Big|^2
&=\Big|D_{\frac{\partial u_\varepsilon}{\partial x_i}} R_{1-\tau\phi(x,t),Q}(u_\varepsilon(x,t))\Big|^2
-\tau \frac{\partial\phi}{\partial x_i}
\frac{\partial d^2(R_{1-\tau\phi(x,t),Q}(u_\varepsilon(x,t), Q)}{\partial x_i}\\
&+\tau^2\big(\frac{\partial\phi}{\partial x_i}\big)^2
\Big|\frac{\partial R_{1-\tau\phi(x,t),Q}}{\partial\lambda}(u_\varepsilon(x,t))\Big|^2, \ 1\le i\le n, 
\end{align*}
and
\begin{align*}
\Big|\frac{\partial u_\varepsilon^\tau}{\partial t}(x,t)\Big|^2
&=\Big|D_{\frac{\partial u_\varepsilon}{\partial t}} R_{1-\tau\phi(x,t),Q}(u_\varepsilon(x,t))\Big|^2
-\tau \frac{\partial\phi}{\partial t}
\frac{\partial d^2(R_{1-\tau\phi(x,t),Q}(u_\varepsilon(x,t), Q)}{\partial t}\\
&+\tau^2\big(\frac{\partial\phi}{\partial t}\big)^2
\Big|\frac{\partial R_{1-\tau\phi(x,t),Q}}{\partial\lambda}(u_\varepsilon(x,t))\Big|^2.
\end{align*}    
Hence
\begin{align*}
\mathcal{I}_\varepsilon(u_\varepsilon^\tau)&=\mathcal{I}_\varepsilon(u_\varepsilon)-2\tau \iint e^{-\frac{t}{\varepsilon}}
\big(\varepsilon|\partial_t u_\varepsilon|^2+|\nabla u_\varepsilon|^2)\phi\,dxdt\\
&-\tau\iint e^{-\frac{t}{\varepsilon}}\big(\varepsilon\partial_t\phi\cdot\partial_t d^2(R_{1-\tau\phi, Q}(u_\varepsilon),Q)
+\nabla\phi\cdot\nabla d^2(R_{1-\tau\phi, Q}(u_\varepsilon),Q)\big)\,dxdt+O(\tau^2).
\end{align*}
This, together with $\frac{d}{d\tau}\big|_{\tau=0} \mathcal{I}_\eps(u_\eps^\tau)\ge 0$,  implies that 
\begin{align*}   
0\le -2\iint e^{-\frac{t}{\varepsilon}}
\big(\varepsilon|\partial_t u_\varepsilon|^2+|\nabla u_\varepsilon|^2)\phi\,dxdt
-\iint e^{-\frac{t}{\varepsilon}}\big(\varepsilon\partial_t\phi\cdot\partial_t d^2(u_\varepsilon,Q)
+\nabla\phi\cdot\nabla d^2(u_\varepsilon,Q)\big)\,dxdt.
\end{align*}
By integration by parts, we obtain that
\begin{align*}   
2\iint e^{-\frac{t}{\varepsilon}}
\big(\varepsilon|\partial_t u_\varepsilon|^2+|\nabla u_\varepsilon|^2)\phi\,dxdt
\le\iint e^{-\frac{t}{\varepsilon}}\Big[-\partial_t\phi +\varepsilon\partial_t^2\phi+\Delta\phi\Big] d^2(u_\varepsilon,Q)\,dxdt.
\end{align*}
If we choose $\phi=e^{\frac{t}{\eps}}\psi$, then $\psi\in C_0^\infty(\R^n\times (0,\infty))\ge 0$, and 
\begin{equation}\label{subharmonic1}
2\iint (\varepsilon |\partial_t u_\varepsilon|^2
+|\nabla u_\varepsilon|^2)\psi\,dxdt
\le \iint d^2(u_\varepsilon,Q)\big[(\partial_t\psi+\Delta\psi) 
+\varepsilon \partial^2_t\psi\big]\,dxdt.
\end{equation}
This completes the proof.
\end{proof}

\smallskip
\subsection{Stationarity identity of $u_\eps$}
{Since the result is of local nature, for simplicity we will assume that $M=\R^n$ with the Euclidean metric. }
For any vector-field $Y\in C^\infty(\mathbb R^n\times\mathbb R_+,\mathbb{R}^{n+1})$, with $Y=0$ for $t$ near $0$,  
we define $u_\eps^\tau(x,t)=u_\eps((x,t)+\tau Y(x,t))$ for sufficiently small $|\tau|>0$.
Then, by the minimality of $u_\eps$, direct calculations imply 
\begin{align}\label{stationarity0}
0&=\frac{d}{d\tau}\big|_{\tau=0}\frac12\iint_{\mathbb R^{n+1}_+} e^{-\frac{t}{\eps}}(\eps |\partial_t u_\eps^\tau|^2+|\nabla u_\eps^\tau|^2)\,\d x\d t\\
&=\iint_{\mathbb R^{n+1}_+} e^{-\frac{t}{\eps}}\Big\{ \eps \partial_tu_\eps\big(\partial_{x_j} u_\eps \partial_tY^j
+\partial_tu_\eps \partial_t Y^{n+1}\big) +\partial_{x_i} u_\eps\big(\partial_{x_j} u_\eps \partial_{x_i} Y^j
+\partial_tu_\eps \partial_{x_i}Y^{n+1}\big)\Big\}\,\d x\d t\nonumber\\
&\ \ \ +\frac12\iint_{\mathbb R^{n+1}_+} e^{-\frac{t}{\eps}}\big(\eps |\partial_tu_\eps|^2+|\nabla u_\eps|^2\big)\big(\frac{1}{\eps}Y^{n+1}-\textrm{div}_{(x,t)} Y\big)\,\d x\d t.\nonumber
\end{align}

\medskip
\subsection{Energy bounds of $u_\eps$}
\label{ssec:energy_bounds}
If we choose $Y=(0', Z(t))$, with $Z\in C^\infty(\mathbb{R}_+,\mathbb R)$ and $Z=0$ for $t$ near $0$,  then \eqref{stationarity0}
reduces to
\begin{align*}
0=2\eps\iint_{\mathbb R^{n+1}_+} e^{-\frac{t}{\eps}}|\partial_tu_\eps|^2 Z'(t) \,\d x\d t
+\iint_{\mathbb R^{n+1}_+} e^{-\frac{t}{\eps}}\big(\eps |\partial_tu_\eps|^2+|\nabla u_\eps|^2\big)(\frac{1}{\eps}Z-Z'(t))\,\d x\d t.
\end{align*}
For $\eta\in C_0^\infty(0,\infty)$, define $Z(t)=\int_0^t \eta(s)\,ds,  \ t\ge 0 $ (see also \cite{serra-tilli}).  
Then $Z\in C^\infty(\mathbb R_+)$ and $Z=0$ for $t$ near $0$. Substituting $Z$ into the above identity,
we obtain that
\begin{align*}
0=2\eps\iint_{\mathbb R^{n+1}_+} e^{-\frac{t}{\eps}}|\partial_t u_\eps|^2 \eta(t) \,\d x\d t
+\iint_{\mathbb R^{n+1}_+} e^{-\frac{t}{\eps}}\big(\eps |\partial_t u_\eps|^2+|\nabla u_\eps|^2\big)(\frac{Z}{\eps}-\eta(t))\,\d x\d t.
\end{align*}

For $0\le t<\infty$, define 
$$I_\eps(t)=\int_{\mathbb R^n} |\partial_tu_\eps|^2\,\d x,$$
$$L_\eps(t)=\int_{\mathbb R^n} \big(\eps |\partial_tu_\eps|^2+|\nabla u_\eps|^2\big)(x,t)\,\d x,$$
and
$$H_\eps(t)=\int_t^\infty e^{-\frac{s}{\eps}}\int_{\mathbb R^n} \big(\eps |\partial_tu_\eps|^2+|\nabla u_\eps|^2\big)(x,s)\,\d x\d s
=\int_t^\infty e^{-\frac{s}{\eps}} L_\eps(s)\,\d s.$$
Then
\begin{align*}
\iint_{\mathbb R^{n+1}_+} e^{-\frac{t}{\eps}}\big(\eps |\partial_t u_\eps|^2+|\nabla u_\eps|^2\big)\frac{1}{\eps}Z(t)\,\d x\d t
&=-\frac{1}{\eps}\int_0^\infty H_\eps'(t) Z(t)\,\d t\\
&=-\frac{1}{\eps}\Big(H_\eps(t)Z(t)\Big|_{t=0}^\infty-\int_0^\infty \eta(t)H_\eps(t)\,\d t\Big)\\
&=\frac{1}{\eps}\int_0^\infty \eta(t)H_\eps(t)\,\d t,
\end{align*}
where we have used the fact that $H_\eps(\infty)=Z(0)=0$. Substituting this into the above identity, we obtain that
$$
\int_0^\infty(2\eps I_\eps(t)+\frac{1}{\eps}e^{\frac{t}{\eps}}H_\eps(t)-L_\eps(t)) \eta(t)\,\d t=0.
$$
Since $\eta\in C_0^\infty(0,\infty)$ is arbitrary, this yields
\begin{equation}\label{energyineq0}
2\eps I_\eps(t)+\frac{1}{\eps}e^{\frac{t}{\eps}}H_\eps(t)-L_\eps(t)=0.
\end{equation}
Note that
$$
\big(e^{\frac{t}{\eps}}H_\eps(t)\big)'=\frac{1}{\eps}e^{\frac{t}{\eps}}H_\eps(t)-L_\eps(t),
$$
we arrive at
$$
2I_\eps(t)+\frac{1}{\eps}\big(e^{\frac{t}{\eps}}H_\eps(t)\big)'=0.
$$
Integrating it over $0\le t\le T$, $0<T<\infty$, yields
$$
2\int_0^TI_\eps(t)\,dt+\frac{1}{\eps}e^{\frac{T}{\eps}}H_\eps(T)=\frac{1}{\eps}H_\eps(0).
$$

By the minimality of $u$, we know that
\begin{align*}
\frac{1}{\eps}H_\eps(0)&=\frac{1}{\eps}\int_0^\infty e^{-\frac{t}{\eps}}\int_{\mathbb R^n} \big(\eps |\partial_t u_\eps|^2+|\nabla u_\eps|^2\big)(x,t)\,\d x\d t\\
&\le \frac{1}{\eps}\int_0^\infty e^{-\frac{t}{\eps}}\int_{\mathbb R^n} \big(\eps |\partial_tu_0|^2+|\nabla u_0|^2\big)(x,t)\,\d x\d t\\
&=\big(\frac{1}{\eps} \int_0^\infty e^{-\frac{t}{\eps}}\,\d t\big)\int_{\mathbb R^n}|\nabla u_0|^2\,\d x
=2E(u_0).
\end{align*} 
Therefore we have 
\begin{equation}\label{energyineq}
2\int_0^TI_\eps(t)\,\d t+\frac{1}{\eps}e^{\frac{T}{\eps}}H_\eps(T)\le 2E(u_0).
\end{equation}
It follows from \eqref{energyineq} that
\begin{equation}\label{energyineq1}
\int_0^T\int_{\mathbb R^n}|\partial_t u_\eps|^2\,\d x\d t\le E(u_0),  \ \forall T>0,
\end{equation}
and 
\begin{equation}\label{energyineq2}
\frac{1}{\eps}e^{\frac{t}{\eps}}H_\eps(t)\le 2E(u_0), \ \forall t>0.
\end{equation}
On the other hand, integrating \eqref{energyineq0} over $0<s<t$ implies
\begin{align}
\int_s^t\int_{\mathbb R^n}|\nabla u_\eps|^2
&\le \int_s^t L_\eps(\tau)\,d\tau=\int_s^t (2\eps I_\eps(\tau)+\frac{1}{\eps}e^{\frac{\tau}{\eps}}H_\eps(\tau))\,\d \tau\nonumber\\
&\le 2\eps \int_s^t\int_{\mathbb R^n}|\partial_t u_\eps|^2\,\d x\d\tau+(t-s)\sup_{\tau>0} \frac{1}{\eps}e^{\frac{\tau}{\eps}}H_\eps(\tau)\nonumber\\
&\le 2(t-s+\eps)E(u_0).
\end{align}

For any $0\le t_1<t_2<\infty$, if we replace $Z(t)=e^{\frac{t}{\eps}}Z_\delta(t)\in C^\infty_0(t_1,t_2)$,
with $Z_\delta(t)\xrightarrow{\delta\to 0}\chi_{[t_1,t_2]}(t)$, in \eqref{stationarity0}, then
\begin{align*}
&\int_{\mathbb R^n} (|\nabla u_\eps|^2-\eps |\partial_t u_\eps|^2)\,dx\big|_{t=t_1}- \int_{\mathbb R^n} (|\nabla u_\eps|^2-\eps |\partial_t u_\eps|^2)\,\d x\big|_{t=t_2}\\
&=\int_{t_1}^{t_2}\int_{\mathbb R^n} \big(|\partial_t u_\eps|^2+\frac{1}{\eps}|\nabla u_\eps|^2\big)\,\d x\d t
+\int_{t_1}^{t_2}\int_{\mathbb R^n}\frac{1}{\eps}\big(\eps|\partial_t u_\eps|^2-|\nabla u_\eps|^2\big)\,\d x \d t\\
&=\int_{t_1}^{t_2}\int_{\mathbb R^n} 2|\partial_t u_\eps|^2\,\d x\d t.
\end{align*}
If we choose $t_1=0$ and $t_2=t>0$, then the above inequality implies
\begin{equation}
2\int_{0}^{t}\int_{\mathbb R^n} |\partial_tu_\eps|^2\,\d x\d t+\int_{\mathbb R^n} (|\nabla u_\eps|^2-\eps |\partial_tu_\eps|^2)(x,t)\,\,\d x
\le 2E(u_0).
\end{equation} 

\medskip
\subsection{Energy monotonicity of $u_\eps$} To derive it, we first modify  \eqref{stationarity0} into a slightly different form. 
Replacing $Y=e^{\frac{t}{\eps}}Z$, with $Z\in C_0^\infty(\R^n\times (0,\infty),\R^{n+1})$, into  \eqref{stationarity0}, we obtain
\begin{align}\label{stationarity00}
0&=\iint_{\mathbb R^{n+1}_+} 
\Big\{ \varepsilon \big(\partial_tu_\eps\cdot \partial_{x_j} u_\eps \partial_tZ^j+|\partial_tu_\eps|^2 \partial_t Z^{n+1}\big) 
+\big(\partial_t u_\eps\cdot \partial_{x_j}u_\eps Z^j+|\partial_tu_\eps|^2 Z^{n+1}\big)\nonumber\\
&\qquad\qquad\qquad +\big(\partial_{x_i} u_\eps\cdot \partial_{x_j} u_\eps\partial_{x_i}Z^j+\partial_t u_\eps \cdot \partial_{x_i} u_\eps \partial_{x_i} Z^{n+1}\big)\Big\}\,dxdt\nonumber\\
&\ \ \ -\frac12\iint_{\mathbb R^{n+1}_+} 
\big(\eps |\partial_tu_\eps|^2+|\nabla u_\eps|^2\big)\textrm{div}_{(x,t)} Z\,dxdt
\end{align}
holds for all $Z\in C_0^\infty(\R^n\times (0,\infty),\R^{n+1})$.

If we choose $Z=(Z^1,\cdots, Z^n, 0)$, then 
\eqref{stationarity00} yields
\begin{align}\label{stationarity01}
0&=\iint_{\mathbb R^{n+1}_+} 
\Big\{ \varepsilon \partial_tu_\eps\cdot \partial_{x_j}u_\eps \partial_tZ^j
+\partial_tu_\eps\cdot \partial_{x_j} u_\eps Z^j
+\partial_{x_i} u_\eps\cdot \partial_{x_j} u_\eps \partial_{x_i}Z^j\Big\}\,dxdt\nonumber\\
&\ \ \ -\frac12\iint_{\mathbb R^{n+1}_+} 
\big(\eps |\partial_tu_\eps|^2+|\nabla u_\eps|^2\big) \partial_{x_j}Z^j\,dxdt.
\end{align}
And, if we choose $Z=(0,\cdots, 0, Z^{n+1})$, then
\eqref{stationarity00} yields
\begin{align}\label{stationarity02}
0&=\iint_{\mathbb R^{n+1}_+} 
\Big\{ \varepsilon |\partial_tu_\eps|^2 \partial_tZ^{n+1}
+|\partial_tu_\eps|^2 Z^{n+1}
+\partial_tu_\eps\cdot \partial_{x_i} u_\eps \partial_{x_j} Z^{n+1}\Big\}\,dxdt\nonumber\\
&\ \ \ -\frac12\iint_{\mathbb R^{n+1}_+} 
\big(\eps |\partial_tu_\eps|^2+|\nabla u_\eps|^2\big) \partial_tZ^{n+1}\,dxdt.
\end{align}

Assume $z_0=(0,0)$. Let $G=G_{(0,0)}$ be the backward heat kernel centered at $(0,0)$. For $t_2<t_1<0$, let $\theta^\delta
\in C_0^\infty([t_2, t_1])$ be such that
$\theta^\delta\to \chi_{[t_2,t_1]}$ as $\delta\to 0$. Similar to Feldman \cite{Feldman}, 
by substituting $Z^{n+1}=tG\theta^\delta$ into \eqref{stationarity02} and applying
\[
\partial_t(tG) =\big(-\frac{n-2}2 -\frac{|x|^2}{4t}\big)G,
\ \ \nabla G=\frac{x}{2t} G, 
\]
we obtain, after sending
$\delta\to 0$, that
\begin{align}\label{stationarity03}
&\frac12 \iint_{\mathbb R^{n}\times [t_2, t_1]} 
\big(\eps |\partial_tu_\eps|^2+|\nabla u_\eps|^2\big) \big(-\frac{n-2}2 -\frac{|x|^2}{4t}\big)G\nonumber\\
&\qquad+
\frac12\Big(\int_{t=t_2}-\int_{t=t_1}\Big)\big(\eps |\partial_tu_\eps|^2+|\nabla u_\eps|^2\big)tG\nonumber\\
&=\iint_{\R^n\times [t_2, t_1]}
\varepsilon |\partial_t u_\eps|^2 \big(-\frac{n-2}2 -\frac{|x|^2}{4t}\big)G+\frac12 \partial_tu_\eps\cdot(x\cdot\nabla u_\eps+2t\partial_t u_\eps)G
\nonumber\\
&\qquad+\Big(\int_{t=t_2}-\int_{t=t_1}\Big)\varepsilon |\partial_t u_\eps|^2 tG.
\end{align}

\smallskip
Next we substitute $Z(x,t)=(xG(x,t)\theta^\delta(x,t), 0)$
into \eqref{stationarity01} and applying
\[
\partial_{x_j}(x^iG)=\big(\delta_{ij}+\frac{x_ix_j}{2t}\big)G,
\ div(xG)=\big(n+\frac{|x|^2}{2t}\big)G,
\]
we obtain, after sending $\delta\to 0$, that
\begin{align} \label{stationarity04}   
&\frac12\iint_{\mathbb R^n\times [t_2, t_1]} 
\big(\eps |\partial_tu_\eps|^2+|\nabla u_\eps|^2\big) \big(n+\frac{|x|^2}{2t}\big)G\,dxdt\nonumber\\
&=\iint_{\R^n\times [t_2,t_1]}\eps \partial_tu_\eps\cdot (x\cdot\nabla u_\eps) G_t
+\Big(\int_{t=t_2}-\int_{t=t_1}\Big) \eps \partial_tu_\eps\cdot (x\cdot\nabla u_\eps) G\nonumber\\
&+\iint_{\R^n\times [t_2,t_1]}
\Big(u_t\cdot(x\cdot\nabla u_\eps)G +|\nabla u_\eps|^2 G+\frac{|x\cdot\nabla u_\eps|^2}{2t} G\Big).
\end{align}

Now, by multiplying \eqref{stationarity04} by $\frac12$ and
adding it to \eqref{stationarity03}, we obtain
\begin{align}\label{stationarity05}
&\frac12 \iint_{\mathbb R^{n}\times [t_2, t_1]} 
\big(\eps |\partial_tu_\eps|^2+|\nabla u_\eps|^2\big) G+
\frac12\Big(\int_{t=t_2}-\int_{t=t_1}\Big)\big(\eps |\partial_tu_\eps|^2+|\nabla u_\eps|^2\big)tG\nonumber\\
&=\frac12\iint_{\R^n\times [t_2,t_1]} |\nabla u_\eps|^2 G
-\iint_{\R^n\times [t_2, t_1]}\frac{1}{4|t|}
\big|x\cdot\nabla u_\eps+2t \partial_t u_\eps\big|^2 G\nonumber\\
&+\frac12 \iint_{\R^n\times [t_2, t_1]}
\eps u_t\cdot (x\cdot\nabla u_\eps) G_t +
\frac12\Big(\int_{t=t_2}-\int_{t=t_1}\Big) \varepsilon \partial_t u_\eps \cdot(x\cdot\nabla u_\eps+2t \partial_t u_\eps) G.
\end{align}
In particular, by choosing $t_1=-R_1^2$ and $t_2=-R_2^2$
 with $0<R_1<R_2$, we obtain from \eqref{stationarity05} that 
\begin{align}\label{stationarity06}
&\frac12 R_2^2\int_{t=-R_2^2}\big(\eps |\partial_t u_\eps|^2+|\nabla u_\eps|^2\big)G-\frac12 R_1^2\int_{t=-R_1^2}\big(\eps |\partial_t u_\eps|^2+|\nabla u_\eps|^2\big)G\nonumber\\
&=\frac12\iint_{\R^n\times [-R_2^2,-R_1^2]} \eps|\partial_tu_\eps|^2 G
+\quad\iint_{\R^n\times [-R_2^2, -R_1^2]}\frac{1}{4|t|}
\big|x\cdot\nabla u_\eps+2t \partial_tu_\eps\big|^2 G\nonumber\\
&\quad-\frac12 \iint_{\R^n\times [-R_2^2, -R_1^2]}
\eps \partial_tu_\eps\cdot (x\cdot\nabla u_\eps) G_t \nonumber\\
&\quad-\frac12\Big(\int_{t=-R_2^2}-\int_{t=-R_1^2}\Big) \varepsilon \partial_tu_\eps \cdot(x\cdot\nabla u_\eps
+2t \partial_t u_\eps) G.
\end{align}

\subsection{Monotonicity of frequency functions for $u$} It follows from the energy bounds from section 4.3 that we can assume that
after passing to a possible subsequence, that
there exists $u\in {H}^1_{\rm{loc}}(\R^n\times (0,\infty), X)$ such that for any $Q\in X$, 
\[
d(u_\varepsilon, Q)\to d(u, Q)\ \ {\rm{in}} \ \ L^2_{\rm{loc}}(\R^n\times (0,\infty)),
\]
and
\[
(\partial_t u_\varepsilon, \nabla u_\varepsilon)
\rightharpoonup (\partial_t u, \nabla u) 
\ \ {\rm{in}}\ \ L^2(\R^n\times (0,\infty)),
\]
and there exists a nonnegative Radon measure
$\nu$ on $\R^n\times (0,\infty)$ such that
\[
e_\eps(u_\eps)\,dxdt
\rightharpoonup \mu=\mu_t\d t:=(|\nabla u|^2\,dx+\nu_t)\d t
\]
as weak convergence of Radon measures, where $\nu=\nu_t \,dt$, with $\{\nu_t\}_{t>0}$ a family of nonnegative Radon measures on $\R^n$.  

\begin{lemma}\label{sub_caloric0} For any $Q\in X$, the following
\begin{equation}\label{sub_caloric}
\partial_t d^2(u, Q)-\Delta d^2(u, Q)
\le -2 \mu_t\,dt
\end{equation}
holds in $\mathscr{D}'(\R^n\times (0,\infty))$.
\end{lemma}

\begin{proof} \eqref{sub_caloric} follows directly from 
 \eqref{subharmonic1} of Lemma \ref{sub-harmonicity}, by sending $\eps\to 0$.
\end{proof}

\medskip
\noindent For $z_0=(x_0, t_0)\in\R^n\times (0,\infty)$, let $G_{z_0}$ be the backward heat kernel on $\R^n$ with center $z_0$, given by
\[
G_{z_0}(x,t)=(4\pi |t_0-t|)^{-\frac{n}2} \exp\big(-\frac{|x-x_0|^2}{4(t_0-t)}\big), \ \ \ x\in\R^n, \ t<t_0.
\]
Then we have
\begin{lemma}\label{struwe_mono} For any $z_0=(x_0,t_0)\in \R^n\times (0,\infty)$ and $0<R_1<R_2<\sqrt{t_0}$, it holds
\begin{align}\label{struwe_mono0}
2 R_2^2\int_{t=t_0-R_2^2}\mu_t G_{z_0}-2 R_1^2\int_{t=t_0-R_1^2}\mu_t G_{z_0}
\ge\iint_{\R^n\times [t_0-R_2^2, t_0-R_1^2]}\frac{1}{|t|}
\big|x\cdot\nabla u+2t \partial_tu\big|^2 G_{z_0}.
\end{align}
\end{lemma} 
\begin{proof} \eqref{struwe_mono0} follows directly from \eqref{stationarity06} after sending $\eps\to 0$.
\end{proof}

\medskip
Now, for $Q\in X$, $z_0\in \R^n\times (0,\infty)$, and  $0<R<\sqrt{t_0}$, we define
\[
E(u; z_0, R)=2R^2\int_{t=t_0-R^2} G_{z_0}(x,t)\,d\mu_t(x),
\]
\[
H(u; z_0, R, Q)=\int_{t=t_0-R^2} d^2(u,Q)G_{z_0}\,dx.
\]

\smallskip
Without loss of generality, we may assume that for any
$z_0=(x_0, t_0)\in\R^n\times (0,\infty)$,
\begin{equation}\label{non-constant}
H(u; z_0, R, Q)>0,\ \forall 0<R<\sqrt{t_0}.
\end{equation}
For, otherwise, there exist $z_*=(x_*, t_*)
\in \R^n\times (0,\infty)$ and $R_*\in (0,\sqrt{t_*})$
such that
\[
d^2(u(x,t), Q)=0, \ \forall (x,t)\in \R^n\times 
\big\{t_*-R_*^2\big\}.  
\]
Since $d^2(u, Q)$ satisfies \eqref{sub_caloric}, it follows from the maximal principle that
\[
d^2(u(x,t), Q)=0, \ \forall (x,t)\in\R^n\times [t_*-R_*^2,\infty).
\]
That is, $u\equiv Q$ in $\R^n\times (t_*-R_*^2,\infty)$.  In particular, $u$ is Lipschitz continuous 
in $\R^n\times\{t_0\}$.

\medskip
With the help of \eqref{non-constant}, we can, as in \cite{Poon}, define the parabolic frequency function
\[
N(u; z_0, R, Q)=\frac{E(u; z_0, R)}{H(u; z_0, R, Q)}, \ \forall 0<R<\sqrt{t_0}.
\]
Then we have

\begin{theorem} \label{frequency_mono} For any $z_0=(x_0,t_0)\in \R^n\times (0,\infty)$ and $Q\in X$, 
$N(u; z_0, R, Q)$ is monotonically increasing for $0<R<\sqrt{t_0}$. In particular,
\begin{equation}\label{frequency_mono1}
N(u; z_0, Q)=\lim_{R\to 0^+} N(u; z_0, R, Q)
\end{equation}
exists and is upper semi-continuous in $z_0\in\R^n\times (0,\infty)$.
\end{theorem}

\begin{proof} First, we observe that for all $t>0$, \eqref{sub_caloric} also holds in $\R^n$. 
Next we can test \eqref{sub_caloric} at $t=t_0-R^2$ for a sequence of test functions
$\phi_k(x)=G_{z_0}(x,t_0-R^2)\eta^1_k(x)$,
where $\eta^1_k\in C^\infty_0(B_{2k})$ satisfies $0\le \eta_k^1\le 1$, $\eta^1_k=1$ in $B_k$,  and $|\nabla\eta_k^1|\le \frac{4}{k}$.
Then we obtain, by sending $k\to\infty$,
\begin{align}
 E(u; z_0, R)&\le R^2\int_{t=t_0-R^2} d^2(u,Q)(\Delta +\partial_t) G_{z_0}
 +\frac{R}{2}\frac{d}{dR}\int_{t=t_0-R^2}d^2(u,Q)G_{z_0}\,dx\nonumber\\
 &=\frac{R}{2}\frac{d}{dR}\int_{t=t_0-R^2} d^2(u,Q)G_{z_0}(x, t)\,dx\nonumber\\
 &=\frac{R}2 \frac{d}{dR}H(u; z_0, R, Q),\label{subharmonic2}
\end{align}
where we have used 
the fact that $\Delta G_{z_0}+\partial_t G_{z_0}=0$
in $\R^n\times \{t_0-R^2\}$.

Now we calculate 
\begin{align*}
\frac{d}{dR}\log N(u; z_0, R,Q)&=\frac{d}{dR}\log E(u; z_0, R)-\frac{d}{dR}\log H(u; z_0, R,Q)\\
&=\frac{\frac{d}{dR}E(u; z_0, R)}{E(u; z_0,R)}-\frac{\frac{d}{dR}H(u; z_0, R, Q)}{H(u; z_0,R, Q)}.
\end{align*}

Define $u_R(y,t)=u(x_0+Ry, t_0+R^2 t)$. Then we have
\begin{align*}
\frac{d}{dR}H(u; z_0, R, Q)&=\frac{d}{dR}\int_{t=t_0-R^2} d^2(u, Q) G_{z_0}(x,t)\,dx\\
&=\frac{d}{dR}\int_{t=-1} d^2(u_R, Q) G(y,t)\,dy\\
&=2\int_{t=-1} d(u_R, Q) \frac{d}{dR}d(u_R,Q) G(y,t)\,dy
\end{align*}
so that by H\"older's inequality, we have
\begin{align*}
\big(\frac{d}{dR}H(u; z_0,R, Q)\big)^2&=4\Big(\int_{t=-1} d(u_R, Q) \frac{d}{dR}d(u_R,Q) G(y,t)\,dy\Big)^2\\
&\le 4\int_{t=-1} d^2(u_R, Q) G(y,t)\,dy
\int_{t=-1} \big(\frac{d}{dR}d(u_R, Q)\big)^2 G(y,t)\,dy\\
&=4 H(u; z_0, R, Q) \int_{t=-1} \big(\frac{d}{dR}d(u_R, Q)\big)^2 G(y,t)\,dy.
\end{align*}
Observe that since $d(y, Q): X\to \R$ is Lipschitz continuous, with Lipschitz norm at most $1$, it follows
that
\[
\big|\frac{d}{dR} d(u_R, Q)\big|
\le \big|\frac{d}{dR} u_R\big|.
\]
Hence, we arrive at
\begin{align*}
\big(\frac{d}{dR}H(u; z_0,R, Q)\big)^2\le 
4 H(u; z_0, R, Q) \int_{t=-1} \big|\frac{d}{dR}u_R\big|^2 G(y,t)\,dy.
\end{align*}
On the other hand, it follows from \eqref{struwe_mono0} that
\begin{align}\label{struwe_mono1}
\frac{d}{dR}E(u; z_0, R)&\ge \frac{2}{R}\int_{t=t_0-R^2} \big|(y-x_0)\cdot\nabla u+2(t-t_0)\partial_tu\big|^2G_{z_0} \,dy\nonumber\\
&=2R\int_{t=-1} \big|\frac{d}{dR} u_R\big|^2G \,dy.
\end{align}
Hence, we have
\[
\big(\frac{d}{dR}H(u; z_0, R, Q)\big)^2
\le \frac{2}{R}H(u; z_0,R, Q) \frac{d}{dR}E(u; z_0, R),
\]
which, combined with \eqref{subharmonic}, implies
that 
\begin{align*}
&\frac{d}{dR}E(u; z_0,R)H(u; z_0,R, Q)-E(u; z_0,R)
\frac{d}{dR}H(u; z_0,R, Q)\\
&\ge \frac{d}{dR}E(u; z_0,R)H(u; z_0,R, Q)
-\frac{R}2 \big(\frac{d}{dR}H(u; z_0, R, Q)\big)^2\\
&\ge 0.
\end{align*}
Therefore, we obtain
\[
\frac{d}{dR}N(u; z_0, R, Q)
\ge \displaystyle\frac{\frac{d}{dR}E(u; z_0, R)H(u; z_0, R, Q)-E(u; z_0,R)
\frac{d}{dR}H(u; z_0,R, Q)}{H^2(u; z_0,R, Q)}\ge 0.
\]
This completes the proof of Theorem \ref{frequency_mono}.
\end{proof}

\section{Regularity of the limiting map $u$} 
\label{sec:Lipregularity}
In this section, we first prove Theorem 1.1, part (3), namely the H\"older continuity of the limiting map $u$ holds for any locally compact
CAT(0) space $(X,d)$.  Then we  show Theorem 1.2, that is, the spatial Lipschitz regularity of limiting map $u$ 
when $(M,g)=(\R^n, dx^2)$.

\subsection{H\"older continuity} 


In this subsection, we prove  Theorem 1.1, part (3).  The argument is based on an application of Harnack's inequality
to \eqref{sub_caloric} from Lemma \ref{sub_caloric0}. We refer the reader to Lin \cite{Lin1996} (see also Jost \cite [Chapter 9] {Jost})
for the argument on harmonic maps into  locally compact CAT(0) spaces.

For any $z_0=(x_0,t_0)\in M\times (0,\infty)$ and $0<r<\min\big\{i_M, \sqrt{t_0}\big\}$, set $P_r(z_0)=B_r(x_0)\times (t_0-r^2, t_0)\subset M\times (0,\infty)$.
It suffices to show that there exists $0<\alpha<1$ such that 
the oscillation function of $u$ on $P_r(z_0)$,  $\omega(z_0,r)={\rm{diam}} (u(P_r(z_0))\leq C r^{\alpha}$ for all small $r>0$. 

For simplicity, we will assume $(M, g)=(\R^n, dx^2)$ so that $i_M=\infty$.  Set $r_0=\frac{\sqrt{t_0}}2$ and $\lambda_0=\frac{\omega(z_0, r_0)}2>0$, and define
$v(x,t)=u(x_0+r_0 x, t_0+r_0^2 t): P_1(0,0)\to (X, d_{\lambda_0})$,
where $d_{\lambda_0}(P, Q)={\lambda_0}^{-1}d(P,Q)$ for
$P,Q\in X$. It is easy to check 
that $(X, d_{\lambda_0})$ is also  a CAT(0)-space, $v$ has ${\rm{diam}}_{d_{\lambda_0}} (v(P_1(0,0))=2$, and 
\eqref{sub_caloric} implies that 
\begin{equation}\label{sub_caloric01}
\partial_t d_{\lambda_0}^2(v, Q)-\Delta d_{\lambda_0}^2(v, Q)\le 0 \ \ {\rm{on}} \ \ P_1(0,0).
\end{equation} 
From this reduction, we may assume that $u:P_1(0,0)\to (X,d)$ has ${\rm{diam}}(u(P_1(0,0)))=2$ and then show that there is $\alpha\in (0,1)$ such that 
$\omega((0,0), r) \le C r^{\alpha}$ for $0<r\le \frac12$. For this purpose, we first need to establish the following claim.

\medskip
\noindent{\it Claim}: There exist $0<\eps_0<\frac{1}{32}$ such that if $1<{\rm{diam}}(u(P_1(0,0)))\le 2$ and $u(P_1(0,0))$ is covered by $k$ balls $B(u(z_i), \eps)\subset X$, 
which have center $u(z_i)$ for some $z_i\in P_1(0,0)$ and radius $\eps\in (0, \eps_0]$,  for $i=1,\cdots, k$. Then $u(P_{\frac12}(0,0))$ can be covered by at most
$k-1$ balls among $B(u(z_1),\eps),\cdots, B(u(z_k),\eps)$. 

From ${\rm{diam}}(u(P_1(0,0)))\le 2$, $u(P_1(0,0))$ is contained in a ball $\overline{B}\subset X$ of radius $2$. It follows from the local compactness
of $X$ that $\overline{B}$ contains at most $k_1$ points whose mutual distance is at least $\frac{1}{16}$. Hence $k_1$ of
the balls $B(u(z_i), \frac1{8})$ cover $u(P_1)$, say $i=1,\cdots, k_1$.
This implies 
$$P_\frac12(0,0)\subset\bigcup_{i=1}^{k_1} u^{-1}(B(u(z_i), \frac18)).$$
Hence, there is at least one $i$, say $i=1$, such that
$$
\Big|u^{-1}(B(u(z_1), \frac18))\cap P_\frac12(0,0)\Big|\ge \delta_0:=\frac{|P_\frac12(0,0)|}{k_1}>0.
$$

Set $q_1=u(z_1)$ and define 
$$f(z)=\sup_{w\in P_1(0,0)} d^2(u(w), q_1)-d^2(u(z), q_1), \ z\in P_1(0,0).$$
Then we have $f\ge 0$ on $P_1$, and by \eqref{sub_caloric01} that $f$ is a weak solution of 
\[
\partial_t f-\Delta f\ge 0, \ \ {\rm{in}}\ \ P_1(0,0). 
\]
It follows from the Harnack inequality for non-negative sub-caloric
functions that 
\begin{align}\label{lower-bound-of-f}
\inf_{P_\frac12(0,0)} f&\ge C(n)\int_{P_\frac12(0,0)} f(z)\,dz\nonumber\\
&\ge C(n) \int_{u^{-1}(B(q_1, \frac18)\cap P_\frac12(0,0)} f(z)\,dz\nonumber\\
&\ge \frac{15C(n)}{64}\Big|u^{-1}(B(q_1, \frac18)\cap P_\frac12(0,0)\Big|\ge c(\delta_0),
\end{align}
where we have used 
\[
\inf_{z\in  u^{-1}(B(q_1, \frac18))\cap P_\frac12(0,0)} f(z)\ge\frac{15}{64},
\]
which follows from
$$\sup_{w\in P_1(0,0)} d(u(w), q_1)\ge \frac12
\ ({\mbox{thanks to diam}}(u(P_1(0,0)))\ge 1),
$$
and
\[
d(u(z), q_1)\le \frac18, \ \forall z\in u^{-1}(B(q_1, \frac18))\cap P_\frac12(0,0).
\]

\medskip
Now we want to show that \eqref{lower-bound-of-f} implies that for a sufficiently small
$\eps_0$, it is impossible that
\begin{equation}\label{no-missing-balls}
u(P_\frac12(0,0))\cap B(u(z_i), 2\eps)\not=\emptyset, \forall i=1,\cdots, k
\end{equation}
Suppose \eqref{no-missing-balls} were true. First, there exists  $z_*\in P_1(0,0)$ 
so that
\[
d(u(z_*), q_1)=\sup_{z\in P_1(0,0)}d(u(z), q_1).
\]
Second, since $u(P_1(0,0))\subset\bigcup_{i=1}^k B(u(z_i), 2\eps)$, there exists 
$1\le i\le k$ such that
\[
d(u(z_*), u(z_i))\le 2\eps.
\]
This, combined with \eqref{no-missing-balls}, implies that there exists
$w_*\in P_\frac12(0,0)$ such that $d(u(w_*), u(z_i))\le 2\eps$. Hence, we have
\[
d(u(w_*), u(z_*)\le d(u(w_*), u(z_i))+d(u(z_*), u(z_i)\le 4\eps\le 4\eps_0.
\]
Thus,
\begin{align*}
\inf_{P_\frac12(0,0)} f(z)\le f(w_*)
&=\sup_{w\in P_1(0,0)} d^2(u(w), q_1)-d^2(u(w_*), q_1)\\
&=d^2(u(z_*), q_1)-d^2(u(w_*), q_1)\\
&=(d(u(z_*), q_1)+d(u(w_*), q_1))(d(u(z_*), q_1)-d(u(w_*), q_1))\\
&\le 2{\rm{diam}}(u(P_1(0,0)))d(u(z_*), u(w_*))\le 16\eps_0.
\end{align*}
This contradicts \eqref{lower-bound-of-f}, provided we choose
$\eps_0<\frac{\delta_0}{32}$.

Finally, since ${\rm{diam}}(u(P_1(0,0)))=2$, there exists an integer $M_0=M_0(X)\ge 1$ such that 
$u(P_1(0,0))$ can be covered by $M_0$ balls, $\{B_i^X\}_{i=1}^{M_0}$, of radius $\eps_0$. 
Suppose ${\rm{diam}}(u(P_\frac12(0,0)))>1$. Then we can repeat the same arguments
to $v(x,t)=u(\frac{x}2, \frac{t}{4}), \ (x,t)\in P_1(0,0)$ to conclude that
$u(P_\frac14(0,0))$ is covered by at most $M_0-2$ balls. Since this procedure
can continue at most $M_0$ steps, we must have
$${\rm{diam}}(u(P_{2^{-M_0}}(0,0)))\le 1:=\frac12 {\rm{diam}}(u(P_1(0,0))).$$
Iterating this inequality finitely many times yields that there exists $\alpha\in (0,1)$ such that
$${\rm{\diam}}(u(P_r(0,0)))\le C r^\alpha, \ \forall 0<r<\frac12.$$
Hence H\"older continuity of $u$ follows.
\qed

\subsection{Lower bound of frequency function} 
{First,  we recall the metric differentiability by Kirchheim \cite{Kirchheim1994}.  

\begin{definition} Given a measurable map $f: \R^n\to (X,d)$ and $x_0\in\R^n$, we say that $f$ is metric differentiable at $x_0$, if there exists a {\it seminorm}
${\rm md}_{x_0}(f): \R^n\to \R_+$, called metric differential of $f$ at $x_0$, such that  
\begin{equation}
d(f(x), f(x_0))= {\rm md}_{x_0}(f)(x-x_0)+o(|x-x_0|),
\end{equation}
or, equivalently,
\begin{equation}
\lim_{x\to x_0}\frac{\big|d(f(x), f(x_0))-{\rm md}_{x_0}(f)(x-x_0)\big|}{\big|x-x_0\big|}=0.
\end{equation}
\end{definition}

\smallskip
Next, we recall the metric differentiability theorem by \cite{Kirchheim1994}, which says that for any Lipschitz map $f:\R^n\to (X,d)$, that is,  there exists $C>0$
such that $d(f(x), f(y))\le C|x-y|$ for $x, y\in\R^n$, then $f$ is metric differentiable for almost every $x\in \R^n$. Moreover, the map $x\mapsto {\rm{md}}_x(f)$ is a
Borel map.

Finally, we recall a special case of the approximate metric differentiability result by Gigli-Tyulenev \cite{GigliTyulenev} for a $H^1$-map from $\R^n$ to $(X,d)$. 

\begin{theorem} \label{app-diff1}For any complete metric space $(X,d)$, If $f\in H^1_{\rm{loc}}(\R^n, X)$ is a locally $H^1$-map, 
then $f$ is approximately metric differentiable for almost every $x\in \R^n$, that is, for a.e. $x\in \R^n$, it holds
\begin{equation}\label{app-diff2}
\lim_{r\to 0} \frac{\Big|\big\{y\in B_r(x):  |d(f(y), f(x))-{\rm{md}}_x(f)(y-x)\big|>\varepsilon |y-x|\big\}\Big|}{\omega_n r^n}=0, \ \forall \varepsilon>0,
\end{equation}
where $\omega_n=|B_1(0)|$ is the volume of unit ball in $\R^n$.
\end{theorem}
\begin{proof} See \cite[Proposition 3.6]{GigliTyulenev} for the proof of general cases where the domain is allowed to be any strongly rectifiable space, which,  in particular, 
includes $\R^n$. Here one uses the fact that any $H^1$-map $f:\R^n\to X$ enjoys the Lusin-Lipschitz property, that is, off a set of small measure, $f$ equals to a Lipschitz
map. This enables us to apply the a.e. metric differentiability theorem of Kirchheim on Lipschitz maps to deduce the approximate metric differentiability for $f$.
\end{proof}

\begin{theorem} \label{lowerboundfreq} For $u_0\in H^1(\R^n, X)$,  let $u$ be given by Theorem \ref{th:main_intro_0} and $X$ be a CAT$(0)$ space. 
Then, for a.e. $z_0=(x_0,t_0)\in \R^n\times (0,\infty)$ it holds that
\begin{equation}\label{lowerboundfreqquency}
N(u; z_0, u(z_0))=\lim_{R\to 0} {N}(u; z_0, R, u(z_0))\ge 1.
\end{equation}
\end{theorem}

\begin{proof}  First, since $u\in L^2_{\rm{loc}}(\R^n\times (0,\infty), X)$ satisfies 
$\int_0^T\int_{\R^n} |\partial_t u|^2\,dxdt<\infty$ for $0<T<\infty$ and $\sup_{t>0}\int_{\R^n}|\nabla u(x,t)|^2\,dx<\infty$, we must have that by the differentiation theorem, 
for a.e. $x_0\in\R^n$ and a.e. $t_0\in (0,\infty)$, 
\begin{equation}\label{app-diff3}
\lim_{\delta\to 0}  \fint_{P_\delta(z_0)} \big||\nabla u|^2-|\nabla u|^2(z_0)\big|\,dxdt=0
\ \ {\rm{and}}\ \ \lim_{\delta\to 0} \fint_{P_\delta(z_0)} \big||\partial_t u|^2-|\partial_t u|^2(z_0)\big|\,dxdt=0
\end{equation} 
\begin{equation}\label{app-diff30}
\lim_{\delta\to 0}  \fint_{P_\delta(z_0)} \big(d^2(u(x,t), u(x_0,t))+d^2(u(x_0, t), u(x_0,t_0))\Big)\,dxdt=0,
\end{equation}
and
\begin{equation}\label{app-diff31}
 \lim_{\delta\to 0} \fint_{P_\delta(z_0)} \big||\partial_t u(x_0,t)|^2-|\partial_t u|^2(z_0)\big|\,dxdt=0,
\end{equation} 
hold  for $z_0=(x_0,t_0)$.  According to \cite{KS}, it holds that 
\begin{equation}\label{app-diff4}
|\nabla u(x_0, t_0)|^2=(n+2)\lim_{\delta\to 0} \delta^{-2}\fint_{B_\delta(x_0)} d^2(u(x,t_0), u(z_0))\,dx.
\end{equation}
Moreover, by Theorem \ref{app-diff1},  we assume that $u(\cdot, t_0)$ is approximately metric differentiable at $x_0$, that is, 
\begin{equation}\label{app-diff4}
\lim_{\delta\to 0} \fint_{B_\delta(x_0)} \Big|\frac{d^2(u(x,t_0), u(z_0))}{\delta^2}-\frac{1}{n+2}\fint_{\mathbb{S}^{n-1}}{\rm{md}}_{x_0}^2(u(\cdot, t_0))(\theta)\,d\mathcal{H}^{n-1}(\theta)\Big|\,dx=0,
\end{equation}
and
\begin{equation}\label{app-diff40}
\lim_{\delta\to 0} \fint_{B_\delta(x_0)} \Big|\frac{{\rm{md}}_{x_0}^2(u(\cdot, t_0))(x-x_0)}{\delta^2}-\frac{1}{n+2}\fint_{\mathbb{S}^{n-1}}{\rm{md}}_{x_0}^2(u(\cdot, t_0))(\theta)\,d\mathcal{H}^{n-1}(\theta)\Big|\,dx=0,
\end{equation}

Now we divide the argument into two cases: 

\noindent {\bf Case 1}.  The  energy density of $u$ at $z_0$, $|\nabla u|^2(z_0)=|A_0|^2> 0$.  We will proceed this case in three steps.

\medskip
\noindent {\it Step 1}. We claim that \eqref{app-diff3} and \eqref{app-diff4} imply
\begin{align}\label{app-diff5}
\lim_{\delta\to 0}\sup_{t_0-\delta^2\le t\le t_0}\fint_{B_\delta(x_0)}
\Big|\frac{d^2(u(x,t), u(x_0,t))}{\delta^2}-\frac{1}{n+2}\fint_{\mathbb{S}^{n-1}}{\rm{md}}^2_{x_0}(u(\cdot, t_0))(\theta)\,d\mathcal{H}^{n-1}(\theta)\Big|\,dx=0.
\end{align}
To see \eqref{app-diff5}, let $t_0-\delta^2\le t_1\le t_0$. Then
\begin{align*}
&\Big|\delta^{-(n+2)}\int_{B_\delta(x_0)}
d^2(u(x,t_1), u(x_0, t_1))\,dx
-\delta^{-(n+2)}\int_{B_\delta(x_0)}
d^2(u(x,t_0), u(x_0, t_0))\,dx\Big|\\
&\le 2\delta^{-(n+2)}
\int_{t_1}^{t_0}\int_{B_\delta(x_0)}d(u(x,t), u(x_0, t_0))\big|\partial_td(u(x,t), u(x_0, t))\big| \,dxdt\\
&\le 2\delta^{-(n+2)}
\int_{t_1}^{t_0}\int_{B_\delta(x_0)}d(u(x,t), u(x_0, t))\big(|\partial_tu|(x,t)+| \partial_tu|(x_0, t)\big) \,dxdt\\
&\le C\Big(\fint_{P_\delta(z_0)}d^2(u(x,t), u(x_0, t))\,dxdt\Big)^\frac12 \Big(\fint_{P_\delta(z_0)}\big(|\partial_t u|^2(x,t)+|\partial_t u|^2(x_0,t)\big)\,dxdt\Big)^\frac12.
\end{align*}
Taking supremum of $t_1\in [t_0-\delta^2, t_0]$, we obtain
\begin{align*}
&\Big|\sup_{t_1\in [t_0-\delta^2, t_0]}\delta^{-2}\fint_{B_\delta(x_0)}
d^2(u(x,t_1), u(x_0, t_1))\,dx
-\delta^{-2}\fint_{B_\delta(x_0)}
d^2(u(x,t_0), u(x_0, t_0))\,dx\Big|\\
&\le 
C\Big(\fint_{P_\delta(z_0)}d^2(u(x,t), u(x_0, t))\,dxdt\Big)^\frac12 \Big(\fint_{P_\delta(z_0)}\big(|\partial_t u|^2(x,t)+|\partial_t u|^2(x_0,t)\big)\,dxdt\Big)^\frac12.
\end{align*}
After sending $\delta\to 0$ and applying \eqref{app-diff30} and \eqref{app-diff4}, this yields \eqref{app-diff5}.

\medskip
On the other hand, direct calculations imply 
\begin{align}\label{app-diff50}
&\delta^{-(n+2)}\sup_{t_0-\delta^2\le t\le t_0}\int_{B_\delta(x_0)} d^2(u(x_0,t), u(x_0,t_0))\,dx\nonumber\\
&\le 2\fint_{P_\delta(z_0)} d(u(x_0,t), u(x_0,t_0))|\partial_t u|(x_0, t)\,dxdt\nonumber\\
&\le 2\Big(\fint_{P_\delta(z_0)} d^2(u(x_0,t), u(x_0,t_0))\,dxdt\Big)^\frac12\Big(\fint_{P_\delta(z_0)} |\partial_t u|^2(x_0, t)\,dxdt\Big)^\frac12\nonumber\\
&\to 0 
\end{align} 
as $\delta\to 0$, where we have used \eqref{app-diff30} and \eqref{app-diff31} in the last step.

\medskip
Combining \eqref{app-diff5} with \eqref{app-diff51}, we obtain that
\begin{align}\label{app-diff51}
\lim_{\delta\to 0}\sup_{t_0-\delta^2\le t\le t_0}\fint_{B_\delta(x_0)}
\Big|\frac{d^2(u(x,t), u(x_0,t_0))}{\delta^2}-\frac{1}{n+2}\fint_{\mathbb{S}^{n-1}}{\rm{md}}^2_{x_0}(u(\cdot, t_0))(\theta)\,d\mathcal{H}^{n-1}(\theta)\Big|\,dx=0.
\end{align}

\medskip
\noindent{\it Step 2}.  In order to show $N(u;z_0, u(z_0))\ge 1$, we proceed as follows. From \eqref{app-diff51} and \eqref{app-diff4}-\eqref{app-diff40} , there exists $\delta_0>0$
such that for any $0<R,\delta<\delta_0$, it holds that 
\begin{equation}\label{app-diff52}
\fint_{B_\delta(x_0)}
\Big|\frac{d^2(u(x,t), u(x_0,t_0))}{\delta^2}-\frac{1}{n+2}\fint_{\mathbb{S}^{n-1}}{\rm{md}}^2_{x_0}(u(\cdot, t_0))(\theta)\,d\mathcal{H}^{n-1}(\theta)\Big|\,dx=o_{\delta_0}(1),
\end{equation}
and
\begin{equation}\label{app-diff53}
\delta^{-2}\fint_{B_\delta(x_0)}
\Big|d^2(u(x,t_0-R^2), u(x_0,t_0))-{\rm{md}}_{x_0}^2(u(\cdot, t_0-R^2))(x-x_0)\Big|\,dx=o_{\delta_0}(1),
\end{equation}

where $\displaystyle\lim_{\delta_0\to 0} o_{\delta_0}(1)=0$.

For $0<R<\delta_0$, we decompose
\begin{align}\label{dist_compare}
H(u; z_0, R, u(z_0))
&=\Big\{\int_{B_{\delta_0}(x_0)}
+\int_{\R^n\setminus B_{\delta_0}(x_0)}\Big\}
d^2\big(u(x,t_0-R^2), u(z_0)\big) G_{z_0}(x,t_0-R^2)\,dx\nonumber\\
&=I+II.
\end{align}
It follows from the assumption $d(u_0, Q)\in L^\infty(\R^n)$ and \eqref{sub_caloric} that
\begin{equation}\label{bound}
\big\|d(u(\cdot,t), Q)\big\|_{L^\infty(\R^n)}\le\big\|d(u_0, Q)\big\|_{L^\infty(\R^n)}, \forall t\ge 0,
\end{equation}
hence
\begin{align*}
\big|II\big|&\le \int_{\R^n\setminus B_{\delta_0}(x_0)} d^2(u(x, t_0-R^2), u(z_0))
G_{z_0}(x,t_0-R^2)\,dx\\
&\le 2 \int_{\R^n\setminus B_{\delta_0}(x_0)} \big(d^2(u(x, t_0-R^2), Q)+d^2(u(z_0), Q)\big)
G_{z_0}(x,t_0-R^2)\,dx\\
&\le 2\big\|d(u_0, Q)\big\|_{L^\infty(\R^n)}\frac{R^2}{\delta_0^2}\int_{\R^n\setminus B_{\frac{\delta_0}{R}}} |y|^2 e^{-\frac{|y|^2}4}\,dy\\
&\le C(u_0, \delta_0, Q)o_R(1)R^2.
\end{align*}
Here $o_R(1)$ is a quantity, satisfying $\displaystyle\lim_{R\to 0} o_R(1)=0$.

To estimate $I$, we first apply the triangle inequality to bound
\begin{align*}
I^\frac12&\le\Big(\int_{B_{\delta_0}(x_0)}\big({\rm{md}}_{x_0}(u(\cdot, t_0))(x-x_0)\big)^2G_{z_0}(x,t_0-R^2)\,dx\Big)^\frac12\\
&\ \ + \Big(\int_{B_{\delta_0}(x_0)}\Big|d^2\big(u(x,t_0-R^2), u(z_0)\big)-\big({\rm{md}}_{x_0}(u(\cdot, t_0))(x-x_0)\big)^2\Big| G_{z_0}(x,t_0-R^2)\,dx\Big)^\frac12\\
&=I_1+I_2.
\end{align*}

By direct calculations, we can bound
\begin{align*}
I_1^2&=\int_{B_{\delta_0}(x_0)}\big({\rm{md}}_{x_0}(u(\cdot, t_0))(x-x_0)\big)^2G_{z_0}(x,t_0-R^2)\,dx\\
&\le \int_{\R^n}\big({\rm{md}}_{x_0}(u(\cdot, t_0))(x-x_0)\big)^2G_{z_0}(x,t_0-R^2)\,dx\\
&=\Big(R^2\int_0^\infty r^{n+1} \frac{1}{(4\pi)^{\frac{n}2}}e^{-\frac{r^2}4}\,dr\Big)\int_{\mathbb{S}^{n-1}} \big({\rm{md}}_{x_0}(u(\cdot, t_0))(\theta)\big)^2\,d\mathcal{H}^{n-1}(\theta)\\
&=2R^2\fint_{\mathbb{S}^{n-1}}\big({\rm{md}}_{x_0}(u(\cdot, t_0))(\theta)\big)^2\,d\mathcal{H}^{n-1}(\theta).
\end{align*}

Now we want to estimate $I_2$. For this, denote the integer part of $\displaystyle\frac{\ln(\frac{\delta_0}{R})}{\ln 2}$ by 
$k_0=\displaystyle\Big[\frac{\ln(\frac{\delta_0}{R})}{\ln 2}\Big]$. Set $ A_m:=B_{2^m R}(x_0)\setminus B_{2^{m-1}R}(x_0)$ for $1\le m\le k_0+1$.
Then we can bound $I_2$ by
\begin{align*}
{I_2}^2&\le \Big\{\int_{B_{R}(x_0)}+\sum_{m=1}^{k_0+1}
\int_{A_m}\Big\}
\Big|d^2(u(x,t_0-R^2), u(x_0,t_0))-({\rm{md}}_{x_0}(u(\cdot, t_0))(x-x_0))^2\Big|G_{z_0}(x,t_0-R^2)\,dx\\
&\le R^{-n}\int_{B_R(x_0)}\Big|d^2(u(x,t_0-R^2), u(x_0,t_0))-({\rm{md}}_{x_0}(u(\cdot, t_0))(x-x_0))^2\Big|\,dx\\
&+\sum_{m=1}^{k_0+1}e^{-4^{m-2}}
R^{-n}\int_{A_m}
\Big|d^2(u(x,t_0-R^2), u(x_0,t_0))-({\rm{md}}_{x_0}(u(\cdot, t_0))(x-x_0))^2\Big|\,dx\\
&\le R^2\Big(R^{-2}\fint_{B_R(x_0)}\Big|d^2(u(x,t_0-R^2), u(x_0,t_0))-({\rm{md}}_{x_0}(u(\cdot, t_0))(x-x_0))^2\Big|\,dx\\
&+\sum_{m=1}^{k_0+1}2^{m(n+2)}e^{-4^{m-2}}(2^m R^2)^{-2}
\fint_{B_{2^m R}(x_0)}
\Big|d^2(u(x,t_0-R^2), u(x_0,t_0))-({\rm{md}}_{x_0}(u(\cdot, t_0))(x-x_0))^2\Big|\,dx\Big)\\
&\le o_{\delta_0}(1)R^2 \Big(1+\sum_{m=1}^{k_0+1} 2^{m(n+2)} e^{-4^{m-2}}\Big)=o_{\delta_0}(1)R^2.
\end{align*}
Thus, we obtain
\begin{align}\label{H-upbound}
{H}(u;z_0, R, u(z_0))
\le 2R^2\fint_{\mathbb{S}^{n-1}}\big({\rm{md}}_{x_0}(u(\cdot, t_0))(\theta)\big)^2\,d\mathcal{H}^{n-1}(\theta)+\big(o_R(1)+o_{\delta_0}(1)\big) R^2, 
\ \forall 0<R<\delta_0.
\end{align}

\medskip
\noindent{\it Step 3}. For $0<R<\min\big\{\frac{\sqrt{t_0}}{2}, \delta_0\big\}$,
define
\begin{align*}
\widehat{E}(u; z_0, R) 
=2\int^{t_0-\frac{R^2}4}_{t_0-R^2}\int_{\R^n}
|\nabla u|^2 (x,t)G_{z_0}(x,t)\,dxdt.
\end{align*}
By the triangle inequality, we have 
\begin{align*}
\Big(\widehat{E}(u;z_0, R)\Big)^\frac12
&\ge \Big(2\int^{t_0-\frac{R^2}4}_{t_0-R^2}\int_{B_{\delta_0}(x_0)}
|\nabla u|^2 (z_0)G_{z_0}(x,t)\,dxdt\Big)^\frac12\\
&\quad-\Big(2\int^{t_0-\frac{R^2}4}_{t_0-R^2}\int_{B_{\delta_0}(x_0)}
\big||\nabla u|^2 (x,t)-|\nabla u|^2(z_0)\big|G_{z_0}(x,t)\,dxdt\Big)^\frac12\\
&=C-D.
\end{align*}
We can bound $C$ from below by 
\begin{align*}
C^2&=2\int^{t_0-\frac{R^2}4}_{t_0-R^2}\int_{\R^n}
|\nabla u|^2(z_0)G_{z_0}(x,t)\,dxdt-2|\nabla u|^2(z_0)
\int^{t_0-\frac{R^2}4}_{t_0-R^2}\int_{\R^n\setminus B_{\delta_0}(x_0)}G_{z_0}(x,t)\,dxdt \\
&=2\int^{t_0-\frac{R^2}4}_{t_0-R^2}\int_{\R^n}
|\nabla u|^2(z_0)G_{z_0}(x,t)\,dxdt-2|\nabla u|^2(z_0)\int^{R^2}_{\frac{3R^2}4}
\frac{1}{(4\pi t)^{\frac{n}2}}\int_{\R^n\setminus B_{\delta_0}(0)}e^{-\frac{|x|^2}{4t}}\,dxdt\\
&\ge 2\int^{t_0-\frac{R^2}4}_{t_0-R^2}\int_{\R^n}
|\nabla u|^2(z_0)G_{z_0}(x,t)\,dxdt-C\frac{R^4}{\delta_0^2}\int_{\{|y|\ge \frac{\delta_0}{R}\}}
|y|^2e^{-\frac{|y|^2}4}\,dy\\
&\ge 2\int^{t_0-\frac{R^2}4}_{t_0-R^2}\int_{\R^n}
|\nabla u|^2(z_0)G_{z_0}(x,t)\,dxdt-C\frac{R^4}{\delta_0^2}.
\end{align*}
Applying \eqref{app-diff3}, we can bound $D$ from above by
\begin{align*}
D^2&\le 2\int^{t_0-\frac{R^2}4}_{t_0-R^2}\int_{B_{R}(x_0)}
\big||\nabla u|^2(x,t)-|\nabla u|^2(z_0)\big|G_{z_0}(x,t)\,dxdt\\
&\quad+2\sum_{m=1}^{k_0+1}
\int^{t_0-\frac{R^2}4}_{t_0-R^2}\int_{B_{2^mR}(x_0)\setminus B_{2^{m-1}R}(x_0)}
\big||\nabla u|^2(x,t)-|\nabla u|^2(z_0)\big|G_{z_0}(x,t)\,dxdt\\
&\le CR^{-n}\int^{t_0-\frac{R^2}4}_{t_0-R^2}\int_{B_{R}(x_0)}
\big||\nabla u|^2(x,t)-|\nabla u|^2(z_0)\big|\,dxdt\\
&\quad+CR^{-n}\sum_{m=1}^{k_0+1} e^{-4^{m-2}}
\int^{t_0-\frac{R^2}4}_{t_0-R^2}\int_{B_{2^mR}(x_0)}
\big||\nabla u|^2(x,t)-|\nabla u|^2(z_0)\big|\,dxdt\\
&\le CR^2\Big[R^{-(n+2)}\int_{P_R(z_0)}\big||\nabla u|^2(x,t)-|\nabla u|^2(z_0)\big|\,dxdt\\
&\qquad+\sum_{m=1}^{k_0+1}e^{-4^{m-2}}2^{m(n+2)} (2^mR)^{-(n+2)}\int_{P_{2^m R}(z_0)}\big||\nabla u|^2(x,t)-|\nabla u|^2(z_0)\big|\,dxdt\Big]\\
&\le C o_{\delta_0}(1) R^2
\Big(1+\sum_{m=1}^{k_0+1}e^{-4^{m-2}}2^{m(n+2)}\Big)\\
&=o_{\delta_0}(1) R^2.
\end{align*}
Thus, we obtain
\begin{align}
\widehat{E}(u; z_0, R)
\ge 2\int^{t_0-\frac{R^2}4}_{t_0-R^2}\int_{\R^n}
|\nabla u|^2(z_0)G_{z_0}(x,t)\,dxdt -C\delta_0^{-2}R^4+o_{\delta_0}(1)R^2.
\end{align}
By the mean value theorem, we conclude that
there exists $R_*\in (\frac{R}2, R)$ such that
\begin{align}\label{lower_bound_of E}
{E}(u; z_0, R_*)
&\ge 2R_*^2\int_{\R^n}
  |\nabla u|^2(z_0)G_{z_0}(x,t_0-R_*^2)\,dx+o_{\delta_0}(1)R_*^2-C\delta_0^{-2}R_*^4\nonumber\\
  &\ge 2|\nabla u|^2(z_0)R_*^2+o_{\delta_0}(1)R_*^2-C\delta_0^{-2}R_*^4.
\end{align}
Putting \eqref{H-upbound} and \eqref{lower_bound_of E} together, we see that for any sufficiently small $0<R<\delta_0$, there exists $R_*\in (\frac{R}2, R)$ such that
\begin{align*}
  \frac{E(u;z_0, R_*)}{{H}(u;z_0, R_*, u(z_0))}
  \ge \displaystyle\frac{|\nabla u(z_0)|^2}{\fint_{\mathbb{S}^{n-1}}\big({\rm{md}}_{x_0}(u(\cdot, t_0))(\theta)\big)^2\,d\mathcal{H}^{n-1}(\theta)}+(o_{\delta_0}(1)+o_{R_*}(1)).
\end{align*}

On the other hand, combining \eqref{app-diff4} with \eqref{app-diff40} yields
\begin{equation}\label{ks=md}
|\nabla u(z_0)|^2=\fint_{\mathbb{S}^{n-1}}\big({\rm{md}}_{x_0}(u(\cdot, t_0))(\theta)\big)^2\,d\mathcal{H}^{n-1}(\theta).
\end{equation}
Hence we obtain
\begin{align*}
  \frac{E(u;z_0, R_*)}{{H}(u;z_0, R_*, u(z_0))}
  \ge 1+(o_{\delta_0}(1)+o_{R_*}(1)).
\end{align*}
This, combined with Theorem \ref{frequency_mono}, implies $N(u;z_0, u(z_0))\ge 1$.
The completes the proof for {\bf Case 1}.

\bigskip
\noindent{\bf Case 2}. $A_0=|\nabla u(z_0)|=0$. We need to show Lemma \ref{lowerboundfreq} remains true in this case. 
%
To do it, we consider an augment variable argument. 
More precisely, let $\widehat{X}=X\times \R^n$, equipped with distance function $\widehat{d}$ given by
\[
\widehat{d}^2((p_1, q_1), (p_2, q_2))
=d^2(p_1, p_2)+|q_1-q_2|^2,
\ \forall (p_1, q_1), \ (p_2, q_2)\in \widehat{X}=X\times \R^n.
\]
For any $\delta>0$, define the map $u^\delta:
\R^n\times (0,\infty)\to \widehat{X}$ by 
\[
u^\delta(x,t)=\big(u(x,t), \delta x), \ (x,t)\in\R^n\times (0,\infty).
\]
Then it is easy to check that \\
1) The map $u^\delta$ is approximately differentiable at $z_0$ with the approximate spatial derivative
\begin{equation}\label{nonvanishing1}
\nabla u^\delta(z_0)=(\nabla u(z_0), \delta \mathbb{I}_n)
=(0, \delta\mathbb{I}_n)\not=0.
\end{equation}
2) From \eqref{sub_caloric}, it follows that $u^\delta$ is a sub caloric function, i.e.,  
\begin{align}\label{sub_caloric2}
(\partial_t -\Delta)\widehat{d}^2 
(u^\delta, u^\delta(z_0))
\le -2\mu_t^\delta \,dt, \ \mu_t^\delta=|\nabla u^\delta|^2\,dx+\nu_t=(|\nabla u|^2+n\delta)\,dx+\nu_t.
\end{align}
3) The map $u^\delta$ satisfies Struwe's monotonicity inequality similar to \eqref{struwe_mono0}, that is,
\begin{align}
2 R_2^2\int_{t=t_0-R_2^2}\mu_t^\delta G_{z_0}
-2 R_1^2\int_{t=-R_1^2}\mu_t^\delta G_{z_0}
\ge\iint_{\R^n\times [t_0-R_2^2, t_0-R_1^2]}\frac{1}{|t|}
\big|x\cdot\nabla u^\delta +2t \partial_tu^\delta\big|^2 G_{z_0}
\end{align}
holds for $0<R_1<R_2<\sqrt{t_0}$. In particular, for a.e.
$0<R<\sqrt{t_0}$ it holds 
\begin{align}\label{stationarity12}
\frac{d}{dR}E(u^\delta; z_0, R)
\ge \frac{2}{R}
\int_{\R^n\times \{t_0-R^2\}}
\big|x\cdot\nabla u^\delta +2t \partial_tu^\delta\big|^2 G_{z_0}\,dx,
\end{align}
where
\[
E(u^\delta; z_0, R)=
2 R^2\int_{\R^n\times \{t_0-R^2\}}\mu_t^\delta G_{z_0}.
\]

From \eqref{nonvanishing1}, \eqref{sub_caloric2}, and
\eqref{stationarity12},  we can apply the same
arguments as in section 4 to $u^\delta$ and conclude that
the frequency function of $u^\delta$ at $z_0$ given by
\begin{align}\label{modified_mono}
N(u^\delta; z_0, R, u^\delta(z_0))=\frac{E(u^\delta; z_0, R)}{H(u^\delta; z_0, R, u^\delta(z_0))}:=
\frac{\displaystyle 2 R^2\int_{\R^n\times \{t_0-R^2\}}\mu_t^\delta G_{z_0}\,dx}
{\displaystyle\int_{\R^n\times \{t_0-R^2\}}\widehat{d}^2(u^\delta(x,t), u^\delta(z_0)) G_{z_0}\,dx}
\end{align}
is monotonically non-decreasing with respect to $0<R<\sqrt{t_0}$,
and Lemma \ref{lowerboundfreq} implies that 
\begin{align}\label{lower_bound_frequency00}
N(u^\delta;z_0, u(z_0))=\lim_{R\to 0} N(u^\delta;z_0, R, u^\delta(z_0))\ge 1.    
\end{align}
This, after sending $\delta\to 0$, implies that
\begin{align}\label{mono_and_lowerbound}
N(u;z_0, R, u(z_0))\ge 1.
\end{align}
holds for all $0<R<\sqrt{t_0}$. In particular, we show
that 
\[
N(u;z_0, u(z_0))=\lim_{R\to 0} N(u;z_0, R, u(z_0))\ge 1.
\]
The proof is complete for {\bf Case 2}. 
\end{proof} 
}

\subsection{Proof Theorem \ref{Lip_reg}} It follows Theorem \ref{lowerboundfreq} and the upper semicontinuity of 
$N(u; z, u(z))$ that for all $z_0\in \R^n\times (0,\infty)$, 
\[
\alpha_0=N(u; z_0, u(z_0))\ge 1.
\]
Furthermore, it follows from
\[
E(u;z_0, R)\le \frac{R}{2}\frac{d}{dR}H(u;z_0,R, u(z_0)), \ \ 0<R<\sqrt{t_0}
\]
that
\begin{equation}\label{H-bound0}
\frac{d}{dR}\ln H(u;z_0, R, u(z_0))\ge\frac{2}{R} N(u; z_0, R, u(z_0))\ge \frac{2\alpha_0}{R},
\ \ 0<R<\sqrt{t_0}.
\end{equation}
Integrating this inequality from $R$ to $\sqrt{t_0}$, we obtain
\begin{equation}
\label{H-bound}
H(u;z_0, R, u(z_0))\le C_0 R^{2\alpha_0},\ \ 0<R<\sqrt{t_0},
\end{equation}
where
\[
C_0=\frac{H(u;z_0, \sqrt{t_0}, u(z_0))}{t_0^{\alpha_0}}=
t_0{^{-\alpha_0}}\int_{\R^n} d^2(u_0(x), u(z_0))G_{z_0}(x,0)\,dx
\le C(u_0) t_0^{-\alpha_0}.\]

From \eqref{H-bound}, we have that for any $z_0=(x_0,t_0)\in \R^n\times (0,\infty)$,
\begin{equation}\label{lip1}
R^{-n}\int_{B_{R}(x_0)\times \{t=t_0-R^2\}}
d^2(u(y,t), u(z_0))\,dy
\le e^{\frac14} C_0 R^{2\alpha_0}\le C_1R^2,
\end{equation}
holds for all $0<R\le \min\big\{1, \frac{\sqrt{t_0}}2$\big\}.

Now, let $x_1, x_2\in \R^n$ be two points
such that 
$$R_*:=|x_1-x_2|\le \min\big\{1, \displaystyle\frac{\sqrt{t_0}}4\big\}.$$
Set $z_1=(x_1, t_0), z_2=(x_2, t_0).$
Then, by \eqref{lip1},  we have
\begin{equation}\label{lip2}
R_*^{-n}\int_{B_{R_*}(x_i)\times \{t=t_0-R_*^2\}}
d^2(u(y,t), u(z_i))\,dy
\le  C_1R_*^2,
\ i=1,2.
\end{equation}
Set $\displaystyle x_*=\frac{x_1+x_2}2$. Applying \eqref{lip1} and the triangle inequality, we can bound
\begin{align*}
 &d^2(u(z_1), u(z_2))\\
 &\leq 32 \big(\frac{R_*}2\big)^{-n}\int_{B_{\frac{R_*}2}(x_*)}d^2(u(y, t_0-R_*^2), u(z_1))\,dy
 +32\big(\frac{R_*}2\big)^{-n}\int_{B_{\frac{R_*}2}(x_*)}d^2(u(y, t_0-R_*^2), u(z_2))\,dy\\
 &\leq 32 R_*^{-n}\int_{B_{{R_*}}(x_1)}d^2(u(y, t_0-R_*^2), u(z_1))\,dy
 +32R_*^{-n}\int_{B_{{R_*}}(x_2)}d^2(u(y, t_0-R_*^2), u(z_2))\,dy\\
 &\le C_0R_*^2\le C_0 |x_1-x_2|^2.
\end{align*}
This implies that $u(\cdot, t_0)$ is Lipschitz continuous
in $\R^n$. This, combined with $\partial_t u
\in L^2(\R^n\times (0,\infty))$, further implies that
$u$ is $\frac12$-H\"older continuous in $t>0$.     
\qed

\section{WED-approximation for smooth NPC target manifolds}
\label{sec:non_positive}

In this section, we particularize the previous abstract approach of WED approximation 
to the heat flow of harmonic maps into smooth NPC manifolds. We establish
uniform spatial Lipschitz estimates of minimizers $u_\eps$ of WED functionals $\mathcal{I}_\eps$,
through a perturbed version of Moser's Harnack inequality applied to the Bochner formula of $u_\eps$. 
In particular, we prove Theorem 1.3.

\subsection{Bochner formula for $u_\eps$}
First, we recall that for any $\eps>0$, $u_\eps$, which can be interpreted as a minimizing harmonic map from $(M\times (0,\infty),{g}_\eps)$
into $(N,h)\hookrightarrow\R^L$. Here we recall that the Riemannian metric ${g}_\eps$ is $e^{2\phi_\eps(t)}\left(\eps g + \d t^2 \right)$
with  $\phi_\eps(t)= \frac{1}{n-1}\left(-\frac{t}{\eps}-\frac{n}{2}\log\eps\right)$ .
Since $(N,h)$ is NPC, it is well-known (see \cite[Theorem 1.5.1]{bookLW} or \cite[Corollary 7.6.1]{Jost})
that $u_\eps\in C^\infty(M\times (0,\infty), N)$. 
The starting point for the uniform (in $\eps$) $C^0_{t}(C^1_x)$ estimate is the following Bochner formula.

\begin{theorem}[Bochner formula]
For $u_0\in H^1(M, N)$, let $u = u_\eps\in \mathfrak{V}_{u_0}$ be a minimizer of $I_\eps$. 
If $N$ has nonpositive sectional curvatures, then the energy density 
\[
e_\eps(u)\dpu \frac{\eps}{2}\abs{\partial_t u}^2 + \frac{1}{2}\abs{\nabla u}^2,
\]
verifies 
\begin{equation}
\label{eq:bochnerN}
\begin{split}
-\eps\partial_t^2 e_\eps(u) -\Delta e_\eps(u) + \partial_t e_\eps(u) & = -\abs{\bar{D}\d u}^2 + \eps \sum_{\alpha=1}^n {\rm{Rm}}^N(\partial_t u, \partial_\alpha u, \partial_\alpha u, \partial_t u)\\
& + \sum_{\alpha, \beta=1}^n {\rm{Rm}}^N(\partial_\alpha u, \partial_\beta u, \partial_\beta u, \partial_\alpha u)- \textup{Ricc}_g(\nabla u,\nabla u).
\end{split}
\end{equation}
Consequently, 
\begin{equation}
\label{eq:bochnerN_negative}
-\eps\partial_t^2 e_\eps(u) -\Delta e_\eps(u) + \partial_t e_\eps(u)\le C e_\eps(u), 
\end{equation}
where $C>0$ is such that 
\[
\textup{Ricc}_g \ge -C.
\]
\end{theorem}
\begin{proof}
Since the computation is local, we can consider, for any $(x_0,t_0)$ in $M\times (0,+\infty)$, the normal coordinates around $x_0$. 
We easily compute (see \cite[Theorem 1.5.1]{bookLW}) at the point $(x_0,t_0)$
\[
\begin{split}
\partial_t e_\eps(u) &= \eps \partial^2_t u\cdot \partial_t u + \nabla \partial_t u:\nabla u\\
\eps\partial_t^2 e_\eps(u) &= \eps\partial_t(\eps\partial^2_t u)\cdot\partial_t u +  \abs{\eps\partial_t^2 u}^2 + \eps\nabla \partial_t^2 u:\nabla u + \eps\abs{\partial_t \nabla u}^2 \\
\Delta e_\eps(u) &= \eps\partial_t \Delta u\cdot\partial_t u + \eps\abs{\nabla \partial_t u}^2 + \nabla\Delta u:\nabla u + \sum_{j,k=1}^n\abs{\frac{\partial^2 u}{\partial x_j\partial x_k}}^2 + \textrm{Ricc}_g(\nabla u,\nabla u).
\end{split}
\]
Therefore
\begin{equation}
\label{eq:bochner1}
\begin{split}
&-\eps\partial_t^2 e_\eps(u) -\Delta e_\eps(u) + \partial_t e_\eps(u)\\
&=\nabla \left(-\eps \partial_t^2 u -\Delta u +\partial_t u\right):\nabla u + \eps\partial_t(-\eps\partial_t^2 u -\Delta u +\partial_t u)\cdot\partial_t u\\
&\ \ \ -\abs{\eps\partial_t^2 u}^2 - 2\eps\abs{\partial_t \nabla u}^2  - \sum_{j,k=1}^n\abs{\frac{\partial^2 u}{\partial x_j\partial x_k}}^2-\textrm{Ricc}_g(\nabla u,\nabla u)\\
&=\nabla \left(-\eps \partial_t^2 u -\Delta u +\partial_t u\right):\nabla u + \eps\partial_t(-\eps\partial_t^2 u -\Delta u +\partial_t u)\cdot\partial_t u\\
&\ \ \ -\sum_{\alpha,\beta=0}^{n}\abs{\partial_{\alpha}^{\eps}
\partial_\beta^{\eps}u}^2_{\R^L}-\textrm{Ricc}_g(\nabla u,\nabla u),
\end{split}
\end{equation}
where
\[
\partial_\alpha^{\eps}:= 
\begin{cases}
\frac{\partial}{\partial x_j}\qquad \alpha=j\ge 1\\
\sqrt{\eps}\frac{\partial}{\partial t}\qquad \alpha = 0.
\end{cases}
\]
Recall that $u$ verifies 
\begin{equation}
\label{eq:ELboch}
-\eps \partial_t^2 u -\Delta u +\partial_t u = {A}(u)[\nabla_{g_{\eps}}u,\nabla_{g_{\eps}}u] = \eps A(u)[\partial_t u,\partial_t u] +{A}(u)[\nabla u,\nabla u].
\end{equation}
Therefore
\[
\begin{split}
\nabla \left(-\eps \partial_t^2 u -\Delta u +\partial_t u\right):\nabla u = \nabla \left(\eps A(u)[\partial_t u,\partial_t u] +{A}(u)[\nabla u,\nabla u]\right):\nabla u \\
=  -\Delta u\cdot \left(\eps A(u)[\partial_t u,\partial_t u] +{A}(u)[\nabla u,\nabla u]\right),
\end{split}
\]
and 
\[
\begin{split}
 \eps\partial_t(-\eps\partial_t^2 u -\Delta u +\partial_t u)\cdot\partial_t u = \eps\partial_t\left(\eps A(u)[\partial_t u,\partial_t u] +{A}(u)[\nabla u,\nabla u]\right)\cdot\partial_t u\\
 = -\eps\partial_t^2 u\cdot \left(\eps A(u)[\partial_t u,\partial_t u] +{A}(u)[\nabla u,\nabla u]\right),
\end{split}
\]
where we used that 
\[
\partial_\alpha u\cdot \left(\eps A(u)[\partial_t u,\partial_t u] +{A}(u)[\nabla u,\nabla u]\right) =0\qquad \forall \alpha\ge 0.
\]
As a result, 
\[
\begin{split}
&\nabla \left(-\eps \partial_t^2 u -\Delta u +\partial_t u\right):\nabla u +  \eps\partial_t(-\eps\partial_t^2 u -\Delta u +\partial_t u)\cdot\partial_t u\\
&= \left(-\Delta u -\eps\partial_t^2 u\right) \cdot \left(\eps A(u)[\partial_t u,\partial_t u] +{A}(u)[\nabla u,\nabla u]\right)\\
&= \left(-\Delta u -\eps\partial_t^2 u +\partial_t u\right) \cdot \left(\eps A(u)[\partial_t u,\partial_t u] +{A}(u)[\nabla u,\nabla u]\right)\\
& = \left(\eps A(u)[\partial_t u,\partial_t u] +{A}(u)[\nabla u,\nabla u]\right)\cdot \left(\eps A(u)[\partial_t u,\partial_t u] +{A}(u)[\nabla u,\nabla u]\right)\\
& = \sum_{\alpha,\beta=0}^{n}A(u)[\partial^\eps_{\alpha} u, \partial_\alpha^\eps u]\cdot A(u)[\partial_\beta^\eps u, \partial_\beta^\eps u].
\end{split}
\]
 Now, a classical computation gives 
 \begin{equation}
 \label{eq:deco_sec_deriv}
 \begin{split}
 \sum_{\alpha,\beta=0}^n\abs{\partial_\alpha^\eps \partial_\beta^\eps u}^2 = \abs{\bar{\nabla}\d u}^2 &= \sum_{\alpha,\beta=0}^n\abs{\bar{D}\d u[\partial_\alpha^\eps,\partial_\beta^\eps]}^2
 + \sum_{\alpha,\beta=0}^n\abs{A(u)[\partial_\alpha^\eps u,\partial_\beta^\eps u]}^2\\
 &= \abs{\bar{D}\d u}^2 +  \sum_{\alpha,\beta=0}^n\abs{A(u)[\partial_\alpha^\eps u,\partial_\beta^\eps u]}^2.
 \end{split}
 \end{equation}
We finally recall the Gauss equation: for any $X, Y, Z, W\in T_p N$ (extended to $\R^L$), there holds 
\[
{\rm{Rm}}^N(X,Y,Z,W) = A(p)[X,W]\cdot A(p)[Y, Z] - A(p)[X,Z]\cdot A(p)[Y,W]. 
\]
Then, with the choices $X = W =\partial_\alpha^\eps u$ and $Y=Z=\partial_\beta^\eps u$ we get 
\[
{\rm{Rm}}^N(\partial_\alpha^\eps u,\partial_\beta^\eps u, \partial_\beta^\eps u, \partial_\alpha^\eps u) = A(u)[\partial_\alpha^\eps u,\partial_\alpha^\eps u]\cdot A(u)[\partial_\beta^\eps u,\partial_\beta^\eps u] -\abs{A(u)[\partial_\alpha^\eps u,\partial_\beta^\eps u]}^2.
\]
Therefore,
\begin{equation}
\label{eq:boch_general}
\begin{split}
&-\eps\partial_t^2 e_\eps(u) -\Delta e_\eps(u) + \partial_t e_\eps(u) \\
&=-\abs{\bar{D}\d u}^2 + \sum_{\alpha, \beta=0}^n {\rm{Rm}}^N(\partial_\alpha^\eps u,\partial_\beta^\eps u, \partial_\beta^\eps u, \partial_\alpha^\eps u) -\textrm{Ricc}_g(\nabla u,\nabla u)\\
& = -\abs{\bar{D}\d u}^2 +
\eps \sum_{\alpha=1}^n {\rm{Rm}}^N(\partial_t u, \partial_\alpha u, \partial_\alpha u, \partial_t u) + \sum_{\alpha, \beta=1}^n {\rm{Rm}}^N(\partial_\alpha u, \partial_\beta u, \partial_\beta u, \partial_\alpha u)- \textrm{Ricc}_g(\nabla u,\nabla u),
\end{split}
\end{equation}
and the thesis follows. 
\end{proof}

\subsection{Harnack inequality for $u_\eps$} 
The Bochner formula allows us to deduce the following regularity estimates. 
\begin{theorem}[Uniform $C^1$ estimates]
\label{th:reg_negative}
Let $(M,g)$ be a compact manifold with dimension $n$ and injectivity radius $i_M$ and let $(N,h)\hookrightarrow\R^L$ be a compact Riemannian manifold with non positive sectional curvature. 
Let $u_0\in H^1(M;N)$ and $u = u_\eps\in \mathfrak{V}_{u_0}$ be a minimizer of $I_\eps$. 
Then, there exists $c=c(n)$ such that for any 
 $z_0= (x_0,t_0)\in M\times (0,+\infty)$ and any $0<r<\min\left\{i_{M},\frac{\sqrt{t_0}}2\right\}$ there holds
\begin{equation}
\label{eq:uniform_negative}
\sup_{P_{r}(z_0)}e_{\eps}(u_\eps)\le c(n)\left[\frac{\eps}{r^{n+2}}+ \frac{1}{r^{n}}\right]E(u_0),\ \forall 0<\eps\le r^2,
\end{equation}
where $P_{r}(z_0):=B_{r}(x_0)\times (t_0-r^2,t_0+r^2)$.
\end{theorem}
\begin{proof}
The proof easily follow from the Lemma \ref{lemma:weak_harnack} below combined with the energy estimate in Theorem \ref{th:energy_est}.
\end{proof}

The following Lemma provides the Harnack type inequality for the energy density $e_\eps(u_\eps)$. The proof is based on
suitable modifications of \cite[Lemma 3.5]{audritoJEQ23}. We would like to emphasize a subtle point related to the uniformity for the Harnack inequality. It is indeed impossible that the {\sl full} Harnack inequality is uniform in $\eps$. This comes from the different nature of the problems for $\eps >0$ and $\eps=0$. However, as already emphasized in \cite{audritoJEQ23}, part of the Harnack inequaltiy is indeed uniform in $\eps >0$ and provides us with the needed $L^\infty$-bound for the energy density.

\begin{lemma}
\label{lemma:weak_harnack}
Let $(M,g)$ be a $n$-dimensional compact Riemannian manifold without boundary, with injectivity radius $i_M$.  Let $f$ be a smooth, nonnegative solution of 
\begin{equation}
\label{eq:subsol}
-\eps\partial_t^2 f +\partial_t f -\Delta f \le C\, f 
\end{equation}
in $M\times (0,+\infty)$. There exists a constant $c=c(n, M)$ depending only on $n$, $M$, and $C$ such that,
for any $z_0=(x_0,t_0)\in M\times (0,+\infty)$ and $0<r_0\le\min\left\{i_M,\frac{\sqrt{t_0}}2\right\}$,
\begin{equation}
\label{eq:weak_harnack}
\sup_{P_{r_0/2}(z_0)}f \le \frac{c}{r_0^{n+2}}\iint_{P_{r_0}(z_0)}f \Vg \d t,
\end{equation}
for all $0<\eps\le r_0^2$.
\end{lemma}
\begin{proof} 
The result is local so that we assume $M$ to be $\R^n$ with the standard Euclidean metric.
Observe that if we define $\widetilde{f}(x,t)=f(x_0+r_0 x, t_0+r_0^2 t): P_1(0,0)=B_1(0)\times (-1,1)\to\R$, then $\widetilde{f}$
satisfies 
$$
-\tilde{\eps}\partial_t^2 \widetilde{f} +\partial_t \widetilde{f} -\Delta \widetilde{f} \le C\, \widetilde{f} \ \ {\rm{in}}\ \ P_1(0,0),
$$
with $\widetilde{\eps}=\frac{\eps}{r_0^2}$. Thus, without loss of generality, we may assume that $z_0=(0,0)$ and $r_0=1$.

We now follow the approach by \cite[Lemma 3.5]{audritoJEQ23} which provides the energy estimate for the Moser iteration scheme. 
Let $\psi\in C_0^\infty(P_1(0,0))$ and multiply \eqref{eq:subsol} by $f \, \psi^2$.
A standard integration by parts gives 
\begin{align}\label{energy-inq}
&\frac12 \iint_{P_1(0,0)} \partial_t (f^2) \psi (\psi +2 \eps \partial_t \psi) \d x \d t + \iint_{P_1(0,0)} |\nabla f|^2 \psi^2  \d x \d t +\eps \iint_{P_1(0,0)} |\partial_t f|^2 \psi^2  \d x \d t \nonumber\\
&\leq 
-2 \iint_{P_1(0,0)} f \psi \nabla \psi \cdot \nabla f  \d x \d t +  C\iint_{P_1(0,0)} f^2 \psi^2  \d x \d t . 
\end{align}

We now choose the test function $\psi$ 
$$
\psi(x,t)=\vf(x)\phi(t)
$$
as follows. 
Fix $\frac12<\sigma<\eta<1$, let $\vf:\R^n\to \R$ be a smooth cut off function  such that $0\le \vf(x)\le 1$, $\vf\equiv 1$ in $B_{\sigma}(0)$, $\vf\equiv 0$ in $\R^n\setminus B_{\eta}(0)$ and $\abs{\nabla \vf}\le \frac{4}{\eta-\sigma}$.
And $\phi$ is chosen similar to that in \cite[Lemma 3.5]{audritoJEQ23}. More precisely, we first 
choose $t_*\in (-\sigma^2, \sigma^2)$ so that
\begin{equation}\label{fubini1}
\int_{B_{\eta}(0)} f^2(x,t_*)\vf^2(x) \d x\ge \frac12 \sup_{t\in (-\sigma^2, \sigma^2)}\int_{B_{\eta}(0)} f^2(x,t)\vf^2(x) \d x,
\end{equation} 
then define $\phi$ by
\begin{equation*}
\phi(t)=\begin{cases} 0 & \eta^2\le |t|\le 1\\ 
\displaystyle\frac{\eta^2+t}{\eta^2-\sigma^2}  & t\in [-\eta^2, -\sigma^2]\\
1 & t\in (-\sigma^2, t_*]\\
a_*+(1-a_*) e^{\frac{t_*-t}{2\eps}} & t\in (t_*, \eta^2],
\end{cases} 
\end{equation*}
where $a_*=-\big(e^{\frac{\eta^2-t_*}{2\eps}}-1\big)^{-1}<0$. In particular, 
$\phi$ satisfies 
$$\phi +2 \eps \phi'=a_* \ {\rm{for}}\ t\in  (t_*, \eta^2);  \ \phi(t_*)=1,\ \phi(\eta^2)=0.$$
As in \cite[Lemma 3.5]{audritoJEQ23}, we can bound
\begin{align}\label{fubini2}
\frac12\iint_{P_1(0,0)} \partial_t (f^2) \psi (\psi +2 \eps \partial_t \psi) \d x \d t&\ge \frac14
\sup_{t\in (-\sigma^2, \sigma^2)}\int_{B_{\eta}(0)} f^2(x,t)\vf^2(x) \d x\\
&\ \ - \frac{c}{(\eta-\sigma)^2}\iint_{P_{\eta}(0,0)}f^2 \d x\d t.
\nonumber
\end{align}
While, by Young's inequality, we can bound
\begin{equation}\label{young}
-2 \iint_{P_1(0,0)} f \psi \nabla \psi \cdot \nabla f  \d x \d t
\le 2\iint_{P_1(0,0)} f^2 |\nabla \psi|^2  \d x \d t
+\frac12\iint_{P_1(0,0)}  |\nabla f|^2\psi^2 \d x \d t.
\end{equation} 
Substituting \eqref{fubini2} and \eqref{young} into \eqref{energy-inq}, we obtain
\begin{equation}
\label{eq:audrito1}
\begin{split}
\sup_{t\in (-\sigma^2, \sigma^2)}\int_{B_{\eta}(0)}f^2 \vf^2 \d x  
& \le \frac{c}{(\eta-\sigma)^2}\iint_{P_{\eta}(0,0)}f^2 \d x \d t +\int_{P_{\eta}(0,0)}f^2\vf^2 \d x \d t\\
&\le \frac{c}{(\eta-\sigma)^2}\iint_{P_{\eta}(0,0)}f^2 \d x \d t.
\end{split}
\end{equation}

Next we choose the test function $\psi=\vf(x)\phi(t)$ in \eqref{energy-inq}, where $\vf$ is the same as above and   
$\phi$ with compact support in $(-\eta^2, \eta^2)$, smooth, such that $0\le \phi\le 1$,
$\phi =1$ on $(-\sigma^2, \sigma^2 )$ and $|\phi '| \leq \frac{c'}{\eta-\sigma}$. This implies that for all $0<\eps\le 1$, 
\begin{equation}
\label{eq:audrito2}
\begin{split}
& \iint_{P_1(0,0)}\abs{\nabla f}^2 \vf^2\phi^2\d x\d t+ \eps\iint_{P_{1}(0,0)}\abs{\partial_t f}^2 \vf^2\phi^2\d x \d t\\
&\le 4\iint_{P_1(0,0)}f^2 [\vf^2\phi|\phi'|+|\nabla\vf|^2\phi^2+\eps\vf^2(\phi')^2]\d x \d t + C\iint_{P_1(0,0)}f^2\vf^2\phi^2 \d x \d t\\
&\le \frac{c}{(\eta-\sigma)^2}\iint_{P_\eta(0,0)}f^2 \d x \d t.
\end{split}
\end{equation}
Adding \eqref{eq:audrito1} together with \eqref{eq:audrito2} gives
\begin{equation}
\label{eq:audrito}
\begin{split}
&\sup_{t\in (-\sigma^2, \sigma^2)}\int_{B_{\eta}(0)}f^2 \vf^2 \d x +  
\int_{-\sigma^2}^{\sigma^2}
\int_{B_{\eta}(0)}\abs{\nabla f}^2 \vf^2\d x\d t\\
&\qquad\qquad\qquad\le \frac{c}{(\eta-\sigma)^2}\iint_{P_{\eta}(0,0)}f^2 \d x \d t.
\end{split}
\end{equation}
Starting from this inequality we can implement the classical Moser iteration scheme. We sketch the argument. 
The function $\vf f$ satisfies the Sobolev inequality in parabolic domains (see Lemma \ref{lemma:para_sobo} below), 
namely
\begin{align}
\label{eq:sobolev_para}
&\int_{-\sigma^2}^{\sigma^2}\int_{B_{\eta}(0)}\abs{\vf f}^{\frac{2(n+2)}{n}}\d x \d t\nonumber\\
&\qquad\le c \left(\sup_{t\in (-\sigma^2, \sigma^2)}\int_{B_{\eta}(0)}\abs{\vf f}^2\d x\d t\right)^{\frac{2}{n}}
\int_{-\sigma^2}^{\sigma^2}\int_{B_{\eta}(0)}\abs{\nabla (\vf f)}^2 \d x\d t.
\end{align}
This, combined $\abs{\nabla (\vf f)}^2 \le 2 \left(\abs{\nabla f}^2\vf^2 + \abs{\nabla \vf}^2 f^2\right)$ with \eqref{eq:audrito}, 
yields
\begin{align}
\label{eq:sobolev1}
\iint_{P_{\sigma}(0,0)}\abs{\vf f}^{\frac{2(n+2)}{n}}\d x \d t
&\le c \left(\sup_{t\in (-\sigma^2, \sigma^2)}\int_{B_{\eta}(0)}\abs{\vf f}^2\d x \d t\right)^{\frac{2}{n}}\nonumber\\
&\qquad\cdot\int_{-\sigma^2}^{\sigma^2}\int_{B_{\eta}(0)}\left(\abs{\nabla f}^2\vf^2 + \abs{\nabla \vf}^2 f^2\right)\d x \d t \nonumber\\
&\le c\left[\frac{1}{(\eta-\sigma)^2}\iint_{P_{\eta}(0,0)}f^2 \d x \d t \right]^{1+\frac{2}{n}}
\end{align}
which finally gives, setting $\lambda:= \frac{n+2}{n}$,
\begin{equation}
\label{eq:sobolev2}
\begin{split}
\iint_{P_{\sigma}(0,0)}\abs{f}^{2\lambda}\Vg \d t 
&\le\frac{c}{(\eta-\sigma)^{2\lambda}}\left(\iint_{P_{\eta}(0,0)}f^2 \Vg \d t\right)^{\lambda}.
\end{split}
\end{equation}
Now, the function $f^p$ is still a subsolution for any $p\ge 1$. Therefore, 
\[
\begin{split}
\iint_{P_{\sigma}(0,0)}\abs{f}^{2 p \lambda}\d x \d t &\le \frac{c(n)}{(\eta-\sigma)^{2\lambda}}
\left(\iint_{P_{\eta}(0,0)}f^{2p} \d x \d t\right)^{\lambda}.
\end{split}
\]
We are now in the position to start the Moser iteration scheme which gives that
\[
\sup_{P_{\sigma}(0,0) }f^2 \le \frac{c}{(\eta-\sigma)^{n+2}}\iint_{P_\eta(0,0)}f^2 \d x \d t. 
\]
Now fix $\sigma\in (1/2,1)$ and $\eta = \sigma + \frac{1-\sigma}{4}$ and apply the above inequality to obtain 
\begin{equation}
\label{eq:start_interpol}
\sup_{P_{\sigma}(0,0)}f^2 \le \frac{c}{(\eta-\sigma)^{n+2}}\iint_{P_{\eta}(0,0)}f^2 \d x \d t.
\end{equation}
This is the starting point of an interpolation argument that eventually will lead to \eqref{eq:weak_harnack}. 
First of all, note that H\"older inequality and \eqref{eq:start_interpol}
give 
\begin{equation}
\label{eq:interpol_1}
\begin{split}
\norm{f}_{L^{\infty}(P_{\sigma}(0,0))} &\le c \frac{1}{(\eta-\sigma)^{\frac{n+2}{2}}}\norm{f}_{L^{1}(P_{\eta}(0,0))}^{1/2}\norm{f}_{L^{\infty}(P_{\eta}(0,0))}^{1/2}\\
&\le c\frac{1}{(\eta-\sigma)^{\frac{n+2}{2}}}\norm{f}_{L^{1}(P_1(0,0))}^{1/2}\norm{f}_{L^{\infty}(P_{\eta}(0,0))}^{1/2}
\end{split}
\end{equation}
Fix $\delta \in (1/2,1)$ and consider the sequence 
\[
\begin{cases}
\sigma_0  &= \delta \\
\sigma_{k+1} &= \sigma_k + \frac{1-\sigma_k}{4}.
\end{cases}
\]
Inequality \eqref{eq:interpol_1} with $\sigma =\sigma_k$ and $\eta = \sigma_{k+1}$ thus becomes
\[
\norm{f}_{L^{\infty}(P_{\sigma_k}(0,0))}\le c \left(\frac{4}{3}\right)^{k\frac{n+2}{2}}\frac{1}{(1-\delta)^{\frac{n+2}{2}}}
\norm{f}_{L^{1}(P_1(0,0))}^{1/2}\norm{f}_{L^{\infty}(P_{\sigma_{k+1}}(0,0))}^{1/2}
\]
Therefore, 
\[
\norm{f}_{L^{\infty}(P_{\sigma_0}(0,0))}\le c\left(\frac{4}{3}\right)^{\frac{n+2}{2}\sum_{k=0}^{m-1}k\left(\frac{1}{2}\right)^{k}}\left[\frac{1}{(1-\delta)^{\frac{n+2}{2}}}\norm{f}_{L^1(P_{1}(0,0))}^{1/2}\right]^{\sum_{k=0}^{m-1}\left(\frac{1}{2}\right)^k}\norm{f}_{L^{\infty}(P_{\sigma_m}(0,0))}^{\frac{1}{2^n}}
\]
Thus, if let $m\to +\infty$, we get 
\[
\sup_{P_{\delta}(0,0)}f \le c(n)\left(\frac{4}{3}\right)^{2(n+2)}\frac{1}{(1-\delta)^{n+2}}\iint_{P_{1}(0,0)}f\d x \d t,
\]
which readily implies \eqref{eq:weak_harnack} by scaling back to the original variables. 
\end{proof}

\begin{lemma}[Sobolev embedding in parabolic cylinders]
\label{lemma:para_sobo}
Let $B_r(x_0)$ be a geodesic ball and 
$U:P_r(z_0)\to \R$ be a $C^1$ function such that $U\equiv 0$ on $\partial B_{r}(x_0)\times (t_0-r^2,t_0+r^2)$. 
Then
\[
\iint_{P_{r}(z_0)}\abs{U}^{\frac{2n}{n+2}}\d x \d t \le c(n) \left(\sup_{t\in (t_0-r^2,t_0+r^2)}\int_{B_r(x_0)}\abs{U}^2\Vg \right)^{\frac{2}{n}}\iint_{P_{r}(z_0)}\abs{\nabla U}^2 \d x \d t.
\] 
\end{lemma}
\begin{proof}
For any fixed $t\in (t_0-r^2,t_0+r^2)$ the Sobolev embedding \cite[Theorem 2.21]{Aubin} gives that 
\[
\left(\int_{B_r(x_0)}\abs{U(x,t)}^{\frac{2n}{n-2}}\d x \right)^{\frac{n-2}{n}}\le c(n) \int_{B_r(x_0)}\abs{\nabla u(x,t)}^2 \d x.
\]
Therefore
\[
\norm{U}_{L^2(t_0-r^2,t_0+r^2;L^{\frac{2n}{n-2}}(B_r(x_0)))}\le c(n) \norm{\nabla U}_{L^2(t_0-r^2,t_0+r^2;L^2(B_r(x_0)))}.
\]
Moreover, by Bochner spaces interpolation we have that 
\[
L^{\frac{2(n+2)}{n}}(P_r(z_0))\subset L^{\infty}(t_0-r^2,t_0+r^2;L^2(B_r(x_0)))\cap L^2(t_0-r^2,t_0+r^2;L^{\frac{2n}{n-2}}(B_r(x_0))),
\]
and thus the thesis follows by the properties of the interpolation functor. 
\end{proof}

\subsection{Proof of Theorem 1.3}

In this section, we prove Theorem 1.3 through the convergence by PDEs. 
{Recall that the Euler Lagrange equation for $u_\eps$ is given by 
\begin{equation}\label{wed-hmf11}
\begin{cases}
\partial_t u_\eps-\Delta u_\eps=\eps\big(\partial_t^2 u_\eps+A(u_\eps)(\partial_t u_\eps, \partial_t u_\eps)\big)+A(u_\eps)(\nabla u_\eps,
\nabla u_\eps) &\qquad  {\rm{in}}\ \ M\times (0,\infty),\\
u_\eps(x,0) = {u}_0 &\qquad \textrm{ in } M\times \left\{0\right\}.
\end{cases}
\end{equation}

Next we establish the strong convergence of $u_\eps$ in $H^1_{\rm{loc}}$. From Theorem 
\ref{th:energy_est} and Theorem \ref{th:reg_negative}, we first have that for any $\delta>0$, there exists a constant $c>0$ 
depending on $n, E(u_0)$ and $M$ such that
$$
e_\eps(u_\eps)(x,t)=\frac12 (\eps|\partial_t u_\eps|^2+|\nabla u_\eps|^2)(x,t)\le c, \ \forall  x\in M,  t\ge \delta.
$$
This, combined with $\displaystyle\int_0^\infty \int_M |\partial_t u_\eps|^2\le E(u_0)$, implies that
for any $x_0\in M$ and $t_0\ge \delta$, $0<r<r_0=\min\{i_M, \delta^{\frac{1}{n+2}}\}$,
\begin{equation}\label{oscillation-est}
\abs{u_\eps(x,t)-u_\eps(y,s)}\le cr, \ \ \forall x, y\in B_r(x_0), \ \forall s,t:  t_0-r^{n+2}<s<t< t_0+r^{n+2}.
\end{equation}
To see \eqref{oscillation-est}, we estimate
\begin{align*}
\abs{u_\eps(x,t)-u_\eps(y,s)}
&\le \abs{u_\eps(x,t)-\fint_{B_r(x_0)}u_\eps(w,t)\, \d w}+ \abs{u_\eps(y,s)-\fint_{B_r(x_0)}u_\eps(w,s)\,\d w}\\
&\quad+\abs{\fint_{B_r(x_0)}u_\eps(w,t)\,\d w-\fint_{B_r(x_0)}u_\eps(w,s)\,\d w}\\
&\le \big(\|\nabla u_\eps(\cdot, t)\|_{L^\infty(M)}+\|\nabla u_\eps(\cdot, s)\|_{L^\infty(M)}\big)r\\
&\quad+\Big(\int_s^t\fint_{B_r(x_0)}\abs{\partial_tu_\eps}^2(w,\tau)\,\d w\d \tau\Big)^\frac12 \abs{t-s}^\frac12\\
&\le cr+c\frac{\sqrt{\abs{t-s}}}{r^{\frac{n}2}}\le cr.
\end{align*}

It follows from \eqref{oscillation-est} and Theorem 3.3 that we may assume that 
there exists a map $u: M\times (0, \infty)\to N$ such that for any $0<\delta<T<\infty$ and any ball $B\subset M$,
after passing to a subsequence, 
$$
u_\eps\rightharpoonup u \ \ \textrm{in}\ \ H^1(B\times (\delta, T)),  \ u_\eps\rightarrow u \ \ \textrm{in}\ \  C^0(B\times (\delta, T)).
$$

Now we want to show that
\begin{equation}\label{h1-conv}
\nabla u_\eps\to\nabla u \ \ \textrm{in} \ \ L^2(B\times (\delta, T)).
\end{equation}
It suffices to show $\nabla u_\eps$ is a Cauchy sequence in  $L^2(B\times (\delta, T))$. For this, 
let $\psi\in C_0^\infty(B\times (\delta, T))$ and let $\eps, \eps'\to 0$, by testing the equations for $u_\eps$ and $u_{\eps'}$
by $(u_\eps-u_{\eps'})\psi^2$ we obtain
\begin{align*}
&\iint_{M\times \R_+} \abs{\nabla (u_\eps-u_{\eps'})}^2\psi^2\Vg \d t\\
&=-\iint_{M\times \R_+} \partial_t (u_\eps-u_{\eps'}) (u_\eps-u_{\eps'})\psi^2\Vg \d t-\iint_Q (u_\eps-u_{\eps'})\cdot\nabla (u_\eps-u_{\eps'})\nabla\psi^2 \Vg \d t\\
&\quad-\iint_{M\times \R_+}(\eps\partial_t u_\eps-\eps'\partial_t u_{\eps'})\cdot\big(\partial_t(u_\eps-u_{\eps'})\psi^2+(u_\eps-u_{\eps'})\partial_t\psi^2\big)\Vg \d t\\
&\quad+\iint_{M\times \R_+} \big(\eps A(u_\eps)(\partial_t u_\eps, \partial_t u_\eps)-\eps' A(u_{\eps'})(\partial_t u_{\eps'}, \partial_t u_{\eps'})\big)(u_\eps-u_{\eps'})\psi^2\Vg \d t\\&\quad+\iint_{M\times\R_+}\big(A(u_\eps)(\nabla u_\eps,
\nabla u_\eps)-A(u_{\eps'})(\nabla u_{\eps'},
\nabla u_{\eps'})\big)(u_\eps-u_{\eps'})\psi^2\Vg \d t\\
&={A_1 + A_2 + A_3 + A_4 + A_5}
\end{align*}
It is not hard to check that as $\eps, \eps'\to 0$,
$$
|A_1|\leq c (\|\partial_t u_\eps\|_{L^1(B\times (\delta, T))}+\|\partial_t u_{\eps'}\|_{L^1(B\times (\delta, T)})\|u_\eps-u_{\eps'}\|_{C^0(B\times (\delta, T))}
\to 0,
$$
$$
|A_2|\leq c (\|\nabla u_\eps\|_{L^1(B\times (\delta, T))}+\|\nabla u_{\eps'}\|_{L^1(B\times (\delta, T)})\|u_\eps-u_{\eps'}\|_{C^0(B\times (\delta, T))}
\to 0,
$$
\begin{align*}
|A_3|&\le c(\eps\|\partial_t u_\eps\|_{L^2(B\times (\delta, T))}+\eps'\|\partial_t u_{\eps'}\|_{L^2(B\times (\delta, T)})(\|\partial_t u_\eps\|_{L^2(B\times (\delta, T))}+\|\partial_t u_{\eps'}\|_{L^2(B\times (\delta, T)})\\
&+c(\eps\|\partial_t u_\eps\|_{L^1(B\times (\delta, T))}+\eps'\|\partial_t u_{\eps'}\|_{L^1(B\times (\delta, T)})\|u_\eps-u_{\eps'}\|_{C^0(B\times (\delta, T))}\\
&\le c(\eps+\eps')\to 0,
\end{align*}
\begin{align*}
|A_4|&\le c\Big(\eps\|\nabla u_\eps\|_{L^2(B\times (\delta, T))}^2+\eps'\|\nabla u_{\eps'}\|_{L^2(B\times (\delta, T))}^2\Big))\|u_\eps-u_{\eps'}\|_{C^0(B\times (\delta, T))}\\
&\le c(\eps+\eps')\to 0
\end{align*}
and
\begin{align*}
|A_5|&\le c\Big(\|\nabla u_\eps\|_{L^2(B\times (\delta, T))}^2+\|\nabla u_{\eps'}\|_{L^2(B\times (\delta, T))}^2\Big))\|u_\eps-u_{\eps'}\|_{C^0(B\times (\delta, T))}\\
&\le c\|u_\eps-u_{\eps'}\|_{C^0(B\times (\delta, T))}\to 0.
\end{align*} 
Thus \eqref{h1-conv} holds.

It follows from \eqref{h1-conv} that 
$A(u_\eps)(\nabla u_\eps, \nabla u_\eps)\rightarrow A(u)(\nabla u, \nabla u)$ in $L^1(B\times (\delta, T))$.
On other hand, it is easy to see that
$\eps\big(\partial_t^2 u_\eps+A(u_\eps)(\partial_t u_\eps, \partial_t u_\eps)\big)\rightharpoonup 0$ in the sense of distributions. 
Hence, by passing the limit $\eps\to 0$ in \eqref{wed-hmf11}, we obtain that
$$
\partial_t u-\Delta u=A(u)(\nabla u,\nabla u) 
$$
in $M\times (0,\infty)$, since $B\subset M$, $\delta, T>0$ are arbitrary.  It is also not hard to verify that
$u=u_0$ at $t=0$. The uniform Lipschitz regularity of $u(\cdot, t)$, for $t\ge\delta>0$, follows from Lemma 4.4. both higher order
regularity and the uniqueness
of $u$ follows from the well-known theory on heat flow of harmonic maps to non positively curved manifolds. This completes the proof
of Theorem 1.3. \qed
}

\bigskip

Finally we can prove Theorem \ref{th:ES-optimal_control_intro}, namely that given ${u}_0\in H^1(M,N)\cap C^\infty(M, N)$ there exists a (unique) harmonic map in the homotopy class of ${u}_0$. 
The classical proof of Eells-Sampson uses the harmonic map heat flow to deform ${u}_0$ into a harmonic map. 
Here we use the EVI flow of the value function $V_\eps$ and the fact that the minimizers of $V_\eps$ are indeed minimizers of $E$ (hence minimizing harmonic maps). 
We have the following

\begin{theorem}
\label{th:ES-optimal_control}
Let $N$ be a compact Riemannian manifold with nonpositive sectional curvature and let ${u}_0\in H^1(M, N)\cap C^\infty(M,N)$.
 Then, given $u_\eps \in \argmin\left\{I_\eps[v],\ v\in \mathfrak{V}_{{u}_0}\right\}$ there holds
 \[
 \forall \eps>0\quad u_\eps(t)\xrightarrow{t\to +\infty}u_\infty \qquad \textrm{ in } L^2_{\textrm{loc}}(M, N), 
 \]
 with $u_{\infty}$ the unique minimizing harmonic map in the homotopy class of ${u}_0$.
\end{theorem}
\begin{proof}
We work at a fixed $\eps>0$.
We take a sequence $t_k\nearrow +\infty$ and we set, for $t\in (0,1)$, 
\[
u_k^{(\eps)}(t) = u_\eps(t+t_k).
\]
The sequence $u_k^{(\eps)}$ verifies, uniformly in $\eps>0$, (see Theorem \ref{th:energy_est})
\[
\begin{split}
\sup_{t\in (0,1)}E(u_{k}^{(\eps)}(t)) = \sup_{t\in (0,1)}E(u_\eps(t+t_k))\le E({u}_0)\\
\int_{0}^1\norm{\partial_t u_k^{(\eps)}(t)}^2\d t = \int_{t_k}^{1+t_k}\norm{\partial_t u_\eps(t)}^2\d t\le \int_{t_k}^{+\infty}\norm{\partial_t u_\eps(t)}^2\d t.
\end{split}
\]
Therefore, there exists a subsequence of $t_k$, denoted as itself,  and $u_\infty^{(\eps)}\in L^2(0,1; H^1(M, N))$ such that 
for any ball $B$ in $M$ 
\[
u_{{k}}^{(\eps)}\xrightarrow{k\nearrow +\infty}u_{\infty} ^{(\eps)}\qquad \textrm{ in } C^0([0,1];L^2(B, N)).
\]
Note that, by construction, $\partial_t u_\infty^{(\eps)}\equiv 0$ and therefore we identify $u_\infty^{(\eps)}$ with an element in $H^1(M, N)$ 
with the same name. 

Now, since $u_\eps$ is an EVI flow for the 
convex energy 
$V_\eps$ it verifies (see \cite[Theorem 4.04]{ags2005})
\[
V_\eps(u_\eps(t))\le V_\eps(v) + \frac{d^2(\bar{u},v)}{2t} \qquad \forall t\in (0,+\infty), \,\,\forall v\in H^1(M, N). 
\]
Thus,  
\[
\limsup_{t\to +\infty}V_\eps(u_\eps(t))\le V_\eps(v) \qquad \forall v\in H^1(M, N), 
\]
that is
\[
V_\eps(u_\infty^{(\eps)})\le \liminf_{k\to +\infty}V_\eps(u_{\eps}(t+t_{k}))\le \limsup_{t\to +\infty}V_\eps(u_\eps(t))\le V_\eps(v) \qquad \forall v\in H^1(M, N),
\]
namely $u_\infty^{(\eps)}$ is a minimizer for $V_\eps$ and hence for $E$, thanks to Corollary \ref{cor:minV=minE}.

Now we show that $u_\infty^{(\eps)}$ is actually independent of $\eps$. 
First of all note that Lemma \ref{lemma:weak_harnack} gives that for $x\in M$ and $t\in (0,1)$
\[
\begin{split}
&\eps\abs{\partial_t u_k^{(\eps)}(x,t)}^2 + \abs{\nabla u_k^{(\eps)}(x,t)}^2 = e_\eps(u_{k}^{(\eps)}(x,t))\\
& = e_\eps(u_\eps(x,t+t_k))\le \frac{c(n)}{r^{n+2}}\iint_{P_{r}(x,t+t_k)}e_\eps(u_\eps(y,s))\Vg\d s\le c(n) \left(\frac{\eps}{r^{n+2}}+\frac{1}{r^n}\right)E(\bar{u}),
\end{split}
\]
and therefore Ascoli-Arzelà gives that $u_{k}^{(\eps)}$ actually converges to $u_{\infty}^{(\eps)}$ locally uniformly in $M\times (0,\infty)$. 
Consequently, there exists some $\bar k$ such that for $k\ge \bar{k}$ and for any $(x,t)$ in $M\times (0,1)$ we have that 
\[
d_N(u_{\eps}(x,t+t_{k}),u_\infty^{(\eps)}(x)) \le i_N,
\]
where $i_N$ is the injectivity radius of $N$. As a consequence for $k\ge \bar{k}$, $u_\eps(\cdot,t+t_{k})$ is homotopic to $u_\infty^{(\eps)}(\cdot)$ via the unique geodesic in $N$ connecting them. 
Since $u_\eps(\cdot,t+t_{k})$ is homotopic to ${u}_0$ via the EVI flow we conclude that $u_{\infty}^{(\eps)}$ is homotopic to ${u}_0$. 
Now, as $u^{(\eps)}_{\infty}$ is a harmonic map in the homotopy class of ${u}_0$, it is unique thanks to 
{convexity}. Therefore it is independent of $\eps$. 
\end{proof}

\begin{appendix}

\section{Basics on Sobolev spaces, energies with non smooth targets and absolutely continuous curves}
\label{appendixSpaces}
In this appendix, we expand on the definition of the energy and its properties in the case of metric targets. Since the seminal works of Korevaar and Schoen \cite{KS}, there has been an increasing interest in the development of suitable energies when the domain and/or the target are non smooth spaces; or more generally Sobolev type spaces in such context. 

Let $(X,d)$ be a CAT(0)-space. Let $\Omega$ be an open set of the smooth manifold $M$ (or $M \times (0,\infty)$). 
Following \cite{KS}, we have the following definitions. 
\begin{definition}\label{L2space}
A measurable function $f:\Omega \to X$ is an $L^2$ map (or $f \in L^2(\Omega,X)$) if for some $P \in X$ (hence all $P \in X$), one has 
$$
\int_\Omega d^2(f(x), P)\,dx <\infty. 
$$
This induces naturally a metric on the space $L^2(\Omega,X)$. 
\end{definition}

We now introduce the space of Sobolev maps from $\Omega$ to $X$ in the spirit of \cite{KS} (see \cite{GS} as well). 
 \begin{definition}[Korevaar-Schoen energy]\label{defksenery}
Let $f:\Omega \subset M \to X$ be a Borel map and let $r \in (0, \infty)$.
\begin{enumerate} 
\item Define the \textit{energy density at scale $r$ of $f$ at $x \in \Omega$}, denoted by $\mathrm{e}^f_{r}(x)$, by
\begin{equation}
\mathrm{e}^f_{r}(x):=\left(\frac{n+2}{|B_r(x)|)}\int_{B_r(x) \cap A}\frac{d(f(x), f(y))^2}{r^2}\,\, dy\right)^{1/2}.
\end{equation}
\item Define the \textit{Korevaar-Schoen energy at scale $r$}, denoted by $\mathcal{E}_{\Omega, r}^f$, by
\begin{equation}
\mathcal{E}_{\Omega, r}^f:=\int_\Omega\left(\mathrm{e}^f_{r}(x)\right)^2\,\,dx.
\end{equation}
\item Define the \textit{Korevaar-Schoen energy}, denoted by $E(f)$, by
\begin{equation}
E(f):= \frac{1}{2}\limsup_{r \to 0^+}\mathcal{E}_{\Omega, r}^f.
\end{equation}
\end{enumerate}
\end{definition}

Several remarks are in order. In the previous definition, we consider a Riemannian domain in the manifold $M$ (or $M \times (0,\infty)$) since we are mainly interested in harmonic mappings from a smooth manifold into an NPC space. It is clear that such definition extends more generally to domains which are metric measured spaces. The previous definition allows to obtain 
\begin{definition}
A measurable map $f: \Omega \to X$ is a Sobolev map $f \in H^1(\Omega,X)$ if $f \in L^2(\Omega,X)$ and additionally  $f$ has finite energy in the sense of definition \ref{defksenery}. 

\end{definition}

\begin{remark}
In \cite{KS}, Korevaar and Schoen defined the Sobolev maps of $W^{1,p}(\Omega,X)$ with $p>1$. It should be clear from the previous definitions how to proceed to define such scales of spaces. 
\end{remark}

We would like to emphasize that Definition \ref{defksenery} is actually a convergence in measure result of the approximate energy $\mathcal{E}_{\Omega, r}^f$ on the space of Lipschitz functions with compact support $f:\Omega \to X$ as $r \to 0^+$. More precisely, the so-called {\sl energy density measure} $\mathrm{e}^f_{r}\,\Vg$ converges in the sense of measures (here we denote $\Vg$ the Riemannian measure on $M$ ) to the measure $|\nabla f(x)|^2\,\Vg$, and by definition we have 
$$
E(f):= \frac{1}{2}\int_{\Omega} |\nabla f(x)|^2\,\Vg. 
$$  
Therefore we define the Dirichlet energy 
\begin{equation}
\label{eq:diri_appe}
E(v):=
\begin{cases}
 \displaystyle\frac{1}{2}\int_\Omega \abs{\nabla v}^2 \Vg, &\qquad v\in H^1(\Omega;X)\\
  + \infty &\qquad \textrm{ otherwise in } L^2(\Omega, X)
  \end{cases}
\end{equation}
with domain $D(E) = H^1(\Omega,X)$. In this paper, we will then always use the notation above for the energy whenever the target is smooth or not. 

Another important notion to be able to make inner variations of the energy is to define {\sl directional derivatives}. Denote by $\Gamma (T \overline{\Omega})$ the set of smooth vector fields on $\Omega$. Denote by $\abs {f_*(V)}^2$ the energy density of the directional derivative in $H^1(\Omega,X)$ of the map $f$ in the direction $V \in \Gamma (T \overline{\Omega})$, as defined in \cite[Section 1.8]{KS}.
Then one can define a bilinear symmetric continuous nonnegative tensorial operator by 
\[
\langle f_*(V), f_*(W)\rangle =\frac12 |f_*(V+W)|^2-\frac12 |f_*(V-W)|^2, \ \forall V, W\in \Gamma (T \overline{\Omega})
\]
This generalizes the notion of pullback metric by maps into manifolds. It is then natural to define the following quantities: 
\begin{itemize}

\item Denote $|\partial_t f|^2$ and $|\partial_{x_i} f|^2$ for $|f_*(\frac{\partial}{\partial t})|^2$ and $|f_*(\frac{\partial}{\partial x_i})|^2$ respectively, for $i=1,\cdots, n$,
where $\{\frac{\partial}{\partial x_i}\}_{i=1}^n$ denotes the local coordinate frame of $(M,g)$,
\item Denote $\partial_t f\cdot \partial_{x_i} f=\langle f_*(\frac{\partial}{\partial t}), f_*(\frac{\partial}{\partial x_i})\rangle$ for $i=1,\cdots, n$
\item Denote $\partial_{x_i} f\cdot \partial_{x_j} f=\langle f_*(\frac{\partial}{\partial x_i}), f_*(\frac{\partial}{\partial x_i})\rangle$ for $1\le i,j\le n$
and finally $|\nabla f|^2=\sum_{i, j} g^{ij} \partial_{x_i} f\cdot \partial_{x_j} f$. 

\end{itemize}

{
When dealing with gradient flow equations in metric spaces the notion of absolutely continuous curve and of metric derivative is particularly important. 
We have the following (see e.g. \cite{ags2005}..)
\begin{definition}[Absolutely continuous curves]
\label{def:AC_curve}
We let $Y$ be a complete metric space. 
A curve $\gamma:[0,T]\to X$ is said to be absolutely continuous if there exists $f\in L^1(0,T)$ such that 
\begin{equation}
\label{eq:AC_curve}
d(\gamma(t),\gamma(s))\le \int_{s}^t f(r)\d r\qquad \forall 0<s<t<T.
\end{equation}
The space $AC^2(0,T;Y)$ is the space of absolutely continuous curves for which there exists $f\in L^2(0,T)$ that verifies
\eqref{eq:AC_curve}.
\end{definition}

\begin{theorem}
\label{th:metric_derivative}
For any curve $\gamma\in AC^2(0,T;Y)$ the limit 
\[
\abs{\gamma'}(t):= \lim_{h\to 0}\frac{d(\gamma(t+h),\gamma(t))}{h}
\]
exists for almost any $t\in (0,T)$ and the function $t\mapsto \abs{\gamma'}(t)$ is the least function $f$ for which 
\eqref{eq:AC_curve} holds, namely
\[
\abs{\gamma'}(t)\le f(t)\qquad \textrm{a.e.  } \forall f\,\,\,\textrm{ satisfying} \,\,\,\,\eqref{eq:AC_curve}.
\]  
\end{theorem}

Interestingly, there is a strict relation between $AC^2(0,T;Y)$ and $H^1(0,T;Y)$ in the sense of Korevaar \& Schoen, see Theorem 2.11 in \cite{GigliTyulenev}. 
In particular there holds that any curve in $AC^2(0,T;Y)$ belongs to $H^1(0,T;Y)$ and given $\gamma\in H^1(0,T;Y)$ its continuous representative belongs to $AC^2(0,T;Y)$ and 
 for almost any $t\in (0,T)$ 
\[
\abs{\partial_t \gamma(t)} = \abs{\gamma'}(t).
\]
We finally remark that in the running separability hypothesis on $X$ we will always tacitly identify measurable
``abstract curves'' ${\bf f}:(0,T)\to L^2(M, X)$ for $T\in [0,+\infty]$ with their ``concrete'' (measurable) representation 
$f:M\times (0,T)\to X$ given by ${\bf f}(t)(x) = f(x,t)$ for almost any $(x,t)\in M\times (0,T)$. We use the same notation accordingly. 
In particular, if ${\bf f}\in AC^2(0,T;L^2(M,X))$ we will have that 
\[
\abs{{\bf f}'}^2(t) :=\lim_{h\to 0}\frac{d({\bf{f}}(t+h),{\bf{f}}(t))}{h}=  \int_M\abs{f'}^2(t)\Vg,
\]
where $\abs{f'}:= \lim_{h\to 0}\frac{d_Xx(f(x,t+h),f(x,t))}{h}\in L^2(M\times (0,T))$
Moreover, the discussion above gives 
\[
\norm{\partial_t f(t)}^2=\int_M\abs{\partial_t f}^2\Vg =\int_M\abs{f'}^2(t)\Vg = \abs{f'}^2(t), \,\,\,\,\textrm{ a.e. in } (0,T).
\]
%
%
}

\section{Basics on gradient flows in metric spaces}
\label{appendixGF}
\end{appendix}
Given a metric space $(Y,d_Y)$ and an energy $E:Y\to (-\infty,+\infty]$ we recall the definition of slope (see \cite[Chapter 1]{ags2005})
\begin{definition}
\label{def:slope}
The slope of $E$ at $v\in D(E)$ is defined by 
\[
\abs{\partial E}(v):= \limsup_{w\to v}\frac{(E(v)-E(w))^+}{d_Y(v,w)},
\]
where $w\to v$ means that $d_Y(w,v)\to 0$. 
\end{definition}
We also recall that if $E$ is geodesically convex, then the slope is lower semicontinuous. 
Now we concentrate on CAT($0$) metric spaces and on geodesically convex energies. 
The following Theorem holds (see \cite{ags2005}, \cite{MS2020} and \cite{gigli_nobili21}). 
\begin{theorem}[Gradient flows in a CAT($0$) space]
Let $(Y,d_Y)$ be a CAT($0$) metric space and $E: Y\to (-\infty,+\infty]$ be a lower semicontinuous and geodesically convex energy. 
Let ${u}_0 \in D(E)$ and $u\in AC([0,+\infty);Y)$ such that $u(0) = u_0$. 
Then the following are equivalent
\begin{enumerate}
\item 
\begin{equation}
\label{eq:MSI_appendix}
-\frac{\d}{\d t}E(u(t))\ge \frac{1}{2}\abs{u'}^2(t) + \frac{1}{2}\abs{\partial E}^2(u(t))\qquad \textrm{ for a.a. }t>0;
\end{equation}
\item 
\begin{equation}
\label{eq:CR_appendix}
-\frac{\d}{\d t}E(u(t)) = \abs{u'}^2(t) = \abs{\partial E}^2(u(t))\qquad \textrm{ for a.a. }t>0;
\end{equation}
\item 
\begin{equation}
\label{eq:EVI_appendix}
\frac{1}{2}\frac{\d}{\d t}d^2_Y(u(t), v)+ E(u(t)) \le E(v)\qquad \textrm{ in } \mathscr{D}'(0,+\infty), \,\,\,\forall v\in D(E). 
\end{equation}
\end{enumerate} 
Moreover, the above relations are equivalent to the following: the right derivative $u'_{+}$ exists for any $t>0$ and  
\begin{equation}
\label{eq:subdiff_appendix}
\begin{cases}
u'_{+}(t)\in -\partial^{-}E(u(t)) \qquad \forall t>0\\
u(0) = u_0,
\end{cases}
\end{equation}
and $u'_{+}(t)$ is the element of minimal norm in $-\partial^{-}E(u(t))$ for $t>0$. 
\end{theorem}
We refer to \cite{gigli_nobili21} for the definition of right derivative $u'_{+}(t)$ and of the subdifferential $-\partial^{-}E$. 
The Theorem above applies in particular when $E$ is the Korevaar-Schoen energy for maps defined in a Riemannian manifold $M$ and with values in a CAT($0$) space $X$. In this case $(Y, d_Y)= (L^2(M,X), d_2)$ and the Theorem says that once one shows the existence (e.g. via minimizing movements or, as in the present paper via WED functional) of a gradient flow solution, i.e. a curve satisfying one the three equivalent conditions \eqref{eq:MSI_appendix}-\eqref{eq:EVI_appendix}, then it satisfies the differential inclusion \eqref{eq:subdiff_appendix}.
In particular (see again \cite{gigli_nobili21}) the tension field (a.k.a the Laplacian) $\tau(v)$ of a Sobolev map $v$ with values
in a CAT($0$) metric space $(X, d)$ is characterized
 as the element of minimal norm in $-\partial^{-}E(v)$. As a result, when $X$ is a smooth 
Riemannian manifold $N$ isometrically embedded in $\R^L$ with non positive sectional curvature, 
the gradient flow solution is actually
the smooth solution of the heat flow of harmonic maps into $N$ and satisfies \eqref{eq:heat_flow_intro1}.

\section*{Acknowledgements}
F. Lin is partially supported by NSF DMS 2247773.
A. Segatti is partially supported by PRIN 2022 (Project no. 2022J4FYNJ), funded by MUR, Italy, and the European Union -- Next Generation EU, Mission 4 Component 1 CUP F53D23002760006 and by GNAMPA-INdAM.
Y. Sire  is partially supported by NSF DMS 2154219. 
C. Wang is partially supported by NSF DMS 2453789 and Simons Travel Grant TSM-00007723. 

\bigskip{\bf Added in Proof}: After this paper has been completed,  X. P. Zhu very recently informed the first author
of their preprint \cite{ZhangZhu2026}, where they obtained the Lipschitz regularity in both $x$ and $t$ by a different method {similar to that from their
previous work \cite{zhang-zhu}.
Shortly after, Lin and Wang found an alternate new proof of the same regularity result in \cite{LinWang2026} by extending
the Korevaar-Schoen techniques to the heat flows into CAT(0)-spaces}.

\bibliographystyle{plain}
\bibliography{wed_hmhf,wed_hmhf2}
\end{document}